\documentclass{amsart}
\usepackage{amsmath}
\usepackage{amsmath}
\usepackage{amsfonts}
\usepackage{amsthm}
\usepackage{amssymb}
\usepackage{esint}
\usepackage{eucal}
   \usepackage{times}
\usepackage{accents}
\usepackage{tikz-cd}

\newtheorem{theorem}{Theorem}
\newtheorem*{theorem*}{Theorem}
\numberwithin{equation}{section}
\numberwithin{theorem}{section}

\newtheorem*{acknowledgement*}{Acknowledgement}
\newtheorem*{definition*}{Definition}

\newtheorem{corollary}[theorem]{Corollary}

\newtheorem{definition}[theorem]{Definition}

\newtheorem{lemma}[theorem]{Lemma}
\newtheorem{notation}[theorem]{Notation}

\newtheorem{proposition}[theorem]{Proposition}
\newtheorem{remark}[theorem]{Remark}

\newtheorem*{question*}{Question}

\newcommand{\RR}[0]{\mathbb{R}}

\newcommand{\Ss}[0]{\mathbb{S}}
\newcommand{\CP}[0]{\mathbb{CP}}
\newcommand{\pd}[2]{\frac{\partial #1}{\partial#2}}

\newcommand{\pdt}[0]{\frac{\partial}{\partial t}}

\newcommand{\delb}[0]{\overline{\nabla}}

\newcommand{\ve}[1]{\mathbf{#1}}

\newcommand{\Rc}[0]{\operatorname{Rc}}
\newcommand{\Rm}[0]{\operatorname{Rm}}

\newcommand{\dfn}[0]{\doteqdot}
\newcommand{\Xb}[0]{\mathbf{X}}
\newcommand{\Yb}[0]{\mathbf{Y}}
\newcommand{\Xc}[0]{\mathcal{X}}
\newcommand{\Yc}[0]{\mathcal{Y}}

\newcommand{\gh}[0]{\hat{g}}
\newcommand{\pdtau}[0]{\pd{}{\tau}}

\newcommand{\lam}{\lambda}

\newcommand{\Ec}[0]{\mathcal{E}}

\newcommand{\Hc}[0]{\mathcal{H}}

\newcommand{\Mm}[0]{\mathcal{M}}

\newcommand{\Pc}[0]{\mathcal{P}}
\newcommand{\Vc}[0]{\mathcal{V}}

\newcommand{\Cc}[0]{\mathcal{C}}

\newcommand{\ch}[1]{\breve{#1}}
\newcommand{\ub}[1]{\underaccent{\bar}{#1}}
\newcommand{\ua}[0]{\ub{a}}
\newcommand{\ui}[0]{\ub{\imath}}
\newcommand{\uj}[0]{\ub{\jmath}}
\newcommand{\uk}[0]{\ub{k}}
\newcommand{\ul}[0]{\ub{l}}
\newcommand{\um}[0]{\ub{m}}

\newcommand{\up}[0]{\ub{p}}
\newcommand{\uq}[0]{\ub{q}}
\newcommand{\ur}[0]{\ub{r}}

\newcommand{\bi}[0]{\bar{\imath}}
\newcommand{\bj}[0]{\bar{\jmath}}
\newcommand{\bk}[0]{\bar{k}}
\newcommand{\bl}[0]{\bar{l}}
\newcommand{\bm}[0]{\bar{m}}
\newcommand{\bp}[0]{\bar{p}}
\newcommand{\bq}[0]{\bar{q}}
\newcommand{\br}[0]{\bar{r}}
\newcommand{\bs}[0]{\bar{s}}

\newcommand{\ba}[0]{\bar{a}}
\newcommand{\bb}[0]{\bar{b}}

\newcommand{\Rh}[0]{\hat{R}}

\newcommand{\Uc}[0]{\mathcal{U}}

\newcommand{\Tc}[0]{\mathcal{T}}

\newcommand{\trace}[0]{\operatorname{tr}}

\title[Backward propagation of warped products under the Ricci flow]{Backward propagation of warped product structures\\ and asymptotically conical shrinkers}

 \author{Brett Kotschwar}
 \email{kotschwar@asu.edu}
 \address{School of Mathematical and Statistical Sciences,
 	Arizona State University, Tempe, AZ 85287, USA}

  \thanks{The author was partially supported by Simons Foundation grants \#359335 and \#709004.}

\begin{document}
\begin{abstract}
We establish sufficient conditions which ensure that a locally-warped product structure propagates backward in time under the Ricci flow. As an application, we prove that if an asymptotically conical gradient shrinking soliton 
is asymptotic to a cone whose cross-section is a product of Einstein manifolds, the soliton must itself be a multiply-warped product over the same manifolds.

\end{abstract}
\maketitle

\section{Introduction}
Given a smooth solution $g(t)$ to the Ricci flow 
\begin{equation}\label{eq:rf}
\pdt g = -2\Rc(g)
\end{equation}
on $M\times [0, T]$, it is natural to ask about what one may infer about the solution at times $t< T$ from
the knowledge that $g(T)$ has some special structure. In previous work, we have shown, for example, that when $(M, g(t))$ is complete and of bounded curvature, then
any symmetries of the metric at $t= T$ are present at earlier times, 
and that if if $(M, g(T))$ has restricted holonomy at $t=T$, then so does $(M, g(t))$ (see \cite{CookKotschwarHolonomy, KotschwarRFBU,  KotschwarHolonomy, KotschwarFrequency}).  In this paper, we explore what can be said for times $t< T$ if $g(T)$
is a warped product
\begin{equation}\label{eq:wp}
    g(b, x, T) = \pi^*\ch{g}(b) + h^2(b)\bar{g}(x)
\end{equation}
on $M = B\times F$. Here $(B, \ch{g})$ and $(F, \bar{g})$ are  Riemannian manifolds, $\pi:M\longrightarrow B$ is the projection map, and $h$ is a smooth positive function on $B$.  To set our expectations,  we begin by revisiting what is known about the corresponding question for the forward propagation of the structure for times $t > T$.

\subsection{Forward propagation of warped-product structures}
The first observation to be made is that, from the perspective of the Ricci flow, warped products do not automatically qualify as ``special'' structure:
in fact, such structures are in general \emph{not} preserved by the flow. As a short computation shows, if a solution $g(t)$ to \eqref{eq:rf}  has the form
\begin{equation}\label{eq:wps}
    g(b, x, t) = \pi^*\ch{g}(b, t) + h^2(b, t)\bar{g}(x, t)
\end{equation}
on some interval $M\times I$, then either $(F, \bar{g}(t))$ is \emph{Einstein} at each $t\in I$, or $h$ is independent of $b$, in which case
$g(t)$ is an ordinary metric product of solutions on $B$ and $F$.

It is a folklore principle in the Ricci flow literature that the condition that $(F, \bar{g})$ be Einstein is also \emph{sufficient} for the forward propagation of the warped product structure,
subject to conditions met in most cases of interest.  We sketch the argument to provide some context (and contrast)
for the approach we take below, though (as we will discuss shortly) the argument cannot itself be adapted to our backward time question.

Comparing both sides of \eqref{eq:rf} under the assumption that the solution $g(t)$ has the form \eqref{eq:wps} leads to a coupled parabolic system for a family of metrics and functions on the base.
Provided one can obtain a solution to this system for a short time (with the necessary control at infinity in the noncompact case to guarantee its uniqueness) one may conclude by the usual uniqueness results for the Ricci flow \cite{ChenZhu, Hamilton3D}
that the structure must propagate forward in time under the flow.

Specifically, given fixed manifolds $(B, \ch{g}_0)$ and $(F, \gh_0)$, where the latter is Einstein with $\Rc(\gh_0) = \lambda \gh_0$, and a fixed positive $h_0\in C^{\infty}(B)$, the pursuit of  a solution $g(t)$ to \eqref{eq:rf} in the form \eqref{eq:wps} leads one to the following system of equations for a family of functions $u = \log h$ and a family of metrics $\ch{g} = \ch{g}(t)$ on $B$:
\begin{equation}\label{eq:wpsys1}
\left\{\begin{array}{rl}
        \partial_t \ch{g} &= -2\Rc(\ch{g}) + 2m \ch{\nabla}\ch{\nabla} u + 2m \ch{\nabla}u\otimes \ch{\nabla}u, \quad 
        \ch{g}(0)= \ch{g}_0,\\
        \partial_t u &= \tilde{\Delta} u + m |\ch{\nabla}u|^2 - \lambda e^{2u}, \quad u(0) = \log h_0.
     \end{array}\right.
\end{equation}
The equations \eqref{eq:wpsys1} are equivalent, up to diffeomorphism, to the system
\begin{equation}\label{eq:wpsys2}
\left\{\begin{array}{rl}
        \partial_t \tilde{g} &= -2\Rc(\tilde{g}) + 2m \tilde{\nabla}u\otimes \tilde{\nabla}u, \quad 
        \tilde{g}(0)= \ch{g}_0,\\
        \partial_t \tilde{u} &= \tilde{\Delta} \tilde{u} - \lambda e^{2\tilde{u}}, \quad \tilde{u}(0) = \log h_0.
     \end{array}\right.
\end{equation}
for $\tilde{g}$ and $\tilde{u}$ on $B$. From solutions $(\tilde{g}(t), \tilde{u}(t))$ to this latter system,
one can solve the ODE
\[
    \pd{\phi}{t} = m \tilde{\nabla}u\circ \phi, \quad \phi(x, 0) = \operatorname{Id},
\]
 and define $\ch{g}(t) = \phi_t^*\tilde{g}(t)$ and $h = e^{\tilde{u}}\circ \phi_t$ to produce a solution $g(t)$ in the form \eqref{eq:wps}.

The well-posedness of the initial-value problem in \eqref{eq:wpsys2} is discussed in a variety of places in the literature. For example, in \cite{List},
the short-time existence of solutions in the case $\lambda = 0$ is established for complete $(B, \ch{g}_0)$ of bounded curvature and $u_0$ with bounded $C^2$-norm; 
the solutions $\tilde{g}(t)$ and $\tilde{u}$ retain their uniform curvature and $C^2$-bounds, respectively.
Since complete solutions to the Ricci flow of uniformly bounded curvature are
unique \cite{ChenZhu, Hamilton3D}, it follows that every solution in this class starting from $g(0)$ must remain a warped-product.
 
\subsection{Backward propagation}
From the above discussion, we see that it is at least necessary that the fibers be Einstein if a given warped-product structure is to
propagate backward in time under the Ricci flow. Moreover, it is reasonable to expect that, at least for well-behaved solutions,
it is also sufficient. However, the line of argument above is of no direct help to us, as it would lead us to seek $\tilde{u}$ and $\tilde{g}$ solving the ill-posed \emph{terminal-value} problem associated to \eqref{eq:wpsys2}. Since we cannot hope in general to construct a warped-product competitor solution using the data from the terminal time slice (even when the fiber of that slice is Einstein), we do not expect to be able to directly reduce the problem to one of the backward uniqueness of solutions to \eqref{eq:rf}. It is conceivable that one could appeal in some way to the temporal analyticity of the equation \cite{KotschwarRFTimeAnalyticity} (see also \cite{Shao}), however, such an approach would not lend itself well to our application to shrinkers in the next section, where the solution in question is in general incomplete, and the time-slice
of interest (the end of a cone) is genuinely a terminal time-slice rather than one in the interior of the interval of existence of the solution.

Instead, we will pursue an approach in which the warped product structure is characterized without reference to an explicit warping function, framing the problem in terms of a system of tensors related to O'Neill's submersion invariants which measure the extent of the potential failure of the solution to remain a warped-product. These invariants, in turn,
satisfy a system of mixed differential inequalities to which the general backward uniqueness results from our prior work \cite{KotschwarRFBU, KotschwarFrequency} apply. 

This characterization applies equally well forward and backward in time,
and, in conjunction with our approach in \cite{KotschwarRFUniqueness} can also be used to give an alternative proof of the forward propagation of warped-product structures under the flow. Our first main result in this paper is the following, which, for convenience, we have stated in terms of the \emph{backward Ricci flow}
\begin{equation}\label{eq:brf}
\pdtau g = 2\Rc(g).
\end{equation}
Note that one can pass between a solution to \eqref{eq:rf} and \eqref{eq:brf} on $M\times [0, \Omega]$ via the change of variables $\tau = \Omega - t$.

\begin{theorem}\label{thm:warped}
Let $(B, \ch{g})$ and $(F, \bar{g})$ be Riemannian manifolds, where $\bar{g}$ is Einstein.
 Suppose $g(\tau)$ is a smooth solution to \eqref{eq:brf} on $M = B\times F$ for $\tau\in [0, \Omega]$
 which satisfies
 \[
        g(0) = \pi^*\ch{g}_{0} + h_{0}^2\bar{g}
 \]
 for some positive $h_{0}\in C^{\infty}(B)$. Assume that $\sup_{M\times [0, \Omega]}|\Rm| <\infty$.
 and, if $B$ is noncompact, that $\sup_{B}|\ch{\nabla}\log h_{0}| < \infty$.
 Then there exists a smooth family of metrics $\ch{g}(\tau)$ on $B$
and a smooth family of positive functions $h(\tau)\in C^{\infty}(B)$ such that
\begin{equation}\label{eq:warpedsol}
    g(b, x, \tau) = \pi^*\ch{g}(b, \tau) + h^2(b, \tau) \bar{g}(x)
\end{equation}
for all $\tau\in [0, \Omega]$.
\end{theorem}
On a warped-product, the collection $N$ of mean-curvature vectors of the fibers defines a horizontal vector field which is $\pi$-related to the gradient vector field $-m\ch{\nabla}\log h$ on the base $B$. Thus the assumption that the gradient of $\log h$ be bounded is equivalent to the uniform boundedness of the mean curvature of the fibers $\{b\}\times F$.  (This condition can be relaxed somewhat in applications.)

Our argument actually proves something slightly stronger.
\begin{theorem}\label{thm:warped2}
Let $g(\tau)$ be a smooth complete solution to \eqref{eq:brf} on $M\times [0, \Omega]$ with uniformly bounded curvature.
 Suppose that $\pi:(M, g(0)) \longrightarrow (B, \ch{g})$
is a Riemannian submersion and a locally-warped product. Assume that the fibers of the submersion are connected, Einstein, and have bounded mean curvature.
Then there is a smooth family of metrics $\ch{g}(\tau)$
 on $B$ such that, for all $\tau\in [0, \Omega]$,
 \[
\pi:(M, g(\tau))\longrightarrow (B, \ch{g}(\tau))
\]
 is a Riemannian submersion and a locally-warped product. In particular, if $U\subset B$ is an open subset over which $\pi^{-1}(U)$ fibers as $U\times F$,
 and $g$ admits the representation $g(0) = \pi^*\ch{g}_0 + h_0^2\bar{g}$ on $\pi^{-1}(U)$ for some $h_0 \in C^{\infty}(U)$ and Einstein metric $\bar{g}$ on $F$, then
 \[
         g(b, x, \tau) = \pi^*\ch{g}(x, \tau) + h^2(b, \tau)\bar{g}(x)
 \]
on $\pi^{-1}(U)\times [0, \Omega]$ for some $h\in C^{\infty}(U\times [0, \Omega])$.
\end{theorem}

Theorems \ref{thm:warped} and \ref{thm:warped2} also imply analogous statements for \emph{multiply}-warped products. We state only the analog of Theorem \ref{thm:warped} here
(for the case of a global multiply-warped product structure).
\begin{corollary}\label{cor:mwp}
Suppose $F = F_1\times \cdots \times F_k$ is a product of Einstein manifolds $(F_i, \bar{g}_{i})$,  and $g(\tau)$ is a smooth solution to \eqref{eq:brf}
of bounded curvature with
\[
    g(0) = \pi^*\ch{g}+ h_{1}^2\bar{g}_1 + \cdots + h_{k}^2\bar{g}_k
\]
for some Riemannian metric $\hat{g}$ and smooth positive functions $h_{i}$ on $B$. If $B$ is noncompact, assume that $\sup_B|\ch{\nabla}\log h_{i}| < \infty$ for each $i$. Then
there is a smooth family of metrics $\ch{g}(\tau)$ on $B$ and smooth families of positive functions $h_1(\tau), h_2(\tau), \ldots, h_k(\tau)$ on $B$ such that
\[
    g(b, \tau) = \pi^*\ch{g}(\tau) + h_{1}^2(\tau)\bar{g}_1 + \cdots + h_k^2(\tau)\bar{g}_k
\]
for all $\tau\in  [0, \Omega]$.
\end{corollary}
The idea is to apply Theorem \ref{thm:warped} to the $k$ different single-warped product structures one obtains by distinguishing one of the fibers $F_i$ as \emph{the} fiber, and regarding the remaining factors
\[
B\times F_1 \times\cdots F_{i-1}\times F_{i+1} \times \cdots F_k
\]
as the base, and then comparing the results obtained in each case.

As we have noted, the condition in Theorem \ref{thm:warped} that $(F, \bar{g})$ be Einstein is necessary for
the propagation of the warped product structure along the flow in general. However, we have not attempted here
to otherwise optimize the statement of Theorems \ref{thm:warped} and \ref{thm:warped2}. In particular, the condition that fibers have uniformly bounded mean curvature can be relaxed, as can the assumption that
$(M, g(\Omega))$ is \emph{everywhere} locally-warped. Our interest in Theorem \ref{thm:warped} is in the framework its proof provides to measure the failure of a space to remain a warped product under a Ricci flow, and in the potential applications of this framework to related questions. In the next section, we describe one such application to the problem of uniqueness for asymptotically conical shrinking Ricci solitons.

\subsection{Asymptotically conical shrinking solitons} \label{sec:acsapp}
Recall that a gradient shrinking Ricci soliton $(M, g, f)$  (or \emph{shrinker}) consists of a Riemannian manifold $(M, g)$ paired with a smooth function $f$ which satisfies the equation
\begin{equation}\label{eq:grs}
 \Rc(g) + \nabla\nabla f = \frac{g}{2}.
\end{equation}
The equation \eqref{eq:grs} imposes strong restrictions on the geometry at infinity of a complete
noncompact shrinker $(M, g, f)$. At present, all known examples are are either
asymptotic to products or
are asymptotically conical in a sense which we now make precise.

Given a closed $(n-1)$-dimensional manifold $(\Sigma, g_\Sigma)$, let $\Cc^{\Sigma}$ denote the cone (minus the vertex) over $\Sigma$. Thus $\Cc^{\Sigma} = (0, \infty)\times \Sigma$ with the metric $\hat{g} = dr^2 + r^2g_{\Sigma}$. Write $\Cc_a^{\Sigma} = (a, \infty)\times \Sigma$ for $a > 0$.
 Finally, for $\lambda > 0$, denote by $\rho_{\lambda}:\Cc_0^{\Sigma}\longrightarrow \Cc_0^{\Sigma}$ the dilation map $\rho_{\lambda}(r, \sigma) = (\lambda r, \sigma)$.
We take the following definition from \cite{KotschwarWangConical}.
\begin{definition} We say that a Riemannian manifold $(M, g)$ is \emph{asymptotic to $\Cc^{\Sigma}$ along the end $V\subset (M, g)$}
if, for some $a > 0$, there is a diffeomorphism $F:\Cc_a^{\Sigma}\longrightarrow V$ such that $\lam^{-2}\rho_{\lambda}^*F^\ast g\longrightarrow \hat{g}$
as $\lambda \longrightarrow \infty$
in $C^2_{\emph{loc}}(\Cc_0^{\Sigma}, \hat{g})$.
\end{definition}
By the work of O. Munteanu and J. Wang \cite{MunteanuWangConical1}, \cite{MunteanuWangConical2}, a complete shrinker for which $|\Rc|(x)\longrightarrow 0$ as $x\longrightarrow \infty$ will be asymptotically conical on each of its ends in the above sense. (In dimension four, it is enough that $R(x) \longrightarrow 0$ as $x\longrightarrow \infty$.) The local derivative estimates for the Ricci flow, moreover, imply that the convergence above is actually locally smooth.

Currently there are few nontrivial examples of complete shrinking solitons known. Besides the very recent construction of Bamler-Cifarelli-Conlon-Deruelle \cite{BamlerCifarelliConlonDeruelle} (which is asymptotic to the round cylinder $\mathbb{S}^2\times \RR^2$), all other known examples are asymptotically conical. The first of these examples were found by Feldman-Ilmanen-Knopf \cite{FeldmanIlmanenKnopf},
who constructed a family K\"ahler shrinkers with $\mathbb{U}(m)$-symmetry on the tautological line bundle of $\CP^{m-1}$. Their construction was later generalized
by Dancer-Wang \cite{DancerWang} (see also Yang \cite{Yang}) to produce a family of K\"ahler examples on complex line bundles over products of K\"ahler-Einstein manifolds with positive scalar curvature.  Recently, Angenent-Knopf \cite{AngenentKnopf} have constructed a family of examples which are doubly-warped products on $(0, \infty)\times \Ss^p\times \Ss^q$ with $p$, $q \geq 2$
and $p + q \leq 8$. These shrinkers include the first known nontrivial examples that are not K\"ahler.

Stolarski \cite{Stolarski} has shown that every complete noncompact asymptotically conical shrinking Ricci soliton arises as a finite-time
singularity model of a compact Ricci flow. It remains an important problem to classify the asymptotically conical shrinking solitons in dimension four; the classification in the K\"ahler case has recently been completed in the series of papers \cite{BamlerCifarelliConlonDeruelle, CifarelliConlonDeruelle, ConlonDeruelleSun}.

\subsubsection{The determination of the shrinker from its cone}
The main conclusion of our earlier work \cite{KotschwarWangConical} with L.~Wang is that that an asymptotically conical shrinker is determined by its asymptotic cone in the sense that
two shrinkers which are asymptotic to the same cone on some ends of each must actually be \emph{isometric} to each other near infinity on those ends. Both the statement and the method of proof
suggest that the shrinker ought to share many of the same geometrical and structural properties of its asymptotic cone. Indeed, the fundamental idea behind the proof of the uniqueness theorem in  \cite{KotschwarWangConical} is that the end of an asymptotically conical shrinker and the end of its asymptotic cone may realized (up to isometry)
as time slices of a common smooth solution to the Ricci flow. This leads to the general heuristic that
 the geometric properties which an asymptotically conical shrinker inherits from its asymptotic cone should correspond to the
geometric properties of the Ricci flow which propagate backward in time.

Guided by this heuristic principle, we have previously proven in \cite{KotschwarKaehler} that a shrinker asymptotic to K\"ahler cone must itself be K\"ahler, and in \cite{KotschwarWangIsometries}, together with Wang, that the isometry group of the link of the cone embeds in the isometry group of the shrinker. Using the framework we develop here to track the backward propagation
of warped-product structures under the Ricci flow, we obtain another result in the same direction.
\begin{theorem}\label{thm:soliton}
    Suppose $(M, g, f)$ is asymptotic to the cone $\Cc^{\Sigma}$ along the end $V\subset (M, g)$. If $(\Sigma, g_{\Sigma})$ is a product $\Sigma = \Sigma_1\times \cdots\times \Sigma_k$ of compact Einstein manifolds $(\Sigma_i, \bar{g}_i)$,
    then there is a neighborhood $W$ of infinity of $V$ for which $(W, g|_W)$ is isometric to a multiply warped product on $(a, \infty)\times \Sigma_1\times \cdots\times \Sigma_k$ 
    of the form
    \[
        g = dr^2 + h_1^2(r)\bar{g}_1 + \cdots + h_k^2(r)\bar{g}_k.
    \]
\end{theorem}

Of course, a cone $g_c = dr^2 + r^2 g_{\Sigma}$ is not merely a warped product, but a very rigid type of space characterized by the scaling invariance $\rho_{\lambda}^*g_c = \lambda^{2}g_c$. Since the conclusion of Theorem \ref{thm:soliton} is achieved without making any use of this special structure, it is natural to
ask what more might be said about a solution to the Ricci flow which terminates in such a space.  While this line of reasoning does not seem to offer any extra insight into the classification problem for asymptotically conical shrinkers, it does lead to a kind of characterization of of the class as a whole:
it is proven in \cite{KotschwarTerminalCone} that any complete solution to the Ricci flow on $M\times [0, \Omega)$ which converges to a cone on some end as $t\nearrow \Omega$ in a reasonably controlled way (i.e., satisfies
a uniform quadratic curvature bound) must actually be a shrinking Ricci soliton.

\begin{acknowledgement*}  The application to Ricci solitons in Theorem \ref{thm:soliton} relies on a modification of a backward uniqueness principle from the author's joint work \cite{KotschwarWangConical} with Lu Wang. We would like to acknowledge her substantial contribution to this result.
\end{acknowledgement*}

\section{Connection invariants associated to complementary \\ orthogonal distributions}
Our first step is to frame the problem
the backward propagation of a warped product structure as an appropriate problem of backward uniqueness.
As we have noted, we cannot simply use the ansatz (\ref{eq:wpsys1}-\ref{eq:wpsys2}) to construct a competing warped-product solution
to the Ricci flow with the given terminal data, and reduce the problem to one backward uniqueness of \emph{solutions} to \eqref{eq:rf}. Instead we will frame the problem  in terms of the vanishing of a system of invariants
which measure the extent to which a solution to \eqref{eq:rf} fails to retain a warped product structure along the flow.

These invariants will be analogues of O'Neill's invariants $A$ and $T$ for a Riemannian submersion. Since the Riemannian submersion structure will not (a priori) be preserved along the flow, we cannot work with these invariants directly, but we can instead work with algebraic analogues of $A$ and $T$ associated to a pair of complementary orthogonal and (a priori) \emph{time-dependent} distributions which evolve along with our solution $g$.
To establish our notation, we first consider the case of a fixed metric $g$.

\subsection{Invariants associated to a pair of complementary orthogonal distributions}
In this section, $(M, g)$ will denote a fixed $n$-dimensional manifold, and $\Hc$ and $\Vc$ will denote smooth complementary
orthogonal distributions on $M$, where $\Vc$ has dimension $m$.
Let $H$, $V\in \Gamma(\operatorname{End}(TM))$ denote the orthogonal projection operators
\[
H(p):T_pM\longrightarrow\Hc_p, \quad\mbox{and}\quad V(p):T_pM\longrightarrow \Vc_p.
\]
Next, by analogy with O'Neill's submersion invariants \cite{Besse}, we define families of $(2, 1)$-tensors $A$ and $T$ by
\begin{align*}
 A_{E_1}E_2 &\dfn H\nabla_{HE_1}{VE_2} + V\nabla_{HE_1}{HE_2},
\end{align*}
and
\begin{align*}
 T_{E_1}E_2 &\dfn H\nabla_{VE_1}{VE_2} + V\nabla_{VE_1}{HE_2}.
\end{align*}
We will find it useful in our computations to interact with $A$ and $T$ through the tensor
\[
L \dfn \nabla H\in T^*M\otimes \operatorname{End}(TM).
\]
In terms of $L$, the tensors $A$ and $T$ have the representations
\begin{align}
\begin{split}\label{eq:atbchar}
 A_{E_1}E_2 &= VL_{HE_1}{E_2} - HL_{HE_1}{E_2},\\
 T_{E_1}E_2 &= VL_{VE_1}{E_2} - HL_{VE_1}{E_2}.
\end{split}
 \end{align}
 
When the distributions $\Hc$ and $\Vc$ are the horizontal and vertical distributions associated to a Riemannian submersion, the condition $A\equiv 0$ implies the integrability of the horizontal distribution, and thus that the manifold splits locally as $B\times F$ with metric $g = g_B(b) + (g_F)(b, x)$ at $b\in B$ and $x\in F$. The condition $T\equiv 0$, in  turn, implies that the fibers of the submersion are totally geodesic. The vanishing of both $A$ and $T$ in this case imply that the distributions $\Hc$ and $\Vc$ are invariant under parallel transport, and that $g$ is locally reducible as a product.
 
\subsubsection{A remark on the notation}
We have chosen our notation to align with the special case that $\Hc$ and $\Vc$ arise from a Riemannian submersion,
and will use the terms ``horizontal'' and ``vertical'' to describe the distributions $\Hc$ and $\Vc$ and the vectors tangent to their fibers.
However, here, and in our application below, $\Hc$ and $\Vc$ \emph{need not be}  (a priori) the horizontal and vertical distributions
associated to any submersion. In particular, $\Hc$ and $\Vc$ and the tensors $A$ and $T$ do not share all of the properties and symmetries of their namesakes.

For example, for arbitrary complementary orthogonal distributions $\Hc$ and $\Vc$, we \emph{do not know} a priori that the ``vertical'' distribution $\Vc$ is integrable, nor that the identity
\begin{equation}\label{eq:integrability}
A_{HE_1}{HE_2} = \frac{1}{2}\cdot V[HE_1, HE_2],
\end{equation}
which expresses the relationship of $A$ to the integrability of $\Hc$, is valid, as it would be for a Riemannian submersion.

Also, while the two distributions are initially interchangeable --- we have simply decided to call one of them the vertical and the other horizontal --- the definitions we will make in terms of the distributions are not symmetric in $\Hc$ and $\Vc$,
and will depend on this choice, once made. In our application, we will eventually show that the distributions we label $\Hc$ and $\Vc$
are indeed the horizontal and vertical distributions of a warped-product structure.

\subsection{A characterization of locally-warped products.}
The condition that the metric $g$ be described locally as a warped product may also be characterized in terms of invariants of $\Hc$ and $\Vc$ (see, e.g., \cite{Besse}). For this characterization, we define the ``trace-free'' part $T^0$ of the tensor $T$ by
\begin{equation}\label{eq:t0def}
T^0_{E_1} E_2 = T_{E_1} E_2 - \frac{\langle VE_1, VE_2\rangle}{m}N + \frac{\langle N, HE_2\rangle}{m}VE_1,
\end{equation}
where $N$ is defined by
\begin{equation}
 \label{eq:ndef}
    N = \sum_{i=1}^m T_{U_i}U_i
\end{equation}
in terms of a local orthonormal vertical frame $\{U_i\}_{i=1}^m$. 
When $\Vc$ is integrable, $N$ is the mean curvature vector of the fibers. When $\Hc$ and $\Vc$ are the complementary distributions associated to a warped product of the form \eqref{eq:wp}, $N$ is $\pi$-related to the gradient vector field $-m\ch{\nabla}\log h$ on $B$.
\begin{remark}
 By the definition of $T$, the vector field $N$ is horizontal, i.e., $N_p\in \Hc_p$ for all $p\in M$. We will use this fact often below.
\end{remark}

\begin{lemma}[cf. \cite{Besse}, Proposition 9.104]\label{lem:wpconditions}
Suppose $\Hc$ and $\Vc$ are the horizontal and vertical distributions associated to a Riemannian submersion $\pi:(M, g)\longrightarrow (B, \ch{g})$. If
\begin{equation}\label{eq:wpconditions}
    A\equiv 0,\quad  T^0\equiv 0,\quad \mbox{and}\quad H D_{U} N \equiv 0,
\end{equation}
for all vertical vectors $U$,
then for every $b\in B$, there is a neighborhood $W\subset B$ of $b$ and a smooth $m$-dimensional manifold $F$ such that $\pi^{-1}(W)$ is diffeomorphic to $W\times F$ and the diagram
\[
 \begin{tikzcd}
    \pi^{-1}(W) \arrow{r}{\approx} \arrow[swap]{dr}{\pi} & W\times F \arrow{d}{\operatorname{proj}_1} \\
     & W
  \end{tikzcd}
\]
commutes. Moreover, there is a metric $\bar{g}$ on $F$ and a positive function $h\in C^{\infty}(W)$ such that $g$ has the form
\[
    g(b, x) = \pi^*\ch{g}(b) + h^2(b)\bar{g}(x)
\]
on $\pi^{-1}(W)\approx W\times F$.
\end{lemma}
\begin{proof}
 This is a slightly modified version of the converse of Proposition 9.104 in \cite{Besse}, in which, in place of our assumption that $HD_UN =0$, there is the condition that $N$ be  \emph{basic}, i.e., horizontal and $\pi$-related
 to a smooth vector field on $B$.
 
 To obtain this condition from \eqref{eq:wpconditions}, we can argue as follows. Fix $b \in M$.  Being the vertical space of $\pi$, $\Vc$ is integrable, and the condition $A\equiv 0$ implies that $\Hc$ is also integrable. Thus $b$ has a neighborhood $W$ over which $\pi$ fibers trivially as
 $\pi^{-1}(W) \approx W \times F$, where $F$ is connected, and on which $g$ has a decomposition of the form $g = \pi^*\ch{g}(b) + \hat{g}(b, x)$
 for some family of metrics $\hat{g}(b, \cdot)$ on $F$. Since $N$ is horizontal, about any point in $\pi^{-1}(W)$, we have the representation $N = \sum_{i=1}^p N^i E_i$ for some local orthonormal frame $\{E_i\}_{i=1}^p$ of basic vector fields $E_i$. Let $U$ be any vertical vector field defined near $q$.
 Using \eqref{eq:wpconditions}, we have
 \[
    U(N^i) = \langle \nabla_U N, E_i\rangle + \langle N, \nabla_U E_i\rangle = \langle A_{E_{i}}U, N\rangle + 
   \langle [U, E_i], N\rangle =  \langle [U, E_i], N\rangle
 \]
 for each $i$.
However, if $\pi_*E_i = \ch{E}_i$, then $\pi_*[U, E_i] = [\pi_*U, \pi_*E_i] = [0, \ch{E}_i] = 0$. Thus  $[U, E_i]$ is vertical, so $\langle [U, E_i], N\rangle =0$. This implies that $N$ is basic on $\pi^{-1}(W)$.
\end{proof}

In view of the third condition in \eqref{eq:wpconditions}, it will be useful to define the two-tensor
\begin{equation}\label{eq:gdef}
  G(E_1, E_2) \dfn \langle H\nabla_{VE_1}N, E_2\rangle.
\end{equation}

\subsection{Some notational conventions.}\label{ssec:notation}
In the sequel, we will need to perform some detailed tensor computations
which are sensitive to the orthogonal decomposition defined by $\Hc$ and $\Vc$.
To carry out these computations efficiently, we will make use of the following notational shorthands.

\subsubsection{Barred and underlined indices}
The first convention concerns index notation.

\begin{notation}
We will use a \textbf{barred} index to denote a precomposition of that argument
of the tensor with with the vertical projection $V$, and an \textbf{underlined} index
to denote a precomposition with the horizontal projection $H$. 
\end{notation}

For example, given a three-tensor $X$, we will write $X_{\bi\bj\uk}$ to denote the tensor
\[
    X_{\bi\bj\uk} = X_{abc}V_{ai}V_{bj}H_{ck}.
\]
In other words, $X_{\bi\bj\uk}$ represents the \emph{globally-defined}
three tensor $\tilde{X}$ where 
\[
\tilde{X}_{ijk} = X_{\bi\bj\uk}, \quad\mbox{i.e.,}\quad
\tilde{X}(E_1, E_2, E_3) = X(VE_1, VE_2, HE_3),
\]
that is, the bar/underline notation represents a \emph{modification}
to the tensor as opposed to the components of the expression of the tensor in terms
of some local orthonormal frame adapted to the splitting. (Here and throughout, a repeated index
indicates a sum over the components with respect to an orthonormal basis.)

According to this convention, the defining equations \eqref{eq:atbchar} for the tensors $A$ and $T$ become
\begin{align}
\begin{split}
 \label{eq:abindices}
    A_{ijk} &= L_{\ub{\imath}\ub{\jmath}\bar{k}} - L_{\ub{\imath}\bar{\jmath}\ub{k}}, \quad
    T_{ijk} = L_{\bar{\imath}\ub{\jmath}\bar{k}} - L_{\bar{\imath}\bar{\jmath}\ub{k}},
\end{split}
\end{align}
and equations \eqref{eq:t0def}, \eqref{eq:ndef}, and \eqref{eq:gdef}
become
\begin{align*}
        T^0_{ijk} &= T_{ijk} - \frac{V_{ij}N_{k}}{m} + \frac{V_{ik}N_j}{m},\quad
        N_k = -L_{\bar{p}\bar{p} k} = -L_{\bar{p}\bar{p}\ub{k}},
        \quad G_{ij} = \nabla_{\bi}N_{\uj}.
\end{align*}

\subsubsection{Contraction with the mean-curvature vector $N$}
The vector field $N$ will be ubiquitous in our computations, so it will also be convenient to use the convention that $N$, when it appears as the index of a tensor, represents the contraction of that argument of the tensor
with the vector field $N$.  
Thus, for example, we will write
\[
    X_{iNk} = X_{ijk} N_j = X_{ij\uk}N_k.
\]
Note that since the vector field $N$ is horizontal, we have $N_{k} = N_{\uk}$. However, we will typically just write $N_k$ (and likewise for
for $H_{ij} = H_{\ui\uj}$ and $V_{ij} = V_{\bi\bj}$).

\subsubsection{Notation for covariant derivatives}
The indexing convention above leads to a potential ambiguity for covariant derivatives. For example, the notation $\nabla_{\bi}X_{\uj}$ might be interpreted
as either the modification of $\nabla X$ or the modification of the covariant derivative
of the horizontal projection $H(X)$ of $X$.  

We will only use this notation to convey the \emph{former} meaning, i.e., the modification of the tensor $\nabla X$. In particular, by $\nabla_{\bi}X_{\uj}$, we mean
\[
    \nabla_{\bi}X_{\uj} = V_{ai} H_{bj} \nabla_a X_b,
\]
or, in other words, that the tensor represented by the expression
\[
Y_{ij} = \nabla_{\bi}X_{\uj}
\]
is that defined by
\[
 Y(E_1, E_2) =  (\nabla X)(VE_1, HE_2).
\]

For the latter interpretation, i.e., that of the modification of the covariant derivative of $H(X)$,
we will write $\nabla_{\bi}(X_{\uj})$ or simply introduce another symbol $Y_j = X_{\uj}$ and write $\nabla_{\bi}Y_j$.
Note that
\[
 \nabla_{\bi}(X_{\uj}) = \nabla_{\bi}X_{\uj} + \nabla_{\bi}H_{jp}X_{p},
\]
so these two interpretations do not agree in general.

\subsubsection{Asterisk notation}
To reduce the clutter in our expressions, we will use the standard
``asterisk'' convention with a slight twist.
\begin{notation}
By $W_1\ast W_2$,
we will mean a linear combination of contractions of the tensor product of $W_1\otimes W_2$ by the metric $g$ and/or the projections $H$ and $V$.  
\end{notation}
Thus, with an asterisk, we conceal not only factors of the metric, but potentially also
factors of the projections $H$ and $V$. According to this convention, we have as usual, that
\begin{align*}
    |X_1 \ast X_2| \leq C|X_1||X_2|
\end{align*}
for some $C = C(n)$. However, the potential presence of factors of $H$ and $V$ means that
\begin{align*}
    \nabla (X_1 \ast X_2) &= \nabla X_1 \ast X_2 + X_1 \ast \nabla X_2 + X_1 \ast X_2 \ast \nabla H\\
    &= \nabla X_1 \ast X_2 + X_1 \ast \nabla X_2 + X_1 \ast X_2 \ast (T^0 + A + N)
\end{align*}
in view of \eqref{eq:deltah}. (Note that $\nabla V = -\nabla H$.)

\subsubsection{Other conventions}
Finally, in some estimates,  we will use the notation
\[
     x\lesssim y
\]
for $x$, $y \geq 0$ to mean that $x\leq C y$ for some universal constant $C$.

We will also sometimes use
we will use $\Theta$ (or $\Theta_1$, $\Theta_2$, etc.)
to denote a polynomial with nonnegative coefficients that is of the form
$\Theta = \Theta(x_1, x_2, \ldots, x_k)$ with $\Theta(0, 0, \ldots, 0) = 0$.  Thus in particular, such $\Theta$ will satisfy
\[
    \Theta(|x_1|, |x_2|, \ldots, |x_k|) \lesssim |x_1| + |x_2| + \cdots + |x_k|.
\]

\subsection{Some basic identities satisfied by the connection invariants}
We now record a few simple identities satisfied by the tensors $A$, $L$, $T^0$, $N$, and $G$. Since $\Hc$ and $\Vc$ are merely complementary orthogonal distributions, this list will not include all of the relations that would be satisfied by the distributions associated to a Riemannian submersion.

Differentiating the equations $H^2 = H$, $V^2 = V$ and $H + V = \operatorname{Id}$, and examining components, we see that
\[
    \nabla V = -\nabla H = - L,   \quad HL_{E_1}(H E_2) = 0, \quad VL_{E_1}(VE_2) = 0,
\]
that is,
\[
    \nabla_i V_{jk}  = -L_{ijk}, \quad L_{i\bar{\jmath}\bar{k}} = L_{i\ub{\jmath}\ub{k}} = 0.
\]
In particular, $N_{\bar{k}} = - L_{\bar{p}\bar{p}\bar{k}} = 0$,
reflecting that $N$ is horizontal.

Moreover, for any horizontal vector fields $X$ and $Y$, and vertical vector fields $U$ and $V$, we have 
\begin{equation*}\label{eq:hvidents}
     A_U = T_X = 0, \quad\langle A_{X} Y, U\rangle = - \langle Y, A_{X} U\rangle, \quad
    \langle T_U V, X\rangle = - \langle V, T_U X\rangle = 0,
\end{equation*}
so
\[
 A_{\bar{\imath}jk} = A_{\ub{\imath}\bar{\jmath}\bar{k}} = A_{\ub{\imath}\ub{\jmath}\ub{k}} = 0, \quad
 A_{\ub{\imath}\bar{\jmath}\ub{k}} = - A_{\ub{\imath}\ub{k}\bar{\jmath}}, 
\]
and
\begin{gather*}
 T_{\ub{\imath}jk} = T^0_{\ub{\imath}jk} = T_{\bar{\imath}\bar{\jmath}\ub{k}}
  = T_{\bar{\imath}\ub{\jmath}\ub{k}} = T^0_{\bar{\imath}\bar{\jmath}\bar{k}}=  T^0_{\bar{\imath}\ub{\jmath}\ub{k}} = 0, \quad
 T_{\bar{\imath}\bar{\jmath}\ub{k}} = - T_{\bar{\imath}\ub{k}\bar{\jmath}} = T^0_{\bar{\imath}\bar{\jmath}\ub{k}} = - T^0_{\bar{\imath}\ub{k}\bar{\jmath}}. 
\end{gather*}
Additionally,
\[
    T^0_{\bar{p}\bar{p}k} = T^0_{\bar{p}k\bar{p}} = T^0_{k\bar{p}\bar{p}} = 0.
\]

The following identities are crucial to the computations
that follow.
\begin{lemma}\label{lem:hder}
The covariant derivative and Laplacian of $H$ satisfy
\begin{align}
\begin{split}\label{eq:nablah}
   \nabla_i H_{jk} &=  - \frac{1}{m}\left(V_{ij} N_k + V_{ik} N_j\right) + T^0_{i\uj k} - T^0_{i\bj k} + A_{i\uj k} - A_{i\bj k}
\end{split}
\end{align}
and
\begin{align}
\begin{split}\label{eq:deltah}
    \Delta H_{jk}  &=  \frac{2}{m}\left(\frac{|N|^2}{m}V_{jk} - N_j N_k\right)
    \\
    &\phantom{=}+\nabla_mT^0_{m\uj k} - \nabla_mT^0_{m\bj k} + \nabla_mA_{m\uj k} - \nabla_mA_{m\bj k}\\
  &\phantom{=} +2\left(T^0_{\bm \uj \br}T^0_{\bm\br\uk} - T^0_{\bm \bj \ur}T^0_{\bm\ur\bk} 
  -T^0_{\bm r\bk}N_r + A_{m\bj\ur} A_{m\ur \bk} - A_{m \uj \br}A_{m\br \uk}\right)\\
  &\phantom{=} - \frac{1}{m}\left(A_{mm\bj}N_k + A_{mm\bk} N_j + G_{jk} + G_{kj} + (T^0_{\bj p \bk} + T^0_{\bk p \bj}) N_p\right)\\
 \end{split}  
\end{align}
In particular,
\begin{equation}\label{eq:nablahest}
  \left|\nabla_i H_{jk} +\frac{1}{m}\left(V_{ij}N_k + V_{ik}N_j\right)\right| \lesssim |A| + |T^0|,
\end{equation}
and
\begin{align}
\begin{split}
 \label{eq:deltahest}
 \left|\Delta H_{jk} - \frac{2}{m}\left(\frac{|N|^2}{m}V_{jk} - N_jN_k\right)\right| &\lesssim \Theta (|A| + T^0|) +  |\nabla A| + |\nabla T^0| + |G|,
\end{split}
 \end{align}
for some polynomial $\Theta = \Theta(|N|, |A|, |T^0|)$ as in Section \ref{ssec:notation}.
\end{lemma}
\begin{proof}
 For \eqref{eq:nablah}, we compute directly from the identities above that
 \begin{align*}
    L_{ijk} &= L_{\bi\bj\uk} + L_{\bi \uj \bk} + L_{\ui \bj \uk} + L_{\ui \uj \bk} = - T_{i\bj k} + T_{i\uj k} - A_{i\bj k} + A_{i \uj k}\\
            &= T^0_{i\uj k}- \frac{V_{ik}N_j}{m} - T^{0}_{i\bj k} - \frac{V_{ij}N_k}{m} - A_{i\bj k} + A_{i \uj k}.
 \end{align*}
 
For \eqref{eq:deltah}, note that
\begin{align*}
 \nabla_m(T^0_{i\uj k} + A_{i\uj k}) &= L_{mjr}(T^0_{irk} + A_{irk}) + \nabla_m T^0_{i\uj k} + \nabla_m A_{i\uj k}
\end{align*}
and similarly
\[
    \nabla_m(T^0_{i\bj k} + A_{i\bj k}) = -L_{mjr}(T^0_{irk} + A_{irk}) + \nabla_m T^0_{i\bj k} + \nabla_m A_{i\bj k},
\]
while
\[
\nabla_{m}(V_{ij} N_k) = -L_{mij} N_k + V_{ij}\nabla_m N_k.
\]
Thus, using \eqref{eq:nablah}, we have
\begin{align*}
    \nabla_m L_{ijk} &= \nabla_mT^0_{i\uj k} - \nabla_mT^0_{i\bj k} + \nabla_mA_{i\uj k} - \nabla_mA_{i\bj k}
    + 2L_{mjr}(T^0_{irk} + A_{irk})\\
    &\phantom{=} + \frac{1}{m}\left(L_{mij}N_k + L_{mik}N_j - V_{ij} \nabla_m N_k - V_{ik}\nabla_m N_j\right)
\end{align*}
and so
\begin{align*}
 \nabla_m L_{mjk} &= \nabla_mT^0_{m\uj k} - \nabla_mT^0_{m\bj k} + \nabla_mA_{m\uj k} - \nabla_mA_{m\bj k}
    + 2L_{\bm jr}T^0_{\bm rk} \\
    &\phantom{=} +2L_{\um jr} A_{m rk} + \frac{1}{m}\left(L_{mmj}N_k + L_{mmk}N_j - \nabla_{\bj} N_k - \nabla_{\bk} N_j\right).
\end{align*}

Now,
\[
 L_{\bm jr}T^0_{\bm r k} = T^0_{\bm \uj \br}T^0_{\bm\br\uk} - T^0_{\bm \bj \ur}T^0_{\bm\ur\bk}
 -T^0_{\bm r\bk}N_r,
\]
and
\[
 \quad L_{\um jr} A_{m rk} = A_{m\bj\ur} A_{m\ur \bk} - A_{m \uj \br}A_{m\br \uk},
\]
while
\[
    L_{mmj} = L_{\bm\bm j} + L_{\um\um j} = -N_j - A_{\um\um \bj},
\]
and
\begin{align*}
\nabla_{\bj}N_k &= G_{jk} + V_{pk} \nabla_{\bj} N_p = G_{jk} + \nabla_{\bj}(V_{pk} N_k) -\nabla_{\bj}V_{pk} N_p = G_{jk} +L_{\bj \up k}N_p\\ &= G_{jk} + T^0_{\bj\up\bk}N_p - \frac{|N|^2}{m} V_{jk}.
\end{align*}
Combining the above identities, we obtain that
\begin{align*}
 \nabla_m L_{mjk} &= \nabla_mT^0_{m\uj k} - \nabla_mT^0_{m\bj k} + \nabla_mA_{m\uj k} - \nabla_mA_{m\bj k}\\
  &\phantom{=} +2\left(T^0_{\bm \uj \br}T^0_{\bm\br\uk} - T^0_{\bm \bj \ur}T^0_{\bm\ur\bk} 
  -T^0_{\bm r\bk}N_r + A_{m\bj\ur} A_{m\ur \bk} - A_{m \uj \br}A_{m\br \uk}\right)\\
  &\phantom{=} - \frac{1}{m}\left(A_{mm\bj}N_k + A_{mm\bk} N_j + G_{jk} + G_{kj} + (T^0_{\bj p \bk} + T^0_{\bk p \bj}) N_p\right)\\
  &\phantom{=} + \frac{2}{m}\left(\frac{|N|^2}{m}V_{jk} - N_j N_k\right)
\end{align*}
which yields \eqref{eq:deltah}.
\end{proof}

Lemma \ref{lem:hder} implies in particular that when $\Hc$ and $\Vc$ are
the horizontal and vertical distributions of a locally a warped product,
\[
 \nabla_{i} H_{jk} = -\nabla_{i}V_{jk} 
 = -\frac{1}{m}(V_{ij}N_k + V_{ik} N_j),
\]
and
\[
 \Delta H_{jk} = -\Delta V_{jk} 
 = \frac{2}{m}\left(\frac{|N|^2}{m}V_{jk} - N_jN_k\right).
\]

The tensors $\nabla H$ and $\Delta H$ will feature so frequently in our computations below that it will be useful to introduce some notation
for the discrepancy between their actual values and the expressions they satisfy on a locally-warped product.
\begin{definition}
 Let $\Ec^{\prime}$ and $\Ec^{\prime\prime}$ denote the tensors defined by
 \begin{align}\label{eq:eprimedef}
 \begin{split}
 \nabla_{i} H_{jk} &= 
    -\frac{1}{m}(V_{ij}N_k + V_{ik} N_j) + \Ec^{\prime}_{ijk},\\
  \Delta{H}_{jk} &=  \frac{2}{m}\left(\frac{|N|^2}{m}V_{jk} - N_jN_k\right) + \Ec^{\prime\prime}_{jk}. 
 \end{split}
 \end{align}
\end{definition}
Equations \eqref{eq:nablahest} and \eqref{eq:deltahest} show that
$\Ec^{\prime}$ and $\Ec^{\prime\prime}$ can be controlled
by $A$, $T^0$, $G$, $\nabla A$, and $\nabla T^0$.

\section{Curvature invariants}
Let $\Hc$ and $\Vc$ be complementary orthogonal distributions as in the previous section. We have seen that when $\Hc$ and $\Vc$ are associated
to a Riemannian submersion, the vanishing of the tensors $A$, $T^0$, and $G$ 
associated to $\Hc$ and $\Vc$ is sufficient to identify it as a locally warped product.  However, we will shortly allow $\Hc$ and $\Vc$ to evolve along with the Ricci flow, and in order to control the evolution
of the invariants $A$, $T^0$, and $G$, we will also need to control
the evolution of certain curvature quantities.

We will define these curvature invariants
here first in terms of a fixed pair of complementary distributions. However, the motivation for the particular choice of their definitions will have to wait for the computations in the following sections.

\subsection{Some notation}
First let us introduce a bit of notation.
Given any $k$-tensor $X$, let
\begin{align*}
    X^H(E_1, \ldots, E_k) &= X(HE_1,\ldots, HE_k),
\end{align*}
and
\begin{align*}
    X^V(E_1, \ldots, E_k) &= X(VE_1, \ldots, VE_k),
\end{align*}
denote the actions of $H$ and $V$ on $X$. According to our index convention, we have
\[
     X^H_{i_1\cdots i_k} = X_{\ub{\imath}_1\cdots \ub{\imath}_k},      \quad X^V_{i_1 \cdots i_k} = X_{\bar{\imath}_1 \cdots \bar{\imath}_k}.
\]
Also, as usual, we use $X\odot X^{\prime}$ to denote the Kulkarni-Nomizu product
\[
    (X\odot X^{\prime})_{ijkl} = X_{il}X^{\prime}_{jk} + X_{jk}X^{\prime}_{il} - X_{ik}X^{\prime}_{jl} - X_{jl}X^{\prime}_{ik}
\]
of the $2$-tensors $X$ and $X^{\prime}$.
\subsection{An invariant associated to $\Rm$}
When $\Hc$ and $\Vc$ arise from a Riemannian submersion, 
the curvature tensor on the total space $M$ can be completely expressed
via O'Neill's equations \cite{ONeill}
in terms of the curvature tensors of the base and fiber
and the invariants $A$ and $T$ and their first covariant derivatives.
(See, e.g., Proposition 9.28 in \cite{Besse}.)  For a general pair of complementary distributions, the associated invariants $A$
and $T$ lack some of the symmetries they would possess if $\Hc$ and $\Vc$ arose from a genuine Riemannian submersion, but there are analogous formulas which express what is effectively the same thing for all but the purely vertical and purely horizontal components.

We will not need these formulas and will not record them here. What is important to us
are the observations that,
on a warped-product, the mixed (i.e., neither purely vertical nor purely horizontal) components of the curvature tensor
either vanish or are of a relatively simple form, and that, relative to an arbitrary pair of complementary orthogonal distributions, the invariants $A$, $T^0$, and $G$
and $\nabla A$ and $\nabla T^0$ measure the extent to which these mixed components of curvature fail to have
this simple form.

Define
\begin{equation}\label{eq:qdef}
  Q \dfn \Rm - \Rm^H - \Rm^V - \frac{1}{m}W \odot V,
\end{equation}
where, in index notation,
\begin{equation}\label{eq:wdef}
  W_{il} \dfn (\trace_V\Rm)^H_{il} = R_{\ui\bp\bp\ul}, \quad \trace_V(\Rm)_{il} = V_{pq}R_{ipql} = R_{i\bp\bp l}.
\end{equation}

\begin{proposition}\label{prop:qprop}
The tensor $Q$ satisfies
\begin{equation}\label{eq:qeq}
 Q =  (N + A + T^0) \ast (A + T^0)
 + \nabla A + \nabla T^0 + G,
\end{equation}
and $W$ satisfies
\begin{align*}
 W &= \frac{ N\otimes N}{m} - (\nabla N)^H + (N + A+ T^0) \ast (A + T^0)
 + \nabla A + \nabla T^0 + G.
\end{align*}
(Here, we identify $N^{\flat}$ with $N$.) In particular,
\begin{equation}
    |Q| \lesssim \Theta(|A|+|T^0|) + |\nabla A| + |T^0| + |G|
\end{equation}
for some polynomial $\Theta = \Theta(|N|, |A|, |T^0|)$, and
$Q$ vanishes on a warped product.
\end{proposition}

\begin{proof}
  On one hand, by \eqref{eq:nablah},
  \begin{align*}
    \nabla_l \nabla_i H_{jk} - \nabla_i\nabla_l H_{jk}&=
    -R_{lijp}H_{pk} - R_{likp}H_{jp} =R_{ilj\uk} + R_{ilk\uj}.
  \end{align*}
On the other, by Lemma \ref{lem:hder},
\begin{align*}
 \nabla_{i}H_{jk} &= T^0_{i\uj k} - T^0_{i\bj k}
 + A_{i\uj k} - A_{i\bj k} -\frac{1}{m}(V_{ij}N_k + V_{ik}N_j),
\end{align*}
so that
\begin{align*}
    \nabla_l\nabla_i H_{jk} &= 
    -\frac{1}{m^2}\big(V_{lj} N_i N_k + V_{lk}N_{i}N_j + 2V_{il}N_j N_k \big)\\
    &\phantom{=\;} -\frac{1}{m}\big(V_{ij}\nabla_l N_k + V_{ik} \nabla_l N_j\big) + \Cc,
\end{align*}
where, here and below, $\Cc$ denotes an expression of the form
\begin{align}
\begin{split}\label{eq:cexpdef}
  \Cc &=  N \ast \Ec^{\prime} + (N + \Ec^{\prime}) \ast (A + T^0)
 + \nabla A + \nabla T^0 + G\\
&= (N + A + T^0) \ast (A + T^0)
 + \nabla A + \nabla T^0 + G.
\end{split}
 \end{align}
Here we have used that
\begin{align*}
 \nabla_l (V_{ij}N_k  + V_{ik}N_j)
 &= \nabla_lV_{ij} N_k  + \nabla_l V_{ik} N_j + V_{ij}\nabla_l N_k + V_{ik} \nabla_l N_j\\
 &= \frac{1}{m^2}\big(V_{li}N_j N_k + V_{lj} N_i N_k
 + V_{li} N_k N_j + V_{lk}N_{i}N_j\big)\\
 &\phantom{=}+ \frac{1}{m}\big(V_{ij}\nabla_l N_k + V_{ik} \nabla_l N_j\big) 
 + N \ast \Ec^{\prime},
\end{align*}
and that
\[
 \nabla_l( T^0_{i\uj k}) = \nabla_l T^0_{i\uj k} 
 + \nabla_l H_{jp} T^0_{ipk} = \nabla T^0 + (\Ec^{\prime} + N)\ast T^0,
\]
with similar expressions for the other terms with modified indices.

Hence
\begin{align*}
 R_{ilj\uk} + R_{ilk\uj} 
 &= \frac{1}{m^2}\left(
 V_{ij}N_k N_l + V_{ik}N_j N_l - V_{jl} N_i N_k + V_{kl}N_{i}N_j\right)\\
 &\phantom{=}+ \frac{1}{m}\left( V_{jl}\nabla_i N_k
 + V_{kl}\nabla_i N_j - V_{ij}\nabla_l N_k - V_{ik} \nabla_l N_j\right) +\Cc.
\end{align*}
In particular,
\[
 R_{il\bj\uk} = \frac{1}{m^2}\left((N_kN_l  - m\nabla_{l}N_{\uk})V_{ij} +(m\nabla_iN_{\uk} - N_i N_k)V_{jl}\right) + \Cc,
\]
so we have
\begin{gather*}
 R_{\bi\bl\bj\uk} = -\frac{1}{m}\nabla_{\bl}N_{\ul} V_{ij} + \Cc
 = - \frac{1}{m}G_{\bi\ul}V_{ij} + \Cc = \Cc,\\
 R_{\ui\bl\bj\uk} = \frac{1}{m^2}(N_i N_k - m\nabla_{\uk}N_{\ui})V_{jl} + \Cc,
\end{gather*}
and
\begin{align*}
 R_{\ui\ul\bj\uk} &= \Cc.
\end{align*}
 Then the symmetries of the curvature tensor imply
that $R_{ijkl}$ is schematically of the form \eqref{eq:cexpdef} whenever
whenever it has an odd number of vertical or horizontal
components. 
Moreover, by the Bianchi identity,
\[
 R_{\uk\ul\bi\bj} = R_{\bi\bj\uk\ul} = -R_{\bj\uk\bi\ul}  -R_{\uk\bi\bj\ul} = R_{\uk\bj\bi\ul} - R_{\uk\bi\bj\ul}
=\Cc.
\]
Thus the only components
which do not vanish up to a term of the schematic form $\Cc$
are those of the form 
$R_{\ui\bj\bk\ul}$ and those obtained from it by symmetry. This yields \eqref{eq:qeq}.
Finally, on a warped product, the tensors
$A$, $T^0$, and $G$ vanish, so $Q$
does as well.
\end{proof}

\subsection{Other curvature invariants}
Proposition \ref{prop:qprop} implies that the extent to which $Q$ fails to vanish is controlled by just by the tensors $A$, $T^0$, and $G$ and their covariant derivatives. However, in order to control
the evolution equations of $A$, $T^0$, and $G$, we will need to
introduce invariants which
are first-order in the curvature and and not only measure the failure of the warped product structure
to be preserved, but the extent to which (what ought to be) the fiber metrics fail to remain Einstein.
We will need three such invariants.

First, we define the two-tensor
\begin{equation}\label{eq:mdef}
    M \dfn \Rc - \Rc^H - \frac{\Rh}{m}V,
\end{equation}
where
\[
    \Rh \dfn \trace_V(\Rc) \dfn R_{\bp\bp} = V_{ab}R_{ab}.
\]
In components, according to our convention, $M$ is given by
\[
    M_{ij} = R_{ij} - R_{\ui\uj} - \frac{\Rh}{m}V_{ij}.
\]
The tensor $M$ measures the failure of $\Rc$ to possess a diagonal block decomposition relative
to $\Hc\oplus\Vc$ with a trace-free vertical block. Note that, while some components of $M$
can be expressed in terms of $Q$ (e.g., $M_{\bi\uj} = Q_{\bi p p \uj}$), $M$ is not recoverable from $Q$ alone: for example,
$M$ carries information about the vertical components of $\Rc$ which is not captured
by $Q$. As we will see below, $M$ vanishes on a locally-warped product with Einstein fibers.

Next, we define the three-tensor $P$ by
\begin{equation}\label{eq:pdef}
 P_{ijk} \dfn \nabla_{i}R_{jk} - \nabla_{\ui}R_{\uk\uk}
 - \frac{1}{m}\left(\nabla_{\ui}R_{\bp\bp}V_{jk}
 + \nabla_{\bp}R_{\uj\bp}V_{ik}
 + \nabla_{\bp}R_{\bp\uk}V_{ij}\right).
\end{equation}
Thus $P = \Pc(\nabla \Rc)$
where the projection $\Pc:T^{3}(T^*M) \longrightarrow T^{3}(T^*M)$ is characterized by the fact that
that $\Pc(X)^H = 0$ and $\Pc(X - X^V)$ is vertical-trace-free. (See section \ref{sec:pproj}.) Note that
\[
 \nabla_{\bi}R_{\bj\bk} = P_{\bi\bj\bk}, \quad \nabla_{\bi}R_{\uj\uk}
 = P_{\bi\uj\uk}, \quad \mbox{and}\quad \nabla_{\ui}R_{\bj\uk} = P_{\ui\bj\uk}.
\]
The tensor $P$ also vanishes identically on a locally warped-product structure with Einstein fibers.

We will see soon that the quantities $M$ and $P$ and their first covariant derivatives are sufficient
to control the evolution equations for $A$, $T^0$, $G$, $\nabla A$, and $\nabla T^0$. However, to control the evolution of $P$, we will need the an additional invariant involving the full covariant derivative of $\Rm$. Thus we define the five-tensor $U$ by
\begin{align}\label{eq:udef}
\begin{split}
    U_{aijkl} &\dfn \nabla_a R_{i\uj\uk l} - \nabla_{\ua}R_{\ui\uj\uk\ul}
     - \frac{1}{m}\left(V_{ai}\nabla_{\bp}R_{\bp\uj\uk\ul} + V_{al}\nabla_{\bp}R_{\ui \uj\uk \bp}
     + V_{il}\nabla_{\ua}R_{\bp\uj\uk\bp}\right)
\end{split}
\end{align}
In other words,
\[
U \dfn \Uc(\nabla \Rm) = \Pc(\Hc(\nabla\Rm))
\]
where here the projection $\Pc$ is as defined above and acts on the first, second, and fifth indices:
\begin{gather*}
  \Pc(X)_{aijkl} = X_{aijkl} - X_{\ua\ui j k \ul} - \frac{1}{m}(V_{ai}X_{\bp\bp j k \ul}
  + V_{al}X_{\bp\ui j k \bp} + V_{il}X_{\ua\bp j k \bp}),
\end{gather*}
and the projection $\Hc$ is the projection on to the horizontal components which acts on the third and fourth indices:
\begin{gather*}
  \Hc(X)_{aijkl} = X_{ai\uj\uk l}.
\end{gather*}
 (See Section \ref{sec:uproj}). The tensor $U$ vanishes on any locally-warped product structure. Since (in particular) $\Uc$ annihilates the purely vertical components of $\nabla \Rm$, we cannot recover $P$ from $U$ alone.

\begin{proposition}\label{prop:wpmpu}
 Let $(B, \ch{g})$ and $(F, \bar{g})$ be Riemannian manifolds, where  $\bar{g}$ is Einstein with $\Rc(\bar{g}) = \lambda \bar{g}$
 and $\operatorname{dim}F = m$. Let  $\pi: B\times F \longrightarrow B$ be the projection and assume that $g$ is a metric on $B\times F$
 with the warped product representation
 \[
      g(b, x) = \pi^*\ch{g}(b) + h^2(b)\bar{g}(x)
 \]
for some  positive $h\in C^{\infty}(B)$. If $M$, $P$, and $U$ are the tensors defined as above in terms of the horizontal and vertical projections $H$ and $V$, then $M\equiv 0$, $P\equiv 0$, and $U \equiv 0$.
\end{proposition}
\begin{proof}
Let $(x^{\alpha})$ be local coordinates on a product neighborhood $U = U_1\times U_2$ where $(x^{\alpha})$, $\alpha =1, \ldots, n-m$, are local coordinates on $U_1\subset B$ and $(x^{\alpha})$, $\alpha =n-m+1, \ldots n$ are coordinates on $U_2\subset F$.  We will use the convention that lower-case letters denote an index from $1$ to $m-n$ and upper-case letters denote an index from $m-n+1$ to $n$.

We start with $M$. Straightforward, if somewhat tedious, calculations in these local coordinates give that
\[
  M_{jK} = R_{j K} = g^{ab}R_{aj K b} + g^{AB}R_{AjKB} = g^{ab}Q_{ajKb} +g^{AB}Q_{AjKB} = 0,
\]
and
\begin{align*}
  R_{JK} &= g^{ab}R_{aJKb} + g^{AB}R_{AJKB} = -mh\ch{\Delta}h V_{JK} + \bar{R}_{JK} - (m-1)|\ch{\nabla}\log h|^2V_{JK}\\
   & = \frac{1}{h^2}\left(\lambda - mh\ch{\Delta}h - (m-1)|\ch{\nabla}h|^2\right)V_{JK},
\end{align*}
as
\[
\bar{R}_{JK} = \lambda \bar{g}_{JK} = \lambda h^{-2}V_{JK},
\]
since $\bar{g}$ is Einstein. Here, as usual, $V_{\alpha\beta} = g_{\beta\gamma}V_{\alpha}^{\gamma}$.
Thus $R_{JK} = (\operatorname{tr}_V(\Rc)/m) V_{JK}$, so $M_{JK} = 0$. Since $M_{jk} = 0$ by definition, we have $M \equiv 0$.

As for $P$, we may similarly compute that
\begin{align*}
 \nabla_{A}R_{jk} = \nabla_{a}R_{Jk} = 0,
\end{align*}
and that
\begin{align*}
  \nabla_{A} R_{JK} &= \bar{\nabla}_A \bar{R}_{JK} = 0,
\end{align*}
as $\bar{g}$ is Einstein, so  $P_{Ajk} = \nabla_AR_{jk} = 0$, $P_{aJk} = \nabla_aR_{Jk} =0$, and $P_{AJK} = \nabla_AR_{JK} = 0$.
Likewise, we find that
 \begin{align*}
   \nabla_{A} R_{Jk} &=  -(\ch{\nabla}_k \log h) \bar{R}_{AJ} + (m-1)\bar{g}_{AJ}\left((\ch{\nabla}_k\log h) |\ch{\nabla}|^2_{\ch{g}} \
   - \ch{\nabla}_k\ch{\nabla}^l h \ch{\nabla}_lh\right)\\
   &= \frac{V_{AJ}}{h^2}\left((m-1)\left((\ch{\nabla}_k\log h)|\ch{\nabla}|^2_{\ch{g}} \
   - \ch{\nabla}_k\ch{\nabla}^l \ch{\nabla}_lh\right) - \lambda \ch{\nabla}_k\log h\right)
 \end{align*}
so $P_{AJk} = 0$. Finally, again using the Einstein condition, we have
\begin{align*}
  \nabla_a R_{JK} &= - 2(\ch{\nabla}_a\log h) \bar{R}_{JK} -(m-1)\bar{g}_{JK}\left((\ch{\nabla}_a\log h) |\ch{\nabla}h|^2_{\ch{g}}
  + \ch{\nabla}_a\ch{\nabla}_p h \ch{\nabla}^p  h\right)\\
  &=  -\frac{2}{h^2}V_{JK}\left(\lambda (\ch{\nabla}_a\log h) + (m-1)\left((\ch{\nabla}_a\log h) |\ch{\nabla}h|^2_{\ch{g}}
  + \ch{\nabla}_a\ch{\nabla}_p h \ch{\nabla}^p  h\right)\right)
\end{align*}
so $P_{aJK} = 0$, and it follows that $P \equiv 0$.

For $U$, we may compute directly that $U_{AIjkL} = \nabla_AR_{IjkL} = 0$, and
and since
\[
 \quad U_{aijkL} = \nabla_a R_{ijkL} = 0,
\]
we may use the algebraic symmetries of $\nabla \Rm$ to deduce also that $U_{Aijkl} = \nabla_A R_{ijkl} =  0$ and $U_{aIjkl}= \nabla_a R_{Ijkl} = 0$.
Since $U_{aijkl} = 0$ by definition, we need only to consider the components $U_{\alpha \beta jk \gamma}$ where exactly two of $\alpha$, $\beta$, $\gamma$ correspond to vertical entries. But a computation shows that
\begin{align*}
   \nabla_{a}R_{IjkL} &= \bar{g}_{IL}(\ch{\nabla}_j\ch{\nabla}_k h\ch{\nabla}_a h - h\ch{\nabla}_a\ch{\nabla}_j\ch{\nabla}_k h)\\
   &=  h^{-2}V_{IL}(\ch{\nabla}_j\ch{\nabla}_k h\ch{\nabla}_a h - h\ch{\nabla}_a\ch{\nabla}_j\ch{\nabla}_k h),
\end{align*}
so we have $U_{aIjkL} = 0$. Then we also have
\begin{align*}
    U_{AIjkl} &= \nabla_{A}R_{Ijkl} -\frac{1}{m}g^{PQ}\nabla_{P}R_{Qjkl}V_{AI}\\
              &= \nabla_{l}R_{AkjI} - \nabla_kR_{AljI} - \frac{1}{m}g^{PQ}\nabla_{l}R_{PkjQ}V_{AI} + \frac{1}{m}g^{PQ}\nabla_{k}R_{PljQ}V_{AI}\\
              &= U_{lAkjI} - U_{kAljI} = 0,
\end{align*}
and thus also
\[
   U_{AijkL} =  U_{ALkji} = 0,
\]
so $U \equiv 0$ as claimed.
\end{proof}

Finally, it will also be convenient to introduce notation for the vertical projection
\begin{equation}
 \label{eq:sdef}
    S_i \dfn \nabla_{\bi} R = (\nabla R ^V)_i
\end{equation}
of the differential of the scalar curvature. The one-form $S$ also vanishes on a warped product with Einstein fibers, but since
\begin{equation}\label{eq:schar}
  S_i = \nabla_{\bi} R  = \nabla_{\bi}R_{\bp\bp}
  + \nabla_{\bi} R_{\up\up} = P_{\bi\up\up} + P_{\bi\up\up},
\end{equation}
we will not need to include it in our final system.

\section{Invariants associated to a time-dependent splitting of $TM$.} 
Now we consider a smooth solution $g(\tau)$ to the backward Ricci flow \eqref{eq:brf} on $M\times [0, \Omega]$.  Let $\mathcal{V}_0\subset TM$ be a smooth $m$-dimensional distribution on $M$ with orthogonal complement $\mathcal{H}_0 = \mathcal{V}_0^{\perp}$. Let $V_0:TM\longrightarrow \mathcal{H}_0$ denote projection onto $\mathcal{V}$ and let $V = V(\tau):TM\longrightarrow TM$ be the solution to the linear fiberwise ODE
\begin{equation}\label{eq:vev}
    \pdtau V_i^j =  R_i^c V_c^j - R_c^j V_i^c, \quad V(0) = V_0,
\end{equation}
that is, to
\[
D_{\tau} V = 0, \quad V(0) =V_0,
\]
where $D_{\tau}$ is the operator defined by
\begin{align*}
 D_{\tau}W_{b_1b_2 \ldots b_k}^{a_1a_2\ldots a_l} &= \pdtau W_{b_1b_2 \ldots b_k}^{a_1a_2\ldots a_l}
 - R_{b_1}^{c}W_{cb_2\ldots b_k}^{a_1a_2\ldots a_l}  - R_{b_2}^{c}W_{b_1c\ldots b_k}^{a_1a_2\ldots a_l}
 - \cdots - R_{b_k}^{c}W_{b_1b_2\ldots c}^{a_1a_2\ldots a_l}\\
 &\phantom{=} + R_{c}^{a_1}W_{b_1b_2\ldots b_k}^{ca_2\ldots a_l}
 + R_{c}^{a_2}W_{b_1b_2\ldots b_k}^{a_1c\ldots a_l} + \cdots + R_{c}^{a_l}W_{b_1b_2\ldots b_k}^{a_1a_2\ldots c}.
\end{align*}
The operator $D_{\tau}$ coincides with the ``total $\tau$-derivative'' taken relative to evolving $g(\tau)$-orthonormal frames.
In particular, $D_{\tau} g = 0$, and
\[
      \pdtau \langle X, Y\rangle_{g(\tau)} = \langle D_{\tau} X, Y\rangle_{g(\tau)} + \langle X, D_{\tau} Y\rangle_{g(\tau)},
\]
for smooth families of tensors $X$ and $Y$. For a more geometric interpretation of the operator, see, e.g., Appendix F of \cite{RFV2P2}.

Then define
\[
H(\tau) \dfn \operatorname{Id} - V(\tau),
\]
so that
\begin{equation}\label{eq:hev}
D_{\tau} H = 0, \quad H(0) = H_0,
\end{equation}
and let
\begin{equation}\label{eq:hvev}
    \Vc(\tau) \dfn\operatorname{image}({V(\tau))}, \quad
    \mathcal{H}(\tau) \dfn \operatorname{image}{(H(\tau))} =\mathcal{V}(\tau)^{\perp},
\end{equation}
to obtain families of complementary smooth distributions extending $\mathcal{V}_0$ and $\mathcal{H}_0$.

We will write $A = A(\tau)$, $T = T(\tau)$, $N = N(\tau)$, $T^0 = T^0(\tau)$, and $G = G(\tau)$ for the invariants defined in the previous section associated to the distributions $\Hc = \Hc(\tau)$ and $\Vc = \Vc(\tau)$.

Eventually, we will take $\Hc_0$ and $\Vc_0$ to be the horizontal and vertical distributions associated to a locally-warped product structure
at $\tau = 0$. Although the result we are trying to prove
implies that $\Hc(\tau)$ and $\Vc(\tau)$ are actually fixed, it turns out to be convenient to let them vary a priori in time in such a way that they are guaranteed to remain orthogonal with respect to the evolving metric.

Our use of \eqref{eq:vev} to extend the distributions
to $[0, \Omega]$ --- which essentially amounts to using Uhlenbeck's trick to pull everything back to a vector bundle with a fixed metric on the fibers  --- is not the only choice for our this problem. Another natural option, would be to hold $\Vc$ fixed and let the fibers of $\Hc(\tau)$ vary in time as the $g(\tau)$-orthogonal complements of the fibers of $\Vc$. This scheme would have the advantage that $\Vc$ would remain integrable (as the kernel of the differential $d\pi$ of the projection
$\pi:M\longrightarrow B$). However, it would also entail a more complicated
evolution equation for the projections $V$ and $H$, and for the evolution equations of the invariants defined
in terms of these projections.

\subsection{Evolution equations for $A$, $T$, and $G$}\label{sec:conninvev}
Our next task is to determine how $A$, $T$, and $G$ change under the flow. We start by computing the evolution equation for $L =\nabla H$.
It will be convenient to introduce the temporary notation
\begin{equation}\label{eq:edef}
  E_{ijk} \dfn \nabla_{\bp}R_{\bp ijk} = \nabla_k R_{ij} - \nabla_j R_{ik}.
\end{equation}

\begin{proposition}\label{prop:bev} The tensor $L$ satisfies the evolution equation
\begin{equation}\label{eq:lev}
 D_{\tau} L_{ijk} = E_{ij\uk} - E_{i\uj k} -R_{ip}L_{pjk}.
\end{equation}

\end{proposition}
\begin{proof}
Using $D_{\tau} H = 0$, we compute that
\begin{align*}
 D_{\tau} L_{ijk} &= [D_{\tau}, \nabla_i] H_{jk}  \\
             &= \left(\nabla_p R_{ij} - \nabla_j R_{ip}\right)H_{pk}
             + \left(\nabla_p R_{ik} - \nabla_k R_{ip}\right)H_{pj}  -R_{ip}L_{pjk}\\
             &=  \left(\nabla_j R_{\ub{k}i} - \nabla_{\ub{k}} R_{ji}\right)
             + \left(\nabla_{\ub{\jmath}} R_{ik} - \nabla_k R_{i\ub{\jmath}} \right) - R_{ip}L_{pjk}\\
            &=   E_{ij\ub{k}} - E_{i\ub{\jmath}k} - R_{ip}L_{pjk},
\end{align*}
as claimed.
\end{proof}

Next, using the identity \eqref{eq:abindices}, we compute the evolution
of $A$. Note that since $D_{\tau} H \equiv 0$ and $D_{\tau} V\equiv 0$, we have, for example,
that
\[
    D_{\tau}(X_{\bi\uj k}) = D_{\tau}(V_{ip}H_{jq} X_{pqk}) =
    V_{ip}H_{jq} D_{\tau}X_{pqk} =(D_{\tau} X)_{\bi\uj k} = D_{\tau} X_{\bi\uj k},
\]
and there is no ambiguity in the meaning of $D_{\tau} X_{\bi\uj k}$.

\begin{proposition}
\label{prop:aev}
The tensor $A$ satisfies the evolution equation
\begin{align}\label{eq:aev}
\begin{split}
 D_{\tau} A_{ijk} &= M_{\ub{\imath}\bar{p}}( T^0_{\bar{p}\bar{\jmath}\ub{k}} - T^0_{\bar{p}\ub{\jmath}\bar{k}})
  -R_{\ub{\imath}\ub{p}}(A_{\ub{p}\ub{\jmath}\bar{k}} - A_{\ub{p}\bar{\jmath}\ub{k}}) - P_{\bk\uj\ui} + P_{\uj\ui\bk} -
  P_{\uk\bj\ui} + P_{\bj\uk\ui}\\
  &\phantom{=} + \frac{1}{m}\left(M_{\ub{\imath}\bar{\jmath}}N_{\ub{k}} - M_{\ub{\imath}\bar{k}}N_{\ub{\jmath}}\right).
\end{split}
\end{align}
In particular,
\begin{align}\label{eq:dtaest}
 |D_{\tau} A| &\lesssim  |\Rm|(|A| + |T^0|) + |N||M|  + |P|.
\end{align}
 
\end{proposition}
\begin{proof}
 Starting with the identity $A_{ijk} = L_{\ub{\imath}\ub{\jmath}\bar{k}} - L_{\ub{\imath}\bar{\jmath}\ub{k}}$, and using \eqref{eq:lev}, we compute
 \begin{align*}
    D_{\tau} A_{ijk} & = D_{\tau} L_{\ub{\imath}\ub{\jmath}\bar{k}} - D_{\tau}L_{\ub{\imath}\bar{\jmath}\ub{k}}\\
    &=  R_{\ub{\imath}p}L_{p\bar{\jmath}\ub{k}}  - R_{\ub{\imath}p}L_{p\ub{\jmath}\bar{k}} -
    E_{\ub{\imath}\ub{\jmath}\bar{k}}
- E_{\ub{\imath}\bar{\jmath}\ub{k}}\\
&= R_{\ub{\imath}\bar{p}}(- L_{\bar{p}\bar{\jmath}\ub{k}} - L_{\bar{p}\ub{\jmath}\bar{k}})
  +R_{\ub{\imath}\ub{p}}( L_{\ub{p}\bar{\jmath}\ub{k}} - L_{\ub{p}\ub{\jmath}\bar{k}})-E_{\ub{\imath}\ub{\jmath}\bar{k}} - E_{\ub{\imath}\bar{\jmath}\ub{k}}.
 \end{align*}

Now, for example,
\[
A_{\ub{p}\ub{\jmath}\bar{k}}= L_{\ub{p}\ub{\jmath}\bar{k}}, \quad A_{\ub{p}\bar{\jmath}\ub{k}} = - L_{\ub{p}\bar{\jmath}\ub{k}},  \quad L_{\bar{p}\ub{\jmath}\bar{k}} = T^0_{ \bar{p}\ub{\jmath}\bar{k}} - \frac{V_{pk}}{m}N_{\ub{\jmath}},
\]
so
\begin{align*}
    D_{\tau} A_{ijk} &= R_{\ub{\imath}\bar{p}}(T^0_{\bar{p}\bar{\jmath}\ub{k}} - T^0_{\bar{p}\ub{\jmath}\bar{k}})
  -R_{\ub{\imath}\ub{p}}(A_{\ub{p}\ub{\jmath}\bar{k}} + A_{\ub{p}\bar{\jmath}\ub{k}})-E_{\ub{\imath}\ub{\jmath}\bar{k}} - E_{\ub{\imath}\bar{\jmath}\ub{k}}\\
  &\phantom{=} + \frac{R_{\ub{\imath}\bar{p}}}{m}(V_{pj}N_{\ub{k}} - V_{pk}N_{\ub{\jmath}})\\
  &= R_{\ub{\imath}\bar{p}}(T^0_{\bar{p}\bar{\jmath}\ub{k}}-T^0_{\bar{p}\ub{\jmath}\bar{k}})
  -R_{\ub{\imath}\ub{p}}(A_{\ub{p}\ub{\jmath}\bar{k}} + A_{\ub{p}\bar{\jmath}\ub{k}})-E_{\ub{\imath}\ub{\jmath}\bar{k}} - E_{\ub{\imath}\bar{\jmath}\ub{k}}\\
  &\phantom{=} + \frac{1}{m}\left(R_{\ub{\imath}\bar{\jmath}}N_{\ub{k} - R_{\ub{\imath}\bar{k}}N_{\ub{\jmath}}}\right).
\end{align*}
Since
\[
 E_{\ui\uj\bk} = \nabla_{\bk}R_{\uj\ui} - \nabla_{\uj}R_{\bk\ui}
 = P_{\bk\uj\ui} - P_{\uj \bk\ui},
\]
and, likewise,
\[
 E_{\ui\uj\bk} = P_{\bk\uj\ui} - P_{\uj\bk\ui},
\]
we obtain \eqref{eq:aev}.
\end{proof}

In order to compute the evolution equations for $T^0$ and $G$,
we first need the evolution equation for $N$.

\begin{proposition}
\label{prop:nev}
The vector field $N$ satisfies the evolution equation
\begin{align}\label{eq:nev}
\begin{split}
 D_{\tau} N_k &=  M_{\bar{p}\ub{q}}A_{\ub{q}\bar{p}\ub{k}} - M_{\bar{p}\bar{q}}T^0_{\bar{q}\bar{p}\ub{k}}  - E_{\bar{p}\bar{p}\ub{k}}
 - \frac{{\Rh}}{m}N_{k}.
\end{split}
\end{align}

\end{proposition}
\begin{proof}
    Recall that $N_{k} = -L_{\bar{p}\bar{p}\ub{k}}$, so
\begin{align*}
    D_{\tau} N_k &= D_{\tau} L_{\bar{p}\bar{p}\ub{k}}\\
            &= R_{\bar{p}q}L_{q\bar{p}\ub{k}} - E_{\bar{p}\bar{p}\ub{k}}\\
            &= R_{\bar{p}\bar{q}}L_{\bar{q}\bar{p}\ub{k}} - R_{\bar{p}\ub{q}}L_{\ub{q}\bar{p}\ub{k}} - E_{\bar{p}\bar{p}\ub{k}}\\
            &= \left(R_{\bar{p}\bar{q}}-\frac{{\Rh}}{m}V_{pq}\right)L_{\bar{q}\bar{p}\ub{k}} - \frac{{\Rh}}{m}N_{k}
            + R_{\bar{p}\ub{q}}A_{\ub{q}\bar{p}\ub{k}} - E_{\bar{p}\bar{p}\ub{k}}.
\end{align*}
Recalling the definition of the tensor $M$ and using that
\[
    L_{\bar{q}\bar{p}\ub{k}} = - T_{\bar{q}\bar{p}\ub{k}} = -\left(T^0_{\bar{q}\bar{p}\ub{k}} - \frac{V_{pq}}{m} N_k\right),
\]
we find that 
\begin{align*}
 D_{\tau} N_k
 &=   M_{\bar{p}\ub{q}}A_{\ub{q}\bar{p}\ub{k}} -M_{\bar{p}\bar{q}}T^0_{\bar{q}\bar{p}\ub{k}} - E_{\bar{p}\bar{p}\ub{k}}
 - \frac{{\Rh}}{m}N_{k},
\end{align*}
as claimed.
\end{proof}

We also note for later that the covariant derivative of $N$ can be bounded in terms of quantities we have already defined.
\begin{proposition}
The tensor $\nabla N$ satisfies
\begin{equation}\label{eq:covdn}
 |\nabla N| \lesssim |\Rm| + |N|(|N|+ |A| + |T^0|) + |G|.
\end{equation}
\end{proposition}
\begin{proof}
 First, write
 \begin{equation}\label{eq:dern}
    \nabla_i N_j = \nabla_{\bi}N_{\bj} + \nabla_{\bi}N_{\uj}
    +\nabla_{\ui}N_{\bj}
    +\nabla_{\ui}N_{\uj}.
 \end{equation}
Since $N$ is horizontal, we have $N_j = H_{ja}N_a$. Differentiating
this expression implies
\[
\nabla_i N_j = \nabla_{i} N_{\uj} +  \nabla_iH_{ja} N_a,
\] 
and subtracting the first term from the right from both sides
yields that
\[
 \nabla_{i} N_{\bj} = -\frac{1}{m}|N|^2V_{ij} + N \ast \Ec^{\prime}.
\]
This shows both that $\nabla_{\ui}N_{\bj} = N \ast \Ec^{\prime}$
and that $\nabla_{\bi}N_{\bj} = N \ast N + N \ast \Ec^{\prime}$.
On the other hand, we have by definition that $\nabla_{\bi}N_{\uj} =
G_{ij}$.

Finally, for the last term in \eqref{eq:dern}, namely $(\nabla N)^H$, we have
from  Proposition \ref{prop:qprop} that
\begin{align*}
\nabla_{\ui}N_{\uj} &=  (R_{\ui\bp\bp\uj} - Q_{\ui\bp\bp\uj}) - \frac{N_i N_j}{m} = \Rm + N \ast N,
\end{align*}
and the claim follows.
\end{proof}

\begin{proposition}
\label{prop:t0ev}
The tensor $T^0$ satisfies the evolution equation
\begin{align}\label{eq:t0ev}
\begin{split}
 D_{\tau} T^0_{ijk} &=   \frac{{\Rh}}{m}(T^0_{\bar{\imath}\bar{\jmath}\ub{k}} - T^0_{\bar{\imath}\ub{\jmath}\bar{k}})
 -\left(P_{\uk\bj\bi} -P_{\uk\bj\bi}\right)
 + \left(P_{\uj\bk\bi} - P_{\bk\uj\bi}\right)\\
    &\phantom{=}+  M_{\bar{i}\bar{p}}\left(\left(T^0_{\bar{p}\bar{\jmath}\ub{k}}
    -  \frac{V_{pj}}{m} N_{\ub{k}}\right) -\left(T^0_{\bar{p}\ub{\jmath}\bar{k}} - \frac{V_{pk}}{m} N_{\ub{\jmath}}\right)\right)\\
    &\phantom{=} - \frac{M_{\bar{p}\bar{q}}}{m}\left( V_{ij}T^0_{\bar{q}\bar{p}\ub{k}} -V_{ik}T^0_{\bar{q}\bar{p}\ub{\jmath}}\right)
    - \left(M_{\bar{i}\ub{p}} A_{\ub{p}\bar{\jmath}\ub{k}} + \frac{V_{ij}}{m}M_{\bar{p}\ub{q}}A_{\ub{q}\bar{p}\ub{k}}\right)\\
&\phantom{=}
+ \left(M_{\bar{i}\ub{p}}A_{\ub{p}\ub{\jmath}\bar{k}} - \frac{V_{ik}}{m} M_{\bar{p}\ub{q}}A_{\ub{q}\bar{p}\ub{\jmath}}\right),
\end{split}
\end{align}
so that, in particular,
\begin{equation}\label{eq:t0evest}
  |D_{\tau} T^0| \lesssim  |\Rm|(|A| + |T^0|) + |N||M| + |P|.
\end{equation}
\end{proposition}
\begin{proof}
First, using \eqref{eq:lev}, we compute that
\begin{align*}
 D_{\tau} T_{ijk} &=  (D_{\tau}L)_{\bar{i}\bar{\jmath}\ub{k}}- (D_{\tau}L)_{\bar{i}\ub{\jmath}\bar{k}} = R_{\bar{i}p}( L_{p\bar{\jmath}\ub{k}}-L_{p\ub{\jmath}\bar{k}}) - E_{\bar{i}\ub{\jmath}\bar{k}}
    - E_{\bar{i}\bar{\jmath}\ub{k}}\\
    &= R_{\bar{i}\bar{p}}(L_{\bar{p}\bar{\jmath}\ub{k}}- L_{\bar{p}\ub{\jmath}\bar{k}}) +
  R_{\bar{i}\ub{p}}(L_{\ub{p}\bar{\jmath}\ub{k}}- L_{\ub{p}\ub{\jmath}\bar{k}})
 -  E_{\bar{i}\ub{\jmath}\bar{k}}
    - E_{\bar{i}\bar{\jmath}\ub{k}}\\
    &= R_{\bar{i}\bar{p}}(T^0_{\bar{p}\bar{\jmath}\ub{k}}- T^0_{\bar{p}\ub{\jmath}\bar{k}}) -
  R_{\bar{i}\ub{p}}(A_{\ub{p}\ub{\jmath}\bar{k}} + A_{\ub{p}\bar{\jmath}\ub{k}})  
 -   E_{\bar{i}\ub{\jmath}\bar{k}}
    - E_{\bar{i}\bar{\jmath}\ub{k}}\\
    &\phantom{=}+ \frac{R_{\bar{i}\bar{p}}}{m}\left(V_{pk} N_{\ub{\jmath}} - V_{pj} N_{\ub{k}}\right).
\end{align*}
Now,
\begin{align*}
   R_{\bar{i}\bar{p}}( T^0_{\bar{p}\bar{\jmath}\ub{k}} -T^0_{\bar{p}\ub{\jmath}\bar{k}})&=
   M_{\bar{i}\bar{p}}( T^0_{\bar{p}\bar{\jmath}\ub{k}}- T^0_{\bar{p}\ub{\jmath}\bar{k}}) + \frac{{\Rh}}{m}( T^0_{\bar{\imath}\bar{\jmath}\ub{k}}- T^0_{\bar{\imath}\ub{\jmath}\bar{k}}),
\end{align*}
and
\begin{align*}
\frac{R_{\bar{i}\bar{p}}}{m}\left(V_{pk} N_{\ub{\jmath}}-V_{pj} N_{\ub{k}}\right) &= \frac{M_{\bar{i}\bar{p}}}{m}\left(V_{pk} N_{\ub{\jmath}}- V_{pj} N_{\ub{k}}\right)
+ \frac{{\Rh}}{m^2}\left(V_{ik} N_{\ub{\jmath}} - V_{ij} N_{\ub{k}}\right),
 \end{align*}
so
\begin{align*}
 D_{\tau} T_{ijk}  &= M_{\bar{i}\bar{p}}(T^0_{\bar{p}\bar{\jmath}\ub{k}} - T^0_{\bar{p}\ub{\jmath}\bar{k}}) + \frac{{\Rh}}{m}(T^0_{\bar{\imath}\bar{\jmath}\ub{k}} - T^0_{\bar{\imath}\ub{\jmath}\bar{k}}) -
  M_{\bar{i}\ub{p}}(A_{\ub{p}\ub{\jmath}\bar{k}} + A_{\ub{p}\bar{\jmath}\ub{k}})  
 -   E_{\bar{i}\ub{\jmath}\bar{k}}
    - E_{\bar{i}\bar{\jmath}\ub{k}}\\
    &\phantom{=}+ \frac{M_{\bar{i}\bar{p}}}{m}\left(V_{pk} N_{\ub{\jmath}} -V_{pj} N_{\ub{k}}\right)
+ \frac{{\Rh}}{m^2}\left(V_{ik} N_{\ub{\jmath}} - V_{ij} N_{\ub{k}}\right).
\end{align*}
On the other hand, using Proposition \ref{prop:nev},
\begin{align*}
    D_{\tau}\left( V_{ij}N_{\ub{k}} - V_{ik}N_{\ub{\jmath}}\right)
    &= V_{ik}\left(M_{\bar{p}\bar{q}}T^0_{\bar{q}\bar{p}\ub{\jmath}} - M_{\bar{p}\ub{q}}A_{\ub{q}\bar{p}\ub{\jmath}} + E_{\bar{p}\bar{p}\ub{\jmath}}
 + \frac{{\Rh}}{m}N_{\ub{\jmath}}\right)\\
 &\phantom{=} - V_{ij}\left(M_{\bar{p}\bar{q}}T^0_{\bar{q}\bar{p}\ub{k}} - M_{\bar{p}\ub{q}}A_{\ub{q}\bar{p}\ub{k}} + E_{\bar{p}\bar{p}\ub{k}}
 + \frac{{\Rh}}{m}N_{\ub{k}})\right).
\end{align*}
Thus,
\begin{align*}
 D_{\tau}T_{ijk}^0 &= D_{\tau} T_{ijk} + \frac{1}{m}D_{\tau}\left(V_{ij}N_{\ub{k}} -V_{ik}N_{\ub{\jmath}}\right)\\
 &= M_{\bar{i}\bar{p}}(T^0_{\bar{p}\bar{\jmath}\ub{k}} -T^0_{\bar{p}\ub{\jmath}\bar{k}}) + \frac{{\Rh}}{m}(T^0_{\bar{\imath}\bar{\jmath}\ub{k}} -T^0_{\bar{\imath}\ub{\jmath}\bar{k}}) -
  M_{\bar{i}\ub{p}}(A_{\ub{p}\ub{\jmath}\bar{k}} + A_{\ub{p}\bar{\jmath}\ub{k}})  
 -   E_{\bar{i}\ub{\jmath}\bar{k}}
    - E_{\bar{i}\bar{\jmath}\ub{k}}\\
    &\phantom{=}+ \frac{M_{\bar{i}\bar{p}}}{m}\left(V_{pk} N_{\ub{\jmath}} - V_{pj} N_{\ub{k}}\right)
+ \frac{{\Rh}}{m^2}\left(V_{ik} N_{\ub{\jmath}} - V_{ij} N_{\ub{k}}\right)\\
&\phantom{=} + \frac{V_{ij}}{m}\left(M_{\bar{p}\bar{q}}T^0_{\bar{q}\bar{p}\ub{k}} - M_{\bar{p}\ub{q}}A_{\ub{q}\bar{p}\ub{k}} + E_{\bar{p}\bar{p}\ub{k}}
 + \frac{{\Rh}}{m}N_{\ub{k}}\right)\\
&\phantom{=}-\frac{V_{ik}}{m}\left(M_{\bar{p}\bar{q}}T^0_{\bar{q}\bar{p}\ub{\jmath}} - M_{\bar{p}\ub{q}}A_{\ub{q}\bar{p}\ub{\jmath}} + E_{\bar{p}\bar{p}\ub{\jmath}}
 + \frac{{\Rh}}{m}N_{\ub{\jmath}}\right)\\
  &= \frac{{\Rh}}{m}(T^0_{\bar{\imath}\bar{\jmath}\ub{k}} -T^0_{\bar{\imath}\ub{\jmath}\bar{k}})
 \left(E_{\bar{i}\bar{k}\ub{\jmath}} - \frac{V_{ik}}{m}E_{\bar{p}\bar{p}\ub{\jmath}}\right)
 -\left(E_{\bar{i}\bar{\jmath}\ub{k}} - \frac{V_{ij}}{m}E_{\bar{p}\bar{p}\ub{k}}\right)\\
    &\phantom{=}+  M_{\bar{i}\bar{p}}\left(T^0_{\bar{p}\bar{\jmath}\ub{k}} -T^0_{\bar{p}\ub{\jmath}\bar{k}}\right)
    - \frac{1}{m}\left(M_{\bi\bj}N_k - M_{\bi\bk}N_j\right)
    + \frac{M_{\bar{p}\bar{q}}}{m}\left( V_{ij}T^0_{\bar{q}\bar{p}\ub{k}} - V_{ik}T^0_{\bar{q}\bar{p}\ub{\jmath}}\right)\\
&\phantom{=}
+ \left(M_{\bar{i}\ub{p}}A_{\ub{p}\ub{\jmath}\bar{k}} + \frac{V_{ik}}{m} M_{\bar{p}\ub{q}}A_{\ub{q}\bar{p}\ub{\jmath}}\right)
- \left(M_{\bar{i}\ub{p}} A_{\ub{p}\bar{\jmath}\ub{k}} + \frac{V_{ij}}{m}M_{\bar{p}\ub{q}}A_{\ub{q}\bar{p}\ub{k}}\right),
  \end{align*}
and the identity follows by observing
that the third term on the right of the third identity can be written as
\begin{align*}
 E_{\bar{i}\bar{\jmath}\ub{k}} - \frac{V_{ij}}{m}E_{\bar{p}\bar{p}\ub{k}} &= \left(\nabla_{\uk}R_{\bj \bi} - \frac{V_{ij}}{m} \nabla_{\uk}R_{\bp\bp}\right) -\left(\nabla_{\bj}R_{\bi\uk}
 - \frac{V_{ij}}{m}\nabla_{\bp}R_{\bp\uk}\right)\\
 &= P_{\uk\bj\bi} - P_{\bj\bi\uk}
\end{align*}
with a similar identity for the second term on that line.
 \end{proof}

 We will not compute such detailed expressions for the evolutions of the second-order quantities $\nabla A$, $\nabla T^0$,
 and $G$.

 \begin{proposition}\label{prop:dadt0}
  The covariant derivatives of $A$ and $T^0$ satisfy the estimates
  \begin{align}\label{eq:daev}
  |D_{\tau} \nabla A| \lesssim |\nabla P| + \Theta (|A|+|T^0|
+ |\nabla A| + |\nabla T^0| + |M| + |P|)
 \end{align}
 and
 \begin{align}\label{eq:dtev}
  |D_{\tau} \nabla T^0| \lesssim |\nabla P|
  + \Theta (|A|+|T^0|
+ |\nabla A| + |\nabla T^0| + |M| + |P|),
 \end{align}
where 
\[
 \Theta = \Theta(|N|, |\Rm|, |\nabla \Rm|, |A|, |T^0|, |G|).
\]
 \end{proposition}
\begin{proof}
For \eqref{eq:daev}, we begin with the identity
\begin{align*}
    D_{\tau} \nabla_a A_{ijk} &= [D_{\tau}, \nabla_a] A_{ijk} + \nabla_a (D_{\tau} A)_{ijk}.
\end{align*}
Now,
\begin{align*}
 \begin{split}
   [D_{\tau}, \nabla_a] A_{ijk} &=  E_{aip} A_{pjk} + E_{ajp} A_{ipk} + E_{akp}A_{ijp} - R_{ap}\nabla_p A_{ijk}\\
   &= \nabla \Rc \ast A + \Rc \ast \nabla A
 \end{split}
\end{align*}
and
\begin{align*}
 \begin{split}
D_{\tau} A_{ijk} &=  P_{\uj\bk\ui} - P_{\bk\uj\ui} -
P_{\uk\bj\ui} + P_{\bj\uk\ui} + \Rc\ast A + (N + T^0)\ast M
\end{split}
\end{align*}
from \eqref{eq:aev}, so
\begin{align*}
  \nabla_a (D_{\tau} A_{ijk})&= (N + \Ec^{\prime})\ast D_{\tau} A
  + \nabla P + \nabla \Rc \ast A + \Rc \ast \nabla A 
  \\
  &\phantom{=}
  + (\nabla N + \nabla T^0) \ast M + (N\ast T^0)\ast \nabla M.
\end{align*}
Using \eqref{eq:dern}, we know
\[
 |\nabla N| \leq |\Rm| + |N|(|A| + |T^0| + |N|) + |G|,
\]
so, together with the expression \eqref{eq:qeq} for $Q$
and the definition  \eqref{eq:nablah} for $\Ec^{\prime}$,
we obtain \eqref{eq:daev}.

Similarly, for $\nabla T^0$, we start from the equation 
\[
  D_{\tau} \nabla_a T^0_{ijk} = [D_{\tau}, \nabla_a] T^0_{ijk} + \nabla_a (D_{\tau} T^0)_{ijk}.
\]
Then, on one hand,
\[
    [D_{\tau}, \nabla_a]T^0_{ijk} = T^0 \ast \nabla \Rc + \nabla T^0 \ast \Rc.
\]
On the other, note that from \eqref{eq:t0ev}, we have
\begin{align*}
 D_{\tau} T^0_{ijk} &=   (P_{\uj\bk\bi} - P_{\bk\uj\bi}) - (P_{\uk\bj\bi}
 - P_{\bj\uk\bi})
 + \Rc \ast T^0 + (N + A + T^0)\ast M,
\end{align*}
so that
\begin{align*}
  \nabla_a D_{\tau}T^0_{ijk} &= \nabla P + (N +\Ec^{\prime})\ast (P
  +\Rc\ast T^0 + (N + A + T^0)\ast M) + \nabla \Rc \ast T^0
  \\
  &\phantom{=} + \Rc\ast \nabla T^0+ \nabla N \ast M + M\ast \nabla A
  +(N+ A + T^0) \ast \nabla M,
\end{align*}
and \eqref{eq:dtev} follows by estimating $\nabla N$ and $\Ec^{\prime}$ as above.
\end{proof}
 
The time-derivative of $G$ admits a similar estimate, but the calculation
is a bit more delicate.
\begin{proposition}
\label{prop:gev}
The tensor $G$ satisfies the evolution equation
\begin{align}\label{eq:gev}
\begin{split}
 |D_{\tau} G| &\lesssim |\nabla P| + \Theta (|T^0| + |A|
 + |\nabla A| + |G| + |M| + |\nabla M| + |P|),
\end{split}
\end{align}
where
\begin{equation}\label{eq:gerrest}
 \Theta =  \Theta(|N|, |\Rm|, |\nabla \Rm|, |A|, |T^0|, |G|)
\end{equation}
is a polynomial as in Section \ref{ssec:notation}.
\end{proposition}
\begin{proof}
To begin with, we have
\begin{align*}
    D_{\tau} G_{ij} &= V_{ik}H_{jl}D_{\tau}\nabla_k N_l = V_{ik}H_{jl}([D_{\tau}, \nabla_k] N_l + \nabla_k D_{\tau} N_l) \\
    &= ([D_{\tau}, \nabla] N)_{\bi\uj} + (\nabla D_{\tau} N)_{\bi\uj}.
\end{align*}
Now,
\begin{align*}
    [D_{\tau}, \nabla_k] N_l &= E_{klp}N_p - R_{kp} \nabla_p N_l
\end{align*}
so
\begin{align}
\begin{split}\nonumber
  ([D_{\tau}, \nabla]N)_{\bi\uj} &=   E_{\bi \uj N } - R_{\bi p}\nabla_p N_{\uj}
  =  P_{N\uj\bi}-  P_{\uj N \bi}  - \frac{{\Rh}}{m} G_{ij} - M_{\bi p}\nabla_p N_{\uj}
\end{split}\\
\begin{split}\label{eq:nablan00}
  &= \Rc \ast G + \nabla N \ast M + \nabla P.
\end{split}
\end{align}
On the other hand, from \eqref{eq:nev}, we have
\begin{equation}\label{eq:nablan0}
 D_{\tau} N_l =   - E_{\bp\bp\ul} - \frac{{\Rh}}{m}N_l + (A + T^0)\ast M.
\end{equation}
Note that $E_{\bp\bp\ul} = V_{ab}H_{cl}E_{abc}$, so, using
\eqref{eq:nablah}, we see that
\begin{align}
\begin{split}\nonumber
 \nabla_{\bi}(E_{\bp\bp\ul})
  &= \nabla_{\bi}E_{\bp\bp\ul} + \frac{1}{m}(V_{ia}N_b + V_{ib}N_a)E_{ab\ul}
 \\
  &\phantom{=}
   - \frac{1}{m}(V_{ic}N_l + V_{il}N_c)E_{\bp\bp c} + \nabla \Rc \ast \Ec^{\prime}\\
   &= \nabla_{\bi}E_{\bp\bp\ul} + \frac{1}{m}(
   E_{\bi N \ul} + E_{N\bi\ul} - E_{\bp\bp\bi}N_l -E_{\bp\bp N}V_{il})
   + \nabla \Rc \ast \Ec^{\prime}
\end{split}\\
\begin{split}\label{eq:nablan1}
   &= \nabla_{\bi}E_{\bp\bp\ul} -\frac{1}{m}E_{\bp\bp N}V_{il} + N \ast P +\nabla \Rc \ast \Ec^{\prime}.
\end{split}
\end{align}
Consider the first term. Using the contracted second Bianchi identity,  we can write
\begin{align}
\begin{split}\nonumber
 \nabla_{\bi}E_{\bp\bp\ul}&= \nabla_{\bi}\nabla_{\ul}R_{\bp\bp}
 - \nabla_{\bi}\nabla_{\bp}R_{\ul\bp}\\
\end{split}\\
\begin{split}\label{eq:nablan2}
 &= \frac{1}{2}\nabla_{\bi}\nabla_{\ul}R  
 -\nabla_{\bi}\nabla_{\ul}R_{\up\up}
 + \nabla_{\bi}\nabla_{\up}R_{\ul\up}.
\end{split}
\end{align}
Now, for the first term in \eqref{eq:nablan2}, we have
\begin{align*}
\nabla_{\bi}\nabla_{\ul} R &= \nabla_{\ul}\nabla_{\bi} R
= \nabla_{\ul}(\nabla_{\bi} R) + \nabla R \ast \Ec^{\prime}
= \nabla_{\ul}S_i  + \nabla R \ast \Ec^{\prime}\\
&= \nabla P  +  \Ec^{\prime}\ast P + \nabla R \ast \Ec^{\prime},
\end{align*}
in view of \eqref{eq:schar}. Here we have also used
that $\nabla_{\ui} H_{jk} = \Ec^{\prime}_{\ui jk}$.
The second and third terms in \eqref{eq:nablan2} are both traces of terms of the form $\nabla_{\bi}\nabla_{\uj}R_{\uk\ul}$.
To estimate them, we note that
\begin{align*}
 \nabla_{\bi}\nabla_{\uj}R_{\uk\ul} &=\nabla_{\uj}\nabla_{\bi}
 R_{\uk\ul} - R_{\bi\uj\uk p} R_{p\ul} - R_{\bi\uj\ul p} R_{\uk p} \\
 &= \nabla_{\uj}\nabla_{\bi}R_{\uk \ul} -
 R_{\bi\uj\uk \bp} R_{\bp\ul} - R_{\bi\uj\ul \bp} R_{\uk \bp} - R_{\bi\uj\uk \up} R_{\up\ul} - R_{\bi\uj\ul \up} R_{\uk \up}\\
&= \nabla_{\uj}P_{\bi\uk \ul} + \nabla \Rc \ast \Ec^{\prime}
 + \Rm \ast M + \Rc\ast Q.
\end{align*}
Thus we conclude at last that
\[
\nabla_{\bi}(E_{\bp\bp\ul}) = -\frac{1}{m}E_{\bp\bp N}V_{il} 
+ \nabla P 
+ (N + \Ec^{\prime})\ast P +\nabla \Rc \ast \Ec^{\prime}
+ \Rm \ast M + \Rm \ast Q.
\]
It follows from \eqref{eq:nablan0}, then, that
\begin{align}
 \begin{split}\label{eq:nablan3}
  (\nabla D_{\tau} N_l)_{\bi\uj} &= -\frac{1}{m}(\nabla_{\bi}\Rh N_j
  + \Rh G_{ij}) +\nabla A\ast M  + A\ast \nabla M + \nabla T^0 \ast M\\
  &\phantom{=} + T^0\ast \nabla M + (N + \Ec) \ast (A + T^0) \ast M.
\end{split}
\end{align}
Since
\[
 \nabla_{\bi}\Rh = S_i -\nabla_{\bi}R_{\up\up} = \nabla P
 + (N + \Ec^{\prime}) \ast P,
\]
combining \eqref{eq:nablan3} with \eqref{eq:nablan00}
and using \eqref{eq:qeq} and \eqref{eq:dern} to substitute for $Q$ and $\nabla N$ yields
\eqref{eq:gev}.
\end{proof}

\section{Commutation identities for projection operators}
In order to compute the evolution equations for the curvature invariants $M$, $P$, and $U$, we
will need to understand how certain projection operators defined in terms of the projections $H$ and $V$ interact with the Laplacian and covariant derivative
associated to $g$.

Since we have defined $H$ and $V$ in order to ensure $D_{\tau} H \equiv 0$ and $D_{\tau} V \equiv 0$,
we will automatically have $D_{\tau} \Pc \equiv 0$, and $D_{\tau} \Uc \equiv 0$, and similarly for the other such projection operators we will consider.
However, we will not in general have $\nabla H \equiv 0$ and $\nabla V\equiv 0$ even on the model space at $\tau =0$ (unless it splits metrically as product), and we cannot therefore expect $\nabla^{(k)} H$ and $\nabla^{(k)} V$ even to be \emph{approximately} vanishing in our computations. Instead, each covariant derivative of $H$ and $V$ which arises in our calculations will produce a non-trivial correction term (specified by Lemma \ref{lem:hder}) which must be \emph{cancelled} or otherwise controlled rather than immediately estimated away. This significantly increases the complexity of our calculations.

\subsection{The horizontal projection operator}\label{sec:hproj}
First we consider the horizontal projection operator $\Hc:T^{k+l}(T^*M)\longrightarrow T^{k+l}(T^*M)$
taking $X$ to $\Hc(X) \dfn X^H$, that is,
\[
  \Hc(X)_{a_1 \cdots a_k} \dfn  X^H_{a_1\cdots a_k} = X_{\ua_1 \cdots \ua_k}.
\]
where we regard the section $X\in T^{k+l}(T^*M)$ as a $T^l(T^*M)$-valued $k$-tensor.
\begin{proposition}\label{prop:horizlap} Let $k\geq 1$
and $l\geq 0$ and let $X$ be a smooth section of $T^{k+1}(T^*M) \approx T^{k}(T^*M)\otimes T^l(T^*M)$, regarded as a $T^l(T^*M)$-valued $k$-multilinear map $X = X_{a_1 \cdots a_k}$. Then
 \begin{align}\label{eq:horizgrad}
 \begin{split}
    \nabla_{s}X^H_{a_1 \cdots a_k} &=(\nabla_sX)^H_{a_1\cdots a_k} + X\ast \Ec^{\prime}\\
    &\phantom{=} -\frac{1}{m}\sum_{i=1}^k\left(X_{\ua_1 \cdots N \cdots \ua_k}V_{sa_i} +
    X_{\ua_1 \cdots \bar{s} \cdots \ua_k}N_{a_i}\right),
    \end{split}
 \end{align}
and
\begin{align}\label{eq:horizlap}
 \begin{split}
   \begin{split}
     (\Delta X^H)_{a_1 \cdots a_k} &
      = (\Delta X)^H_{a_1 \cdots a_k} + (\nabla X + (\Ec^{\prime} + N) \ast X)\ast \Ec^{\prime}
     + X\ast \Ec^{\prime\prime}\\
      &\phantom{=}
      + \frac{2}{m^2}\sum_{i < j}X_{\ua_1 \cdots N\cdots N \cdots\ua_k}
      V_{a_i a_j}\\
      &\phantom{=} + \frac{2}{m}\sum_{i=1}^k\left(\frac{|N|^2}{m} X_{\ua_1 \cdots \ba_i \cdots \ua_k}
        -\nabla_{\ba_i}X_{\ua_1\cdots N \cdots \ua_k}\right)\\
      &\phantom{=} + \frac{2}{m^2}\sum_{i < j}\left(X_{\ua_1\cdots N \cdots \ba_i \cdots \ua_k}N_{a_j}
      + X_{\ua_1 \cdots \ba_j \cdots N \cdots \ua_k}N_{a_i}\right)\\
      &\phantom{=} - \frac{2}{m}\sum_{i=1}^k\left(\nabla_{\bp}X_{\ua_1\cdots \bp\cdots \ua_k}
      + X_{\ua_1 \cdots N \cdots \ua_k}\right)N_{a_i}\\
      &\phantom{=} + \frac{2}{m^2}\sum_{i < j} X_{\ua_1 \cdots \bp \cdots \bp\cdots \ua_k} N_{a_i}N_{a_j}.
    \end{split}
 \end{split}
 \end{align}
Alternatively,
  \begin{align}\label{eq:horizlap2}
    \begin{split}
     &\Delta X^H_{a_1 \cdots a_k}
      = (\Delta X)^H_{a_1 \cdots a_k} + (\nabla X + (\Ec^{\prime} + N) \ast X)\ast \Ec^{\prime}
     + X\ast \Ec^{\prime\prime}\\
       &\phantom{=}  +\frac{2}{m}\sum_{i=1}^k\left(\frac{|N|^2}{m}X_{\ua_1 \cdots \ba_i \cdots \ua_i} - \nabla_{\ba_i}X_{\ua_1\cdots N \cdots \ua_k} - \nabla_{\bp}(X_{\ua_1 \cdots \bp \cdots \ua_k})N_{a_i}\right)\\
      &\phantom{=}
      + \frac{2}{m^2}\sum_{i < j}\left(X_{\ua_1 \cdots N\cdots N \cdots \ua_k}V_{a_ia_j}
      -X_{a_1 \cdots \bp\cdots\bp \cdots a_k}N_{a_i}N_{a_j} \right).
    \end{split}
 \end{align}

\end{proposition}

\begin{proof}
Fix a smooth section $X$ of $T^{k+l}(T^*M) \approx T^{k}(T^*M)\otimes T^{l}(T^*M)$,
and write $X = X_{a_1 \cdots a_k}$. The horizontal projection $X^H$ of $X$ is described by
\[
 X^H_{a_1 \cdots a_k } = X_{\ua_1 \cdots \ua_k } = X_{b_1 \cdots b_k } H_{a_1b_1} \cdots H_{a_k b_k}.
\]
where there is an implied sum on the indices $b_i$.
For \eqref{eq:horizgrad}, we compute directly, using \eqref{eq:eprimedef}, that
\begin{align*}
\nabla_s X^H_{a_1\cdots a_k } &= \nabla_s\left(X_{b_1b_2\ldots b_k}H_{a_1b_1}\cdots H_{a_kb_k}\right)\\
&= \nabla_s X_{\ua_1\cdots \ua_k} +  \sum_{i=1}^k \nabla_s H_{a_i b_i}X_{\ua_1 \cdots b_i \cdots \ua_k}\\
\begin{split}
&= \nabla_s X_{\ua_1\cdots \ua_k} + \sum_{i=1}^k\left(\Ec^{\prime}_{sa_ib_i} - \frac{1}{m}\left(V_{sa_i}N_{b_i} + V_{s{b_i}}N_{a_i}\right)\right)X_{\ua_1 \cdots b_i\cdots \ua_k}
\end{split}\\
\begin{split}
&= \nabla_s X_{\ua_1\cdots \ua_k} + X\ast \Ec^{\prime} - \frac{1}{m}\sum_{i=1}^k\left(X_{\ua_1 \cdots N \cdots \ua_k }V_{sa_i}+ X_{\ua_1 \cdots \bs \cdots \ua_k}N_{a_i}\right)
\end{split}
\end{align*}
as claimed.

For \eqref{eq:horizlap},  we start with the identity
\begin{align}\label{eq:hl1}
\begin{split}
\Delta (X^H)_{a_1 \cdots a_k} &= \Delta X_{\ua_1 \cdots \ua_k }
+ 2\sum_{i < j} \nabla_s H_{a_i b_i} \nabla_s H_{a_j b_j} X_{\ua_1 \cdots b_i \cdots b_j \cdots \ua_k }
\\
&\phantom{=}
+ \sum_{i=1}^k\left(  \Delta H_{a_i b_i} X_{\ua_1 \cdots b_i\cdots \ua_k }
+ 2\nabla_s H_{a_i b_i}\nabla_s X_{\ua_1\cdots b_i\cdots \ua_k }\right).
\end{split}
\end{align}
Now, for each $i$, we have
\begin{align}\label{eq:hl2}
\begin{split}
  \Delta H_{a_i b_i} X_{\ua_1 \cdots b_i\cdots \ua_k } &= \left(\Ec_{a_ib_i}^2 + \frac{2}{m}\left(\frac{|N|^2}{m}V_{a_i b_i} - N_{a_i}N_{b_i}\right)\right)X_{\ua_1\cdots b_i \cdots \ua_k }\\
  &= X\ast \Ec^{\prime\prime} + \frac{2}{m^2}|N|^2X_{\ua_1 \cdots \ba_i \cdots \ua_k}
  -\frac{2}{m} X_{\ua_1 \cdots N\cdots \ua_k}N_{a_i},
\end{split}
\end{align}
and
\begin{align}\label{eq:hl3}
\begin{split}
& \nabla_s H_{a_i b_i}\nabla_s X_{\ua_1\cdots b_i\cdots \ua_k }\\
&\quad = \left(\Ec^{\prime}_{s a_i b_i} - \frac{1}{m}\left(V_{sa_i} N_{b_i} + V_{sb_i}N_{a_i}\right)\right)\nabla_{s}X_{\ua_1 \cdots b_i \cdots \ua_k}\\
 &\quad= \nabla X \ast \Ec^{\prime} -\frac{1}{m}\left(\nabla_{\ba_i}X_{\ua_1 \cdots N \cdots \ua_k }
 + \nabla_{\bp}X_{\ua_1 \cdots \bp \cdots \ua_k}N_{a_i}\right),
 \end{split}
\end{align}
while, for any $1\leq i < j\leq k$, we have
\begin{align*}
&  \nabla_s H_{a_i b_i} \nabla_s H_{a_j b_j}\\
&\qquad= \left(\Ec^{\prime}_{sa_i b_i}
  - \frac{1}{m}\left(V_{s a_i}N_{b_i} + V_{sb_i} N_{a_i}\right)\right)\left(\Ec^{\prime}_{sa_j b_j}
  -\frac{1}{m}\left(V_{sa_j}N_{b_j}+ V_{sb_j}N_{a_j}\right)\right)\\
&\qquad=  + \frac{1}{m^2}\left(V_{a_i a_j} N_{b_i} N_{b_j} + V_{a_i b_j} N_{b_i} N_{a_j}
+ V_{b_i a_j} N_{a_i} N_{b_j} + V_{b_i b_j} N_{a_i} N_{a_j}\right) \\
& \qquad\;\phantom{=} +(\Ec^{\prime} + N)\ast \Ec^{\prime},
\end{align*}
so that
\begin{align}\label{eq:hl4}
\begin{split}
 &\nabla_s H_{a_i b_i} \nabla_s H_{a_j b_j} X_{\ua_1 \cdots b_i \cdots b_j \cdots \ua_k} = (\Ec^{\prime} + N) \ast X\ast \Ec^{\prime}\\
 &\quad= \frac{1}{m^2}\bigg(X_{\ua_1 \cdots N \cdots N \cdots \ua_k} V_{a_ia_j}
 + X_{\ua_1 \cdots \ba_j \cdots N \cdots \ua_k} N_{a_i} + X_{\ua_1 \cdots N \cdots \ba_i \cdots \ua_k}N_{a_j}\\
 &\quad\phantom{=}
 + X_{\ua_1 \cdots \bp \cdots \bp \cdots \ua_k} N_{a_i} N_{a_j}\bigg).
\end{split}
 \end{align}
Incorporating the identities \eqref{eq:hl2}, \eqref{eq:hl3}, and \eqref{eq:hl4} into \eqref{eq:hl1},
we arrive at \eqref{eq:horizlap}.

For the alternative expression in \eqref{eq:horizlap2}, note that, for any fixed $i$,
\begin{align*}
& \nabla_{\bb_i}(X_{\ua_1 \cdots \bb_i \cdots \ua_k}) = \nabla_{\bb_i}X_{\ua_1 \cdots \bb_i \cdots \ua_k}
  +(\nabla_{\bb_i}V_{b_i p}) X_{\ua_1 \cdots p \cdots \ua_k}\\
&\qquad\phantom{=} +\sum_{j< i}(\nabla_{\bb_i}H_{a_jb_j})
 X_{\ua_1 \cdots b_j \cdots b_i \cdots \ua_k}
 + \sum_{j > i}(\nabla_{\bb_i}H_{a_jb_j})X_{\ua_1 \cdots b_i
 \cdots b_j \cdots \ua_k}\\
 &\qquad=\nabla_{\bb_i}X_{\ua_1 \cdots \bb_i \cdots \ua_k} + X\ast \Ec^{\prime}
  +\frac{1}{m}(V_{b_ib_i}N_p + N_{\bp})X_{\ua_i \cdots p \cdots \ua_k}\\
&\qquad\phantom{=}
 -\frac{1}{m}\sum_{j< i}(V_{b_i a_j}N_{b_j} + V_{b_i b_j}N_{a_j}) X_{\ua_1 \ldots b_j \cdots b_i \cdots \ua_k}\\
 &\qquad\phantom{=}
 - \frac{1}{m}\sum_{j > i}(V_{b_i a_j}N_{b_j} + V_{b_ib_j}N_{a_j})X_{\ua_1 \cdots b_i
 \cdots b_j \cdots \ua_k}\\
 &\qquad=\nabla_{\bb_i}X_{\ua_1 \cdots \bb_i \cdots \ua_k} + X\ast \Ec^{\prime}
  +X_{\ua_i \cdots N \cdots \ua_k}\\
&\qquad\phantom{=}
 -\frac{1}{m}\sum_{j< i}(X_{\ua_1 \cdots N \cdots \ba_j \cdots \ua_k}
 + X_{\ua_1 \cdots \bp \cdots \bp \cdots \ua_k}N_{a_j})\\
 &\qquad\phantom{=}
 - \frac{1}{m}\sum_{j > i}(X_{\ua_1 \cdots \ba_j
 \cdots N \cdots \ua_k} + X_{\ua_1 \cdots \bp
 \cdots \bp \cdots \ua_k}N_{a_j}).
\end{align*}
Consequently,
\begin{align*}
\sum_{i=1}^k\nabla_{\bp}(X_{\ua_1 \cdots \bp \cdots \ua_k}) N_{a_i}
 &= \sum_{i=1}^k\left(\nabla_{\bp}X_{\ua_1 \cdots \bp \cdots \ua_k}N_{a_i} + X_{\ua_1 \cdots N \cdots \ua_k} N_{a_i}\right)\\
&\phantom{=}
 -\frac{1}{m}\sum_{i < j}\left(X_{\ua_1 \cdots N \cdots \ba_i \cdots \ua_k}N_{a_j} + X_{\ua_1 \cdots \ba_j
 \cdots N \cdots \ua_k}N_{a_i}\right)\\
 &\phantom{=} -\frac{2}{m}\sum_{i < j} X_{\ua_1 \cdots \bp \cdots \bp \cdots \ua_k} N_{a_i}N_{a_j} + N \ast X\ast \Ec^{\prime},
\end{align*}
as claimed.
\end{proof}

\subsection{The vertical trace operator}\label{sec:vtproj}

Next, we consider the operator
\[
\trace_V: T^{k+2}(T^*M)\longrightarrow T^k(T^*M),
\]
defined by
\[
  (\trace_{V}X) \dfn V_{pq}X_{pq} = X_{\bp\bp},
\]
and the associated projection
\[
 \Tc: T^{k+2}(T^*M)\longrightarrow T^{k+2}(T^*M)
\]
defined by
\[
  \Tc(X)_{ij} \dfn \trace_V(X)V_{ij} = X_{\bp\bp}V_{ij},
\]
where we regard $X$ as a bilinear $T^k(T^*M)$-valued map.

\begin{proposition}\label{prop:tracelap}
For any smooth section $X$ of $T^{k+2}(T^*M) \approx T^{2}(T^*M)\otimes T^{k}(T^*M)$,
we have
\begin{align}\label{eq:tracegrad}
  \nabla_s (\trace_V X) = \nabla_s X_{\bp\bp} +\frac{1}{m}(X_{\bs N} + X_{N \bs}) +X\ast \Ec^{\prime},
\end{align}
and
\begin{align}
\begin{split}
\label{eq:tracelap}
 \Delta \trace_V(X) &= \trace_V(\Delta X)
 + \frac{2}{m}\left(\nabla_{\bp}(X_{\bp \uq}) + \nabla_{\bp}(X_{\uq\bp})\right)N_{q}\\
 &\phantom{=}+ \frac{2}{m}\left(\frac{|N|^2}{m}X_{\bp\bp} - X_{NN}\right) + (\nabla X + N \ast X)\ast \Ec^{\prime}  + X \ast \Ec^{\prime\prime}.
 \end{split}
 \end{align}
Moreover,
\begin{align}
\begin{split}
\label{eq:traceprojlap}
\Delta \Tc(X)_{ij} &= \Tc(\Delta X)_{ij}
 + \frac{2}{m}\left(\nabla_{\bp}(X_{\bp \uq}) + \nabla_{\bp}(X_{\uq\bp})\right)N_{q}V_{ij}\\
 &\phantom{=}+\frac{2}{m}\left(\nabla_{\bi}X_{\bp\bp}N_j + \nabla_{\bj}X_{\bp\bp}N_i\right) + \frac{2}{m}\left((\trace_V X)N_iN_j - X_{NN}V_{ij}\right)\\
 &\phantom{=} + \frac{2}{m^2}\left(X_{\bi N}N_j + X_{N\bj}N_i\right) + (\nabla X + N \ast X) \ast \Ec^{\prime} + X \ast \Ec^{\prime\prime}.
\end{split}
\end{align}
\end{proposition}
\begin{proof}
For \eqref{eq:tracegrad}, from the expression $\trace_V(X) = V_{pq}X_{pq}$,
we compute that
\begin{align*}
 \nabla_s (\trace_V X) &= \nabla_s V_{pq} X_{pq} + V_{pq}\nabla_s X_{pq}\\
 &= X\ast \Ec^{\prime} + \frac{1}{m}(V_{sp} N_q + V_{sq} N_p)X_{pq} + \nabla_s X_{\bp\bp}\\
&= X\ast \Ec^{\prime} + \frac{1}{m}(X_{\bs N} + X_{N \bs}) + \nabla_s X_{\bp\bp},
\end{align*}
using \eqref{eq:eprimedef}.

For \eqref{eq:tracelap}, we start from the equation
\begin{equation}\label{eq:hlap0}
\Delta(\trace_{V}X) = \Delta X_{\bp\bp}   + 2\nabla_s V_{ab} \nabla_s X_{ab} + \Delta V_{ab}X_{ab}.
\end{equation}
Now,
\begin{align*}
\nabla_s V_{ab} \nabla_s X_{ab}
&= \frac{1}{m}(V_{sa}N_b + V_{sb} N_a)\nabla_s X_{ab} + \nabla X \ast \Ec^{\prime}\\
&= \frac{1}{m}(\nabla_{\bp}X_{\bp N} + \nabla_{\bp}X_{N \bp})
+ \nabla X \ast \Ec^{\prime}.
\end{align*}
Since $V_{pq} = V_{pr}V_{rq}$ and $N_a = H_{ab}N_b$,
\begin{align*}
  \nabla_{\bp}X_{\bp N} &= V_{pq}N_a\nabla_p X_{qa} =V_{pr}V_{rq}H_{ab}N_b\nabla_{p}X_{qa}\\
  &= V_{pr}N_b\left(\nabla_{p}(X_{\bar{r}\ub{b}}) -\nabla_p V_{rq} X_{q\ub{b}}
  -\nabla_p H_{ab} X_{\bar{r} a}\right)\\
\begin{split}
  &=V_{pr}N_b\left(\nabla_{p}(X_{\bar{r}\ub{b}})
  -\frac{1}{m}(V_{pr}N_q +V_{pq}N_r)X_{q\ub{b}}
  + \frac{1}{m}(V_{pa}N_b + V_{pb}N_a)X_{\bar{r} a}\right)\\
  &\phantom{=}+ N \ast X \ast \Ec^{\prime}
\end{split}\\
&= \nabla_{\bar{r}}(X_{\bar{r}\ub{b}})N_b -X_{NN} +
\frac{1}{m}|N|^2X_{\bp\bp} + N\ast X \ast \Ec^{\prime},
\end{align*}
so
\begin{align}
\begin{split}\label{eq:tracelap1}
 \nabla_s V_{ab} \nabla_s X_{ab} &=
 \nabla_{\bp}(X_{\bp \uq} + X_{\uq \bp})N_q
 - 2X_{NN} + \frac{2}{m}|N|^2 X_{\bp\bp}\\
 &\phantom{=} + \nabla X \ast \Ec^{\prime} + N \ast X\ast \Ec^{\prime}.
\end{split}
\end{align}
Also, by \eqref{eq:deltah},
\begin{equation}\label{eq:hlap2}
\Delta V_{ab}X_{ab} = X \ast \Ec^{\prime\prime} + \frac{2}{m}\left(X_{NN} - \frac{|N|^2}{m}X_{\bp\bp}\right).
\end{equation}
Combining \eqref{eq:tracelap1} with \eqref{eq:hlap2} in \eqref{eq:hlap0}
and cancelling terms gives \eqref{eq:tracelap}.

In our proof of \eqref{eq:traceprojlap}, we will regard $X\otimes V$
and $\Tc(X) = \trace_V(X)\otimes V$ as quadrilinear and bilinear maps, respectively, taking values in $T^{k}(T^*M)$,
and write simply $(\trace_V(X)\otimes V)_{ij} = \trace_V(X)V_{ij}$.
We begin with the expression
\begin{align}\label{eq:deltv1}
\begin{split}
    \Delta \Tc(X)_{ij} &= \Delta((\trace_V X)V)_{ij}\\
    &= (\Delta \trace_V X) V_{ij} + 2\nabla_s(\trace_V X) \nabla_s V_{ij}
    + (\trace_{V} X)\Delta V_{ij}.
\end{split}
\end{align}
On one hand, from \eqref{eq:horizgrad} and \eqref{eq:tracegrad}
we have
\begin{align*}
 \nabla_s(\trace_V X) \nabla_s V_{ij} &=
 \frac{1}{m}(\nabla_{\bi}(\trace_V X)N_j +\nabla_{\bj}(\trace_V X)N_i) + \nabla X \ast \Ec^{\prime}\\
 &= \frac{1}{m}(\nabla_{\bi}X_{\bp\bp}N_j + \nabla_{\bj}X_{\bp\bp}N_i)+ \frac{1}{m^2}(X_{\bi N}N_j + X_{N\bj}N_i)\\
 &\phantom{=} + N \ast X \ast \Ec^{\prime} + \nabla X \ast \Ec^{\prime},
 \end{align*}
 and, on the other, from \eqref{eq:deltah}, that
\begin{align*}
    (\trace_V X) \Delta V_{ij} &=
    \frac{2}{m}\left(N_iN_j - \frac{|N|^2}{m}V_{ij}\right)(\trace_V X) + X \ast \Ec^{\prime\prime}.
 \end{align*}
Using these expressions and \eqref{eq:tracelap} in \eqref{eq:deltv1},
we have
\begin{align*}
\Delta \Tc(X)_{ij} &= \Tc(\Delta X)_{ij}
 + \frac{2}{m}\left(\nabla_{\bp}(X_{\bp \uq}) + \nabla_{\bp}(X_{\uq\bp})\right)N_{q}V_{ij}\\
 &\phantom{=}+\frac{2}{m}\left(\nabla_{\bi}X_{\bp\bp}N_j + \nabla_{\bj}X_{\bp\bp}N_i\right) + \frac{2}{m}\left((\trace_V X)N_iN_j - X_{NN}V_{ij}\right)\\
 &\phantom{=} + \frac{2}{m^2}\left(X_{\bi N}N_j + X_{N\bj}N_i\right) + (\nabla X + N \ast X) \ast \Ec^{\prime} + X \ast \Ec^{\prime\prime},x
\end{align*}
which is \eqref{eq:traceprojlap}.
\end{proof}
\subsection{The projection $\Pc$}\label{sec:pproj}
Now we consider the projection operator
\[
\Pc: T^{k+3}(T^*M)\longrightarrow T^{k+3}(T^*M)
\]
defined
by
\[
 \Pc(X)_{ijk} = X_{ijk} - X^H_{ijk}
 - \frac{1}{m}
 \left(X_{\ui\bp\bp}V_{jk} + X_{\bp\uj\bp}V_{ik}
 + X_{\bp\bp\uk}V_{ij}\right),
\]
where $X$ is regarded as a trilinear $T^k(T^*M)$-valued map. For $k=0$, we have $\Pc(\nabla \Rc) = P$ as defined in \eqref{eq:pdef}. When $k=2$, we have $\Pc(\Hc(\nabla \Rm)) = U$ as defined in \eqref{eq:udef}, where $\Hc$ acts on the third and fourth components and $\Pc$ acts on the first, second, and fifth components.

We first derive an expression for the commutator of $\Pc$ with the covariant derivative.
\begin{proposition}\label{prop:pgradient}
For any smooth section $X$ of $T^{k+3}(T^*M) \approx T^3(T^*M)\otimes T^k(T^*M)$, $\Pc = \Pc(X)$ satisfies
\begin{align}\label{eq:pgradient}
\begin{split}
  &\nabla_s \Pc_{ijk} = \Pc(\nabla_s X)_{ijk} + X\ast \Ec^{\prime} + \Cc_{sijk} \\
  &\phantom{=}
  +\frac{1}{m}\left(X_{N\uj\uk}V_{si} + X_{\ui N \uk}V_{sj} + X_{\ui\uj N}V_{sk}\right)\\
     &\phantom{=}   +\frac{1}{m^2}\left(X_{N\bp\bp}V_{si}V_{jk}-X_{\ui\bar{p}\bar{p}}
     \left(V_{sj}N_k + V_{sk}N_j\right)\right)  \\
     &\phantom{=}+\frac{1}{m^2}\left(X_{\bp N \bp}V_{sj}V_{ik} - X_{\bar{p}\uj\bar{p}}\left(V_{si}N_k + V_{sk}N_i\right)\right)
     \\
     &\phantom{=} +\frac{1}{m^2}\left(X_{\bp\bp N}V_{sk}V_{ij}-X_{\bar{p}\bar{p}\uk}
     \left(V_{sj}N_i + V_{si}N_j\right) \right),
\end{split}
\end{align}
where $\Cc = \Cc(X)$ satisfies
\begin{align*}
\Cc_{sijk} &= \frac{1}{m}\left(\Pc_{\bs \uj\uk} N_i + \Pc_{\ui \bs \uk} N_j
     + \Pc_{\ui\uj\bs} N_k\right) +\frac{1}{m^2}\left(\Pc_{\bs\bp\bp} N_i - \Pc_{\ui \bar{s} N} - \Pc_{\ui N \bar{s}}\right)V_{jk}\\
     &\phantom{=}+ \frac{1}{m^2}\left(\Pc_{\bp\bs\bp} N_j-\Pc_{\bar{s}\uj N} - \Pc_{N\uj\bar{s}}\right)V_{ik} +\frac{1}{m^2}\left(\Pc_{\bp\bp\bs} N_k - \Pc_{N\bar{s} \uk} - \Pc_{\bar{s}N \uk}\right)V_{ij},
\end{align*}
that is, $\Cc = N \ast N \ast \Pc$.
\end{proposition}
\begin{proof}
 First, we may apply \eqref{eq:horizgrad} with \eqref{eq:eprimedef} to obtain that
 \begin{align*}
    \begin{split}
    \nabla_{s}X^H_{ijk} &= \nabla_sX_{\ui\uj\uk} + X\ast \Ec^{\prime} -\frac{1}{m}\left(X_{N\uj\uk}V_{si} + X_{\ui N \uk}V_{sj} + X_{\ui\uj N}V_{sk}\right)\\
    &\phantom{=}
    -\frac{1}{m}\left(X_{\bs\uj\uk}N_i +X_{\ui\bs\uk}N_j + X_{\ui\uj\bs}N_k\right)
    \end{split}\\
    \begin{split}
     &= \nabla_sX_{\ui\uj\uk} + X\ast \Ec^{\prime}
     -\frac{1}{m}\left(\Pc_{\bs \uj\uk} N_i + \Pc_{\ui \bs \uk} N_j
     + \Pc_{\ui\uj\bs} N_k\right)\\
     &\phantom{=}
     -\frac{1}{m}\left(X_{N\uj\uk}V_{si} + X_{\ui N \uk}V_{sj} + X_{\ui\uj N}V_{sk}\right).
    \end{split}
 \end{align*}

To compute the other terms, let $\tilde{X}_{ijk} \dfn X_{\ui jk}$.
Then, in the notation of the previous section, we have
\[
    X_{\ui \bp\bp}V_{jk} = \Tc(\tilde{X})_{ijk},
\]
where $\Tc$ acts on the second and third components,
so that, by \eqref{eq:tracegrad},
\begin{align*}
 \begin{split}
   \nabla_s(X_{\ui\bp\bp}V_{jk}) &= \nabla_s \tilde{X}_{i\bp\bp}V_{jk} + X\ast \Ec^{\prime}+\frac{1}{m}\left(X_{\ui \bar{s} N} + X_{\ui N \bar{s}}\right)V_{jk} \\
   &\phantom{=} +
   \frac{1}{m}X_{\ui\bar{p}\bar{p}}\left(V_{sj}N_k + V_{sk}N_j\right)
   \end{split}\\
   \begin{split}
    &= \nabla_s X_{\ui\bp\bp}V_{jk} + X\ast \Ec^{\prime} -\frac{1}{m}\left(X_{\bs\bp\bp} N_i + X_{N\bp\bp}V_{si}\right)V_{jk}\\
   &\phantom{=} +\frac{1}{m}\left(X_{\ui \bar{s} N} + X_{\ui N \bar{s}}\right)V_{jk} +
   \frac{1}{m}X_{\ui\bar{p}\bar{p}}\left(V_{sj}N_k + V_{sk}N_j\right)
   \end{split}\\
     \begin{split}
     &= \nabla_s X_{\ui\bp\bp}V_{jk} + X\ast \Ec^{\prime}  +\frac{1}{m}\left(\Pc_{\ui \bar{s} N} + \Pc_{\ui N \bar{s}}-\Pc_{\bs\bp\bp} N_i\right)V_{jk}\\
     &\phantom{=}
    +\frac{1}{m}\left(X_{\ui\bar{p}\bar{p}}
     \left(V_{sj}N_k + V_{sk}N_j\right) -  X_{N\bp\bp}V_{si}V_{jk}\right),
    \end{split}
\end{align*}
where we have used \eqref{eq:horizgrad} in the second line to write $\nabla_{\bp}\tilde{X}_{i\bp\bp}$ in terms of $X$.
Permuting the arguments, we thus also have
\begin{align*}
\begin{split}
 \nabla_s(X_{\bp\uj\bp}V_{ik}) &= \nabla_s X_{\bp\uj\bp}V_{ik} + X\ast \Ec^{\prime}  +\frac{1}{m}\left(\Pc_{\bar{s}\uj N} + \Pc_{N\uj\bar{s}}-\Pc_{\bp\bs\bp} N_j\right)V_{ik}\\
     &\phantom{=}
    +\frac{1}{m}\left(X_{\bar{p}\uj\bar{p}}
     \left(V_{si}N_k + V_{sk}N_i\right) -  X_{\bp N \bp}V_{sj}V_{ik}\right),
\end{split}
\end{align*}
and
\begin{align*}
 \begin{split}
  \nabla_s(X_{\bp\bp\uk}V_{ij}) &= \nabla_s X_{\bp\bp\uk}V_{ij} + X\ast \Ec^{\prime}  +\frac{1}{m}\left(\Pc_{N\bar{s} \uk} + \Pc_{\bar{s}N \uk}-\Pc_{\bp\bp\bs} N_k\right)V_{ij}\\
     &\phantom{=}
    +\frac{1}{m}\left(X_{\bar{p}\bar{p}\uk}
     \left(V_{sj}N_i + V_{si}N_k\right) -  X_{\bp\bp N}V_{sk}V_{ij}\right).
 \end{split}
\end{align*}
Thus, combining terms, we obtain that
\begin{align*}
\begin{split}
    &\nabla_s \Pc_{ijk} = \Pc(\nabla_s X)_{ijk} + X\ast \Ec^{\prime} +\frac{1}{m}\left(\Pc_{\bs \uj\uk} N_i + \Pc_{\ui \bs \uk} N_j
     + \Pc_{\ui\uj\bs} N_k\right)\\
    &\phantom{=} +\frac{1}{m^2}\left(\Pc_{\bs\bp\bp} N_i - \Pc_{\ui \bar{s} N} - \Pc_{\ui N \bar{s}}\right)V_{jk} + \frac{1}{m^2}\left(\Pc_{\bp\bs\bp} N_j-\Pc_{\bar{s}\uj N} - \Pc_{N\uj\bar{s}}\right)V_{ik}
    \\
    &\phantom{=}  +\frac{1}{m^2}\left(\Pc_{\bp\bp\bs} N_k - \Pc_{N\bar{s} \uk} - \Pc_{\bar{s}N \uk}\right)V_{ij} +\frac{1}{m}\left(X_{N\uj\uk}V_{si} + X_{\ui N \uk}V_{sj} + X_{\ui\uj N}V_{sk}\right)\\
     &\phantom{=}   +\frac{1}{m^2}\left(X_{N\bp\bp}V_{si}V_{jk}-X_{\ui\bar{p}\bar{p}}
     \left(V_{sj}N_k + V_{sk}N_j\right)\right)  \\
     &\phantom{=}+\frac{1}{m^2}\left(X_{\bp N \bp}V_{sj}V_{ik} - X_{\bar{p}\uj\bar{p}}\left(V_{si}N_k + V_{sk}N_i\right)\right)
     \\
     &\phantom{=} +\frac{1}{m^2}\left(X_{\bp\bp N}V_{sk}V_{ij}-X_{\bar{p}\bar{p}\uk}
     \left(V_{sj}N_i + V_{si}N_k\right) \right),
\end{split}
\end{align*}
which implies \eqref{eq:pgradient}.
\end{proof}

Next we compute the commutator of $\Pc$ and the Laplacian.

\begin{proposition}\label{prop:plaplacian}
For any smooth section $X$ of $T^3(T^*M)$, $\Pc = \Pc(X)$ satisfies
\begin{align}\label{eq:plaplacian}
\begin{split}
&  \Delta\Pc_{ijk} =  \Pc(\Delta X)_{ijk} + \Cc_{ijk} + \Cc^{\prime}_{ijk}\\
     &\phantom{=}+\frac{2}{m}\left(
        \nabla_{\bi}X_{N\uj\uk} + \frac{1}{m}\left(\nabla_{\bi}X_{N\bp\bp}V_{jk} - \nabla_{\bi}X_{\bp\uj\bp}N_k - \nabla_{\bi}X_{\bp\bp\uk}N_j\right)\right)\\
     &\phantom{=}+ \frac{2}{m}\left(\nabla_{\bj}X_{\ui N \uk} + \frac{1}{m}\left(\nabla_{\bj}X_{\bp N\bp}V_{ik}
     - \nabla_{\bj}X_{\ui\bp\bp}N_k - \nabla_{\bj}X_{\bp\bp\uk}N_i\right) \right)\\
     &\phantom{=} + \frac{2}{m}\left(\nabla_{\bk}X_{\ui\uj N}+ \frac{1}{m}\left(\nabla_{\bk}X_{\bp\bp N}V_{ij}
     - \nabla_{\bk}X_{\ui\bp\bp}N_j- \nabla_{\bk}X_{\bp\uj\bp}N_i\right)\right).
\end{split}
\end{align}
where  $\Cc$ and $\Cc^{\prime}$ satisfy
\begin{align*}
 \Cc_{ijk} &=   \frac{2}{m}\left(\nabla_{\bp}\Pc_{\ui\bp N} + \nabla_{\bp}\Pc_{\ui N \bp} - \nabla_{\bp}\Pc_{\bp\bq\bq}N_i\right)V_{jk}\\
&\phantom{=}+ \frac{2}{m}\left(\nabla_{\bp}\Pc_{\bp\uj N} + \nabla_{\bp}\Pc_{N\uj \bp} - \nabla_{\bp}\Pc_{\bq\bp\bq}N_j\right)V_{ik}\\
&\phantom{=}+\frac{2}{m}\left(\nabla_{\bp}\Pc_{N\bp \uk} + \nabla_{\bp}\Pc_{\bp N \uk} - \nabla_{\bp}\Pc_{\bq\bq\bp}N_k\right)V_{ij} +N \ast N \ast \Pc,
\end{align*}
and
\begin{equation}\label{eq:c3def}
 \Cc^{\prime} =  (\nabla X + (\Ec^{\prime} + N)\ast X) \ast \Ec^{\prime} + X \ast \Ec^{\prime\prime}.
\end{equation}
\end{proposition}
\begin{proof} In the computations below, we will use $\Cc^{\prime}$ to denote a sequence of tensors having the schematic form \eqref{eq:c3def}.
We first apply \eqref{eq:horizlap} to $X^H$
to find that
  \begin{align*}
    \begin{split}
     &\Delta X^H_{ijk}
      = (\Delta X)^H_{ijk} + \Cc^{\prime}  +\frac{2}{m^2}|N|^2\left(\Pc_{\bi\uj\uk} + \Pc_{\ui\bj\uk} + \Pc_{\ui\uj\bk}\right)\\
       &\phantom{=}
       -\frac{2}{m}\left(\nabla_{\bi}X_{N\uj\uk}
       -\nabla_{\bj}X_{\ui N \uk}
       -\nabla_{\bk}X_{\ui\uj N}\right)\\
       &\phantom{=}
       -\frac{2}{m}\left(\nabla_{\bp}(\Pc_{\bp\uj\uk})N_i
         +\nabla_{\bp}(\Pc_{\ui\bp\uk})N_j
         +\nabla_{\bp}(\Pc_{\ui\uj\bp})N_k\right)\\\
        &\phantom{=}
        + \frac{2}{m^2}\bigg(
        (X_{NN\uk}V_{ij}- X_{\bp\bp \uk} N_i N_j)
         +(X_{N\uj N}V_{ik}- X_{\bp\uj \bp} N_i N_k)\\
         &\phantom{=}\qquad
         +(X_{\ui N N}V_{jk}- X_{\ui\bp\bp} N_j N_k)\bigg),
    \end{split}
 \end{align*}
so that, as before,
\begin{align}\label{eq:hlap1}
    \begin{split}
     &\Delta X^H_{ijk}  = (\Delta X)^H_{ijk}  + \Cc^{\prime}
      + N \ast N \ast P\\
       &\phantom{=}
       -\frac{2}{m}\left(\nabla_{\bi}X_{N\uj\uk}
       -\nabla_{\bj}X_{\ui N \uk}
       -\nabla_{\bk}X_{\ui\uj N}\right)\\
       &\phantom{=}
       -\frac{2}{m}\left(\nabla_{\bp}\Pc_{\bp\uj\uk}N_i
         +\nabla_{\bp}\Pc_{\ui\bp\uk}N_j
         +\nabla_{\bp}\Pc_{\ui\uj\bp}N_k\right)\\
 &\phantom{=}     + \frac{2}{m^2}\bigg(
        (X_{NN\uk}V_{ij}- X_{\bp\bp \uk} N_i N_j)
         +(X_{N\uj N}V_{ik}- X_{\bp\uj \bp} N_i N_k)\\
         &\phantom{=}\qquad
          +(X_{\ui N N}V_{jk}- X_{\ui\bp\bp} N_j N_k)\bigg).
    \end{split}
 \end{align}

Next, we write $\tilde{X}_{ijk} = X_{\ui jk}$, and apply Proposition \ref{prop:tracelap} to obtain that
\begin{align}
\begin{split}
\label{eq:plap1}
\Delta \Tc(\tilde{X})_{ijk} &= \Tc(\Delta \tilde{X})_{ijk}
 + \frac{2}{m}\left(\nabla_{\bp}(X_{\ui\bp \uq}) + \nabla_{\bp}(X_{\ui\uq\bp})\right)N_{q}V_{jk}\\
 &\phantom{=}+\frac{2}{m}\left(\nabla_{\bj}\tilde{X}_{\ui\bp\bp}N_k + \nabla_{\bk}\tilde{X}_{\ui\bp\bp}N_j\right) + \frac{2}{m}\left(X_{\ui\bp\bp}N_jN_k - X_{\ui NN}V_{jk}\right)\\
 &\phantom{=} + \frac{2}{m^2}\left(X_{\ui\bj N}N_k + X_{\ui N\bk}N_j\right) +\Cc^{\prime}(\tilde{X}),
\end{split}
\end{align}
where, within the $\Cc^{\prime}$ term, we have
used that $\nabla\tilde{X} = \nabla X + (N + \Ec^{\prime})\ast X$.
We can immediately simplify this to
\begin{align}
\begin{split}
\label{eq:plap2}
\Delta \Tc(\tilde{X})_{ijk} &= \Tc(\Delta \tilde{X})_{ijk} +\Cc^{\prime} + N \ast N \ast \Pc
 + \frac{2}{m}\left(\nabla_{\bp}\Pc_{\ui\bp N} + \nabla_{\bp}\Pc_{i N \bp}\right)V_{jk}\\
 &\phantom{=}+\frac{2}{m}\left(\nabla_{\bj}\tilde{X}_{\ui\bp\bp}N_k + \nabla_{\bk}\tilde{X}_{\ui\bp\bp}N_j\right)\\
 &\phantom{=}+ \frac{2}{m}\left(X_{\ui\bp\bp}N_jN_k - X_{\ui NN}V_{jk}\right).
\end{split}
\end{align}
Further, by \eqref{eq:horizgrad}, we have
\begin{align}
\nonumber
 \nabla_{\bj}\tilde{X}_{i\bp\bp} &= \nabla_{\bj} X_{\ui\bp\bp}  -\frac{1}{m}(X_{N\bp\bp}V_{ij}+ X_{\bj \bp \bp}N_i)+  X \ast \Ec^{\prime}\\
 \label{eq:plap4}
 &= \nabla_{\bj} X_{\ui\bp\bp}  -\frac{1}{m}X_{N\bp\bp}V_{ij} + X \ast \Ec^{\prime}+ N \ast \Pc,
\end{align}
and
\begin{align}\label{eq:plap5}
 \nabla_{\bk}\tilde{X}_{i\bp\bp}
 &= \nabla_{\bk} X_{\ui\bp\bp}  -\frac{1}{m}X_{N\bp\bp}V_{ik} + X \ast \Ec^{\prime}+ N \ast \Pc.
\end{align}

Finally, by \eqref{eq:horizlap}, we have
 \begin{align}\nonumber
    \begin{split}
     &\Tc(\Delta \tilde{X})_{ijk}
      = \Delta X_{\ui\bp\bp}V_{jk} + \Cc^{\prime}\\
      &\phantom{=}
      + \frac{2}{m}\left(\frac{|N|^2}{m} X_{\bi\bp\bp}V_{jk}
      -\nabla_{\bi}X_{N\bp\bp}-\nabla_{\bp}(X_{\bp ab})V_{ab} N_i\right)V_{jk}
    \end{split}\\
\begin{split}\label{eq:plap6}
    &= \Delta X_{\ui\bp\bp}V_{jk} + \Cc^{\prime} + N\ast N \ast \Pc\\
      &\phantom{=}
      - \frac{2}{m}\left(
      \nabla_{\bi}X_{N\bp\bp} + \nabla_{\bp}\Pc_{\bp \bq\bq}N_i -\frac{1}{m}\left(X_{\bp\bp N}+ X_{\bp N\bp}\right)N_i\right)V_{jk},
\end{split}
 \end{align}
where we have used that
\begin{align*}
 \nabla_{\bp}(X_{\bp ab})V_{ab}
 &= \nabla_{\bp}(\Pc_{\bp\bq\bq}) -\frac{1}{m}(X_{\bp \bp N}
 + X_{\bp N \bp})\\
 &= \nabla_{\bp}\Pc_{\bp\bq\bq} -\frac{1}{m}(X_{\bp \bp N}
 + X_{\bp N \bp}) + (N + \Ec^{\prime}) \ast \Pc\\
 &= \nabla_{\bp}\Pc_{\bp\bq\bq} -\frac{1}{m}(X_{\bp \bp N}
 + X_{\bp N \bp}) + X \ast \Ec^{\prime} + N \ast \Pc.
\end{align*}

Combining the results of the equations \eqref{eq:plap2} - \eqref{eq:plap6} in \eqref{eq:plap1}, we obtain that
\begin{align*}
\begin{split}
&\Delta(X_{\ui\bp\bp}V_{jk}) =\Delta X_{\ui\bp\bp} V_{jk}
 +\frac{2}{m}\left(\nabla_{\bp}\Pc_{\ui\bp N} + \nabla_{\bp}\Pc_{\ui N \bp} - \nabla_{\bp}\Pc_{\bp\bq\bq}N_i\right)V_{jk}\\
  &\quad\phantom{=} -\frac{2}{m}\left(\nabla_{\bi}X_{N\bp\bp} -\frac{1}{m} (X_{\bp N\bp} + X_{\bp\bp N})N_i\right)V_{jk}
  + \frac{2}{m}\left(X_{\ui\bp\bp}N_{j}N_k -X_{\ui N N}V_{jk}\right)
  \\
  &\quad\phantom{=} +\frac{2}{m}\left(\nabla_{\bj}X_{\ui\bp\bp}
  - \frac{1}{m}X_{N\bp\bp}V_{ij}\right)N_k +\frac{2}{m}\left(\nabla_{\bk}X_{\ui\bp\bp}
  - \frac{1}{m}X_{N\bp\bp}V_{ik}\right)N_j\\
  &\phantom{=} +  \Cc^{\prime} + N \ast N \ast \Pc,
\end{split}
\end{align*}
or, regrouping terms, that
\begin{align*}
\begin{split}
&\Delta(X_{\ui\bp\bp}V_{jk}) =\Delta X_{\ui\bp\bp} V_{jk}
 +\frac{2}{m}\left(\nabla_{\bp}\Pc_{\ui\bp N} + \nabla_{\bp}\Pc_{\ui N \bp} - \nabla_{\bp}\Pc_{\bp\bq\bq}N_i\right)V_{jk}\\
&\quad \phantom{=} -\frac{2}{m}\left(\nabla_{\bi}X_{N\bp\bp}V_{jk} - \nabla_{\bj}X_{\ui\bp\bp}N_k
- \nabla_{\bk}X_{\ui\bp\bp}N_j\right)  + \frac{2}{m}\left(X_{\ui\bp\bp}N_{j}N_k -X_{\ui N N}V_{jk}\right)\\
&\quad\phantom{=}+ \frac{1}{m}\left((X_{\bp N \bp} + X_{\bp\bp N})N_iV_{jk}- X_{N\bp\bp}V_{ij}N_k
- X_{N\bp\bp}V_{ik}N_j\right) +  \Cc^{\prime} + N \ast N \ast \Pc.
  \end{split}
\end{align*}

Defining the tensors $X^{\prime}_{ijk} \dfn X_{jik}$ and $X^{\prime\prime}_{ijk} \dfn X_{kji}$, so that
$X_{\bp\uj\bp} = X^{\prime}_{\uj\bp\bp}$
and $X_{\bp\bp\uk} = X^{\prime\prime}_{\uk\bp\bp}$, we can then apply
the above identity to $X^{\prime}$ and $X^{\prime\prime}$ to see that
\begin{align*}
\begin{split}
&\Delta(X_{\bp\uj\bp}V_{ik}) =\Delta X_{\bp\uj\bp} V_{ik} +\frac{2}{m}\left(\nabla_{\bp}\Pc_{\bp\uj N} + \nabla_{\bp}\Pc_{N\uj \bp} - \nabla_{\bp}\Pc_{\bq\bp\bq}N_j\right)V_{ik}\\
 &\quad \phantom{=}  -\frac{2}{m}\left(\nabla_{\bj}X_{\bp N\bp}V_{ik} - \nabla_{\bi}X_{\bp\uj\bp}N_k
- \nabla_{\bk}X_{\bp\uj\bp}N_i\right)  + \frac{2}{m}\left(X_{\bp\uj\bp}N_{i}N_k -X_{N\uj N}V_{ik}\right)\\
&\quad\phantom{=}+ \frac{1}{m}\left((X_{ N \bp \bp} + X_{\bp\bp N})N_jV_{ik}- X_{\bp N\bp}V_{ij}N_k
- X_{\bp N\bp}V_{jk}N_i\right)+  \Cc^{\prime} + N \ast N \ast \Pc,
  \end{split}
\end{align*}
and
\begin{align*}
\begin{split}
&\Delta(X_{\bp\bp\uk}V_{ij}) =\Delta X_{\bp\bp\uk} V_{ij}
+\frac{2}{m}\left(\nabla_{\bp}\Pc_{N\bp \uk} + \nabla_{\bp}\Pc_{\bp N \uk} - \nabla_{\bp}\Pc_{\bq\bq\bp}N_k\right)V_{ij}\\
&\quad \phantom{=} -\frac{2}{m}\left(\nabla_{\bk}X_{\bp\bp N}V_{ij} - \nabla_{\bj}X_{\bp\bp\uk}N_i
- \nabla_{\bi}X_{\bp\bp\uk}N_j\right)  + \frac{2}{m}\left(X_{\bp\bp\uk}N_{i}N_j -X_{N N\uk}V_{ij}\right)\\
&\quad\phantom{=}+ \frac{1}{m}\left((X_{\bp N \bp} + X_{N\bp \bp})N_kV_{ij}- X_{\bp\bp N}N_iV_{jk}
- X_{\bp\bp N}V_{ik}N_j\right)+  \Cc^{\prime} + N \ast N \ast \Pc.
  \end{split}
\end{align*}
Combining the above three identities with \eqref{eq:hlap1} and cancelling terms yields at last that
\begin{align*}
 \begin{split}
  \Delta \Pc(X)_{ijk} &=\Pc(\Delta X)_{ijk} +\Cc + \Cc^{\prime} + \frac{2}{m}\left(
        \nabla_{\bi}X_{N\uj\uk} + \nabla_{\bj}X_{\ui N \uk} + \nabla_{\bk}X_{\ui\uj N}\right)\\
        &\phantom{=} + \frac{2}{m^2}\left(\nabla_{\bi}X_{N\bp\bp}V_{jk} - \nabla_{\bj}X_{\ui\bp\bp}N_k
- \nabla_{\bk}X_{\ui\bp\bp}N_j\right)\\
        &\phantom{=} +\frac{2}{m^2}\left(\nabla_{\bj}X_{\bp N\bp}V_{ik} - \nabla_{\bi}X_{\bp\uj\bp}N_k
- \nabla_{\bk}X_{\bp\uj\bp}N_i\right)\\
        &\phantom{=} +\frac{2}{m^2}\left(\nabla_{\bk}X_{\bp\bp N}V_{ij} - \nabla_{\bj}X_{\bp\bp\uk}N_i
- \nabla_{\bi}X_{\bp\bp\uk}N_j\right),
 \end{split}
\end{align*}
where
\begin{align*}
\Cc &=  \frac{2}{m}\left(\nabla_{\bp}\Pc_{\ui\bp N} + \nabla_{\bp}\Pc_{\ui N \bp} - \nabla_{\bp}\Pc_{\bp\bq\bq}N_i\right)V_{jk}\\
&\phantom{=}+ \frac{2}{m}\left(\nabla_{\bp}\Pc_{\bp\uj N} + \nabla_{\bp}\Pc_{N\uj \bp} - \nabla_{\bp}\Pc_{\bq\bp\bq}N_j\right)V_{ik}\\
&\phantom{=}+\frac{2}{m}\left(\nabla_{\bp}\Pc_{N\bp \uk} + \nabla_{\bp}\Pc_{\bp N \uk} - \nabla_{\bp}\Pc_{\bq\bq\bp}N_k\right)V_{ij} +N \ast N \ast \Pc,
\end{align*}
and
\begin{equation*}
 \Cc^{\prime} = (\nabla X + (\Ec^{\prime} + N)\ast X) \ast \Ec^{\prime} + X \ast \Ec^{\prime\prime},
\end{equation*}
and \eqref{eq:plaplacian} follows.
\end{proof}

\section{The evolution of the curvature invariants $M$, $P$, and $U$.}
We will continue to assume in this section that $g(\tau)$ is a solution to the backward Ricci flow \eqref{eq:brf} on $M\times [0, \Omega]$,
and $\Hc(\tau)$ and $\Vc(\tau)$ are a pair of families of complementary orthogonal distributions ($\Vc(\tau)$ having dimension $m$) on $M$ defined by projections $H = H(\tau)$ and $V= V(\tau)$ evolving according to $D_{\tau} H \equiv 0$, $D_{\tau} V \equiv 0$.

The computations in Section \ref{sec:conninvev}
show that the connection-level invariants $A$, $T^0$, and $G$, and the derivatives $\nabla A$ and $\nabla T^0$, can be controlled via their evolution equations by the curvature-level invariants $M$ and $P$ and their derivatives $\nabla M$ and $\nabla P$. The calculations in this section will show that if we add the tensor $U$ to our system, then the aggregate of $M$, $P$, and $U$ and the connection-level invariants will satisfy a closed system of inequalities.

\subsection{The evolution of the tensor $M$}
Let $\mathcal{M}$ be the projection map
\[
\mathcal{M}: S^{2}(T^*M)\longrightarrow S^2(T^*M),
\]
defined by
\[
 \mathcal{M}(X) = X - X^H - \frac{1}{m}\trace_V(X) V,
\]
so that the tensor $M$ defined in \eqref{eq:mdef} is given by $M = \Mm(\Rc)$.

\begin{proposition}\label{prop:mev} The tensor $M$ satisfies
\begin{align}\label{eq:mev}
\begin{split}
     |(D_{\tau} + \Delta)M| &\lesssim  \Theta_1(|M| + |\nabla M| + |P|) \\
     &\phantom{\lesssim}
   +  \Theta_2 (|A| + |T^0| + |\nabla A| + |\nabla T^0| +|G|),
\end{split}
\end{align}
where
\[
  \Theta_1 = \Theta_1(|N|, |\Rm|) \quad \mbox{and}\quad \Theta_2 = \Theta_2(|N|, |\Rm|, |\nabla \Rm|,  |A|, |T^0|).
\]
are polynomials as described in Section \ref{ssec:notation}.
\end{proposition}

\begin{proof}
First, note that $D_{\tau}\Mm(\Rc)= \Mm(D_{\tau} \Rc)$, so we have
\begin{align}\label{eq:mev1}
(D_{\tau} + \Delta)M &= \Mm((D_{\tau} + \Delta)\Rc) + \Delta (\Mm(\Rc))- \Mm(\Delta \Rc)  .
\end{align}
To simplify the first term, we begin with the evolution equation
 \begin{align}
 \nonumber
    (D_{\tau} + \Delta)R_{ij} &= -2R_{ipqj}R_{pq}
\end{align}
for the Ricci tensor under \eqref{eq:brf}.
Temporarily using the notation $\Rm(X)_{ij} = \Rm_{ipqj}X_{pq}$, we may write
\begin{align*}
 \Rm(\Rc) &= \Rm(M) + \Rm(\Rc^H) + \frac{{\Rh}}{m}\Rm(V)\\
 &=
 \Rm(M) + \Rm\left(\Rc^H - \frac{{\Rh}}{m}H\right) + \frac{{\Rh}}{m}\Rc
\end{align*}
as $H + V = \operatorname{Id}$.
But, for any symmetric two-tensor $X$ such that $X = X^H$, we have
\begin{align*}
    \Rm(X) &= X\ast Q + \Rm(X)^H+ \frac{\trace_V(\Rm(X))}{m}V,
\end{align*}
where $Q$ is defined as in \eqref{eq:qdef}.
Since $\Mm(V) = 0$ and $\Mm(X^H) = 0$ for any symmetric two- tensor $X$, it follows
that
\begin{equation}\label{eq:mev2}
 \Mm ((D_{\tau} + \Delta)\Rc) = - 2\Mm(\Rm(\Rc)) = \Rm \ast M + \Rm \ast Q.
\end{equation}

It remains to compute the commutator
\[
  \Delta \Mm(\Rc)_{jk} - \Mm(\Delta \Rc)_{jk}  =   \Delta M_{jk}  - \Mm(\Delta \Rc)_{jk}.
\]
First, we use \eqref{eq:hlap2} of Proposition \ref{prop:horizlap}
to see that
\begin{align*}
\begin{split}
     \Delta \Rc^H_{ij}  
      &= (\Delta \Rc)^H_{ij} + (\nabla \Rc + (\Ec^{\prime} + N) \ast \Rc)\ast \Ec^{\prime}
      + \Rc\ast \Ec^{\prime\prime}\\
        &\phantom{=}  +\frac{2}{m^2}|N|^2\left(M_{\bi\uj} + M_{\ui\bj}\right)
      -\frac{2}{m}\left(\nabla_{\bi}R_{N\uj} + \nabla_{\bj}R_{\ui N}\right)\\
       &\phantom{=}-\frac{2}{m}
         \left(\nabla_{\bp}\left(M_{\bp\uj}\right)N_i + 
  \nabla_{\bp}\left(M_{\ui\bp}\right)N_j\right) +\frac{2}{m^2}
 \left(X_{NN}V_{ij} - X_{\bp\bp}N_{i}N_j\right)
    \end{split}\\
    &= (\Delta \Rc)^H_{ij} + \frac{2}{m^2}
 \left(R_{NN}V_{ij} - \Rh N_{i}N_j\right) \\
 &\phantom{=}+ N \ast \nabla M  +  N \ast N \ast M + N \ast P+ (\nabla \Rc + (\Ec^{\prime} + N) \ast \Rc)\ast \Ec^{\prime}\\
    &\phantom{=}  
     + \Rc\ast \Ec^{\prime\prime},
 \end{align*} 
where we have used, for example, that 
\[
  \nabla_{\bp}\left(M_{\ui\bp}\right) =  \nabla_{\bp} M_{\ui\bp} + (N + \Ec^{\prime}) \ast M.
\]
Also, using equation \ref{eq:traceprojlap} from Proposition \ref{prop:tracelap}, we have
\begin{align*}
\begin{split}
\label{eq:traceprojlap}
\Delta \Tc(\Rc)_{ij} &= \Tc(\Delta \Rc)_{ij}
 + \frac{2}{m}\left(\nabla_{\bp}(M_{\bp \uq}) + \nabla_{\bp}(M_{\uq\bp})\right)N_{q}V_{ij}\\ 
 &\phantom{=}+\frac{2}{m}\left(\nabla_{\bi}R_{\bp\bp}N_j + \nabla_{\bj}R_{\bp\bp}N_i\right) + \frac{2}{m}\left(\Rh N_iN_j - R_{NN}V_{ij}\right)\\
 &\phantom{=} + \frac{2}{m^2}\left(M_{\bi N}N_j + M_{N\bj}N_i\right) + (\nabla \Rc + N \ast \Rc) \ast \Ec^{\prime} + \Rc \ast \Ec^{\prime\prime}
\end{split}\\
&= \Tc(\Delta \Rc)_{ij} + \frac{2}{m}\left(\Rh N_iN_j - R_{NN}V_{ij}\right) + N \ast \nabla M  + N \ast N \ast M
\\
&\phantom{=}+ N \ast P + (\nabla \Rc + N \ast \Rc) \ast \Ec^{\prime} + \Rc \ast \Ec^{\prime\prime}.
\end{align*}
Since $\Mm(X) = \operatorname{Id} - X^H - \frac{1}{m}\Tc(X)$, we may combine the two identities above and cancel terms
to see that
\[
  \Delta M  - \Mm(\Delta \Rc)=  (\Delta\Rc)^H - \Delta (\Rc^H)
 + \frac{1}{m}\left(\Tc(\Delta X) - \Delta \Tc(X) \right) = \Cc,
\]
where
\[
 \Cc =  N \ast P +  N \ast \nabla M  + N \ast N \ast M+ (\nabla \Rc + N \ast \Rc) \ast \Ec^{\prime} + \Rc \ast \Ec^{\prime\prime}.
\]
Combining this with \eqref{eq:mev1} and \eqref{eq:mev2}, and using
the equations for $Q$ and $\Ec^{\prime}$ and $\Ec^{\prime\prime}$ to rewrite them in terms of
the connection invariants,
we obtain \eqref{eq:mev}.
\end{proof}

\subsection{The evolution of the tensor $P$.}

Now we turn our attention to $P = \Pc(\nabla \Rc)$.  First we observe that, with the specific choice $X = \nabla \Rc$, we can use the Bianchi identities to estimate the output of the commutator identity in Proposition \ref{prop:plaplacian} further.
\begin{proposition}\label{prop:prclap}
 The tensor $P = \Pc(\Rc)$ satisfies
 \begin{align}\label{eq:prclap}
\begin{split}
 |\Delta P - \Pc(\Delta\nabla\Rc)| &\lesssim 
 \Theta_1 \big(|M| + |P| + |\nabla P|\big)\\
 &\phantom{=}+ \Theta_2 \big(|A| + |T^0| + |G| + |\nabla A| + |\nabla T^0|\big),
\end{split}
\end{align}
where 
\[
 \Theta_1 = \Theta_1(|N|, |\Rm|), \quad \mbox{and} \quad
 \Theta_2 = \Theta_2(|N|, |\Rm|, |\nabla \Rm|, |\nabla^2\Rm|, |A|, |T^0|).
\]

\end{proposition}
\begin{proof} 
 With $X = \nabla \Rc$, Proposition \ref{prop:horizlap}
gives that
\begin{align}\label{eq:pev8}
\begin{split}
&   \Delta P_{ijk}  - \Pc(\Delta \nabla\Rc)_{ijk}  +  \Cc\\
     &\phantom{=}+\frac{2}{m}\left(\nabla_{\bi}\nabla_NR_{\uj\uk} + \frac{1}{m}\left(\nabla_{\bi}\nabla_NR_{\bp\bp}V_{jk} - \nabla_{\bi}\nabla_{\bp}R_{\uj\bp}N_k - \nabla_{\bi}\nabla_{\bp}R_{\bp\uk}N_j\right)\right)\\
     &\phantom{=}+ \frac{2}{m}\left(\nabla_{\bj}\nabla_{\ui}R_{N \uk} + \frac{1}{m}\left(\nabla_{\bj}\nabla_{\bp}R_{N\bp}V_{ik}
     -\nabla_{\bj}\nabla_{\ui}R_{\bp\bp}N_k - \nabla_{\bj}\nabla_{\bp}R_{\bp\uk}N_i\right) \right)\\
     &\phantom{=} + \frac{2}{m}\left(\nabla_{\bk}\nabla_{\ui}R_{\uj N}+ \frac{1}{m}\left(\nabla_{\bk}\nabla_{\bp}R_{\bp N}V_{ij}
     - \nabla_{\bk}\nabla_{\ui}R_{\bp\bp}N_j- \nabla_{\bk}\nabla_{\bp}R_{\uj\bp}N_i\right)\right),
\end{split}
\end{align}
where
\[
 |\Cc| \lesssim (|\nabla^2\Rc| + (|\Ec^{\prime}| + |N|)|\nabla\Rc|)|\Ec^{\prime}|
+ |\nabla \Rc||\Ec^{\prime\prime}| + |N|^2|P| + |N||\nabla P|.
\]
We first consider those derivative terms whose first component is vertical but whose remaining three components
are horizontal. We will need to manipulate
these terms a little bit to see that they can indeed be controlled by the elements of our system. Let us
start with the first such term (which is representative of the others). We compute
\begin{align*}
 \nabla_{\bi}\nabla_{N}R_{\uj\uk} &= \nabla_N \nabla_{\bi}R_{\uj\uk} -R_{\bi N\uj p}R_{p\uk}- R_{\bi N \uk p} R_{\uj p}.
\end{align*}
Now, on one hand, we have
\[
 \nabla_N \nabla_{\bi}R_{\uj\uk} = \nabla_N P_{\bi\uj\uk} + N \ast \nabla\Rc \ast \Ec^{\prime}
\]
as $N$ is horizontal and $\nabla_{\ui} H_{jk} = - \nabla_{\ui}V_{jk} = \Ec^{\prime}_{\ui jk}$. On the other,
\begin{align*}
 R_{\bi N\uj p}R_{p\uk} + R_{\bi N \uk p} R_{\uj p} &= R_{\bi N\uj \bp}R_{\bp\uk} + R_{\bi N \uk \bp}R_{\uj \bp} +  R_{\bi N\uj \up}R_{\up\uk} + R_{\bi N \uk \up}R_{\uj \up}\\
 &= R_{\bi N\uj \bp}M_{\bp\uk} + R_{\bi N \uk \bp}M_{\uj \bp} +  Q_{\bi N\uj \up}R_{\up\uk} + Q_{\bi N \uk \up}R_{\uj \up},
\end{align*}
so
\[
 \nabla_{\bi}\nabla_{N}R_{\uj\uk} = N \ast \nabla P + N \ast \nabla\Rm \ast \Ec^{\prime} + N \ast \Rm \ast (M +  Q).
\]
The leftmost terms in the parenthetical expressions in \eqref{eq:pev8} satisfy the same schematic identity.

The other derivative terms in \eqref{eq:pev8} can be estimated similarly, once we complete the trace to swap two vertical components for two horizontal components. For example, for the second term in the first parenthetical
expression, we may compute
\begin{align*}
    \nabla_{\bi}\nabla_NR_{\bp\bp} &= \nabla_{\bi}\nabla_N R - \nabla_{\bi}\nabla_{N}R_{\up\up}.
\end{align*}
On one hand, our prior computation shows that
\begin{align}
 \nabla_{\bi}\nabla_{N}R_{\up\up} =  N \ast \nabla P + N \ast \nabla\Rm \ast \Ec^{\prime} + N \ast \Rm \ast(M + Q).
\end{align}
On the other,
\begin{align*}
\nabla_{\bi}\nabla_N R &=\nabla_{N}\nabla_{\bi} R =\nabla_N S_i - \nabla_N V_{ia} \nabla_a R
= N \ast \nabla S + \nabla R\ast N \ast \Ec^{\prime}\\
&= N \ast N \ast P + N \ast \nabla P + N \ast \nabla \Rm \ast \Ec^{\prime}.
\end{align*}
Using the contracted second Bianchi identity, all of the remaining terms in \eqref{eq:pev8} can be seen
to satisfy the same identity.  We conclude that
\begin{align*}
 |\Delta P - \Pc(\Delta\nabla\Rc)| &\lesssim  ((|N| + |\Ec^{\prime})|\nabla\Rm|+ |\nabla^2\Rm|)|\Ec^{\prime}|
+ |\nabla \Rm||\Ec^{\prime\prime}|\\
&\phantom{\lesssim} +|N||\Rm| (|Q| + |M|) + |N|^2|P| + |N||\nabla P|,
\end{align*}
and \eqref{eq:pev} follows.
\end{proof}

Next, we apply the commutation identity in \eqref{eq:prclap} to derive an expression for the  evolution of $P =\Pc(\Rc)$.
\begin{proposition}\label{prop:pev}
The tensor $P$ satisfies an evolution equation of the form
\begin{align}
\label{eq:pev}
\begin{split}
    |(D_{\tau} + \Delta)P| &\lesssim \Theta_1
    (|M| + |P| + |\nabla P| + |U|)\\
    &\phantom{\lesssim} + \Theta_2(|A| + |T^0| + |G|
    + |\nabla A| + |\nabla T^0|),\\
\end{split}
\end{align}
where 
\[
\Theta_1 = \Theta_1(|N|, |\Rm|, |\nabla\Rm|),
\]
and
\[
 \Theta_2 = \Theta_2(|N|, |\Rm|, |\nabla \Rm|, |\nabla^2 \Rm|, |A|, |T^0|, |\nabla A|, |\nabla T^0|)
\]
are polynomials as in Section \ref{ssec:notation}.
\end{proposition}
\begin{proof}
As in Proposition \ref{prop:mev}, we have
\begin{align}
\begin{split}\label{eq:pev0}
    (D_{\tau} + \Delta)P_{ijk} = \Pc\left((D_{\tau} + \Delta)\nabla\Rc\right)_{ijk}
    + \Delta( \Pc(\nabla\Rc))_{ijk} -  \Pc(\Delta \nabla \Rc)_{ijk}.
\end{split}
\end{align}
By Proposition \ref{prop:prclap}, we only need to consider the first term. Using the identity
 \[
  [D_{\tau} + \Delta, \nabla_i]R_{jk} = -2R_{ipqj}\nabla_pR_{qk} - 2R_{ipqk}\nabla_pR_{qj},
 \]
and the evolution equation for $\Rc$ under \eqref{eq:rf},
we first of all compute that
\begin{align}
\begin{split}\label{eq:pev1}
    (D_{\tau} + \Delta)\nabla_i R_{jk} &= - 2 \nabla_i R_{jpqk}R_{pq} - 2R_{jpqk}\nabla_{i}R_{pq} - 2R_{ipqj}\nabla_pR_{qk}\\
    &\phantom{=}- 2R_{ipqk}\nabla_pR_{qj}.
\end{split}
\end{align}
We will estimate the first two terms on the right. The last two can be estimated in the same way as the second.

For the first term on the right of \eqref{eq:pev1}, note that
that
\begin{align}\label{eq:pev2}
\begin{split}
  \nabla_iR_{jpqk}R_{pq} & = \nabla_iR_{jpqk}M_{pq} + \nabla_iR_{j\up\uq k} R_{\up\uq} + \frac{{\Rh}}{m}\nabla_iR_{j\bp\bp k}\\
   & = \nabla_iR_{jpqk}M_{pq} + \nabla_iR_{j\up\uq k} R_{\up\uq} + \frac{{\Rh}}{m}(\nabla_i R_{jk} - \nabla_iR_{j\up\up k}).
\end{split}
\end{align}
Hence,
\begin{align}\label{eq:pev3}
\begin{split}
    \Pc(\nabla_iR_{jpqk}R_{pq}) &=\nabla \Rm \ast M + U_{ij\up\uq k}R_{\up\uq}
    + \frac{{\Rh}}{m}(P_{ijk} - U_{ij\up\up k})\\
    &=\nabla \Rm \ast M + \Rm \ast U + \Rm \ast P.
\end{split}
\end{align}

For the second term on the right of \eqref{eq:pev1}, we have
\begin{align*}
 &R_{jpqk}\nabla_i R_{pq} = R_{jpqk}\left(P_{ipq} + \nabla_{\ui}R_{\up\uq}
+ \frac{1}{m}\left(\nabla_{\ui}R_{\br\br}V_{pq}  + \nabla_{\br}R_{\br \uq} V_{ip} + \nabla_{\br}R_{\up\br}V_{iq}\right)\right)\\
&= R_{jpqk} P_{ipq} + R_{j\up\uq k}\nabla_{\ui}R_{\up\uq}
+ \frac{1}{m}\left(\nabla_{\ui}R_{\br\br}R_{j\bp\bp k}  + \nabla_{\br}R_{\br \uq}R_{j\bi\uq k}
+ \nabla_{\br}R_{\up\br}R_{j\up\bi k}\right).
\end{align*}
Now,
\begin{align*}
 R_{j\up\uq k}\nabla_{\ui}R_{\up\uq} &= \left(Q_{j\up\uq k} + \frac{1}{m}V_{jk}R_{\br\up\uq\br}\right)\nabla_{\ui}R_{\up\uq},
\end{align*}
and the second term belongs to the kernel of $\Pc$, so
\begin{align}\label{eq:pev4}
\Pc(R_{j\up\uq k}\nabla_{\ui}R_{\up\uq}) = \nabla\Rm \ast Q.
\end{align}
Similarly,
\begin{align*}
 \nabla_{\ui}R_{\br\br}R_{j\bp\bp k} &=  \nabla_{\ui}R_{\br\br}\left(R_{jk} - R_{j\up\up k}\right)\\
 &= \nabla_{\ui}R_{\br\br}\left(M_{jk} + R_{\uj\uk} + \frac{{\Rh}}{m}V_{jk} - Q_{j\up\up k} - \frac{1}{m}R_{\br\up\up\br}V_{jk}\right),
\end{align*}
so
\begin{align}\label{eq:pev5}
 \Pc(\nabla_{\ui}R_{\br\br}R_{j\bp\bp k}) &= \nabla \Rm \ast M + \nabla \Rm \ast Q. 
\end{align}
On the other hand, 
\begin{align*}
 \nabla_{\br}R_{\br \uq}(R_{j\bi\uq k} + R_{j\uq\bi k}) &= 
 \nabla_{\br}R_{\br \uq}(R_{\bi j k\uq} + R_{\bi k j \uq}) \\
 &= \nabla\Rm \ast Q + \frac{1}{m}\nabla_{\br}R_{\br \uq}\left(R_{\bp \uj \bp \uq} V_{ik} + R_{\bp \uk \bp\uq}V_{ij}\right),
\end{align*}
so that
\[
 \Pc(\nabla_{\br}R_{\br \uq}(R_{j\bi\uq k} + R_{j\uq\bi k})) = \nabla \Rm \ast Q,
\]
and together with \eqref{eq:pev4} and \eqref{eq:pev5}, we see that the second term on the right side of \eqref{eq:pev1}
satisfies
\begin{equation}\label{eq:pev6}
 \Pc(R_{jpqk}\nabla_i R_{pq}) = \nabla \Rm \ast Q + \nabla \Rm \ast M.
\end{equation}

The same reasoning shows that the images of the third and fourth terms in  \eqref{eq:pev1} under the projection $\Pc$ satisfy an identity of the same schematic form as \eqref{eq:pev6}. Thus we conclude at last from \eqref{eq:pev3}, \eqref{eq:pev5},
and \eqref{eq:pev6} that
\begin{equation}\label{eq:pev7}
  \Pc((D_{\tau} + \Delta)\nabla \Rc) = \nabla \Rm \ast M + \Rm \ast P + \Rm \ast U + \nabla \Rm \ast Q,
\end{equation}
and \eqref{eq:pev} follows.
\end{proof}
 
\subsection{The evolution equation for $U$.}\label{sec:uproj}
Our computation for the tensor $U$ goes along much as for $P$, but the expression is algebraically more complicated and there is more to organize.

\begin{proposition}\label{prop:uev}
The tensor $U$ defined by \eqref{eq:udef} satisfies the estimate
\begin{align}\label{eq:uev}
\begin{split}
    |(D_{\tau} + \Delta)U| &\lesssim \Theta_1
    (|M| + |P| + |\nabla P| + |U| + |\nabla U|)\\
    &\phantom{\lesssim}+ \Theta_2(|A|
    + |T^0| + |G| + |\nabla A| + |\nabla T^0|),
\end{split}
\end{align}
where
\[
 \Theta_1 = \Theta_1(|N|, |\Rm|, |\nabla \Rm|),
\]
and
\[
 \Theta_2 = \Theta_2(|N|, |\Rm|, |\nabla \Rm|,
 |\nabla^2 \Rm|, |A|, |T^0|)
\]
are polynomials as in Section \ref{ssec:notation}.
\end{proposition}
We will break this estimate up into several pieces. To begin, we recall that the covariant derivative of the curvature tensor evolves along \eqref{eq:brf} by
\begin{align*}
 \begin{split}
 &\left(D_{\tau} + \Delta\right)\nabla_aR_{ijkl} = J_{aijkl} + L_{aijkl},
\end{split}
\end{align*}
where
\begin{align}\label{eq:jdef}
 J_{aijkl} \dfn 2\nabla_a\left(B_{ijkl} - B_{ijlk} + B_{ikjl} - B_{iljk}\right), \quad  B_{ijkl} = g^{pr}g^{qs}R_{pijq}R_{rkls},
\end{align}
and
\begin{equation}\label{eq:ldef}
 L_{aijkl} \dfn 2\big(R_{iqap}\nabla_p R_{qjkl} + R_{jqap}\nabla_p R_{iqkl} + R_{kqap}\nabla_pR_{ijql} + R_{lqap}\nabla_pR_{ijkq}\big).
\end{equation}
Note that our $B_{ijkl}$ differs from that defined in \cite{Hamilton3D} by a factor of $-1$.

\begin{proof}[Proof of Proposition \ref{eq:uev}]
Since $U = \Uc(\nabla \Rm)$ and $[D_{\tau}, \Uc] = 0$, we have
\begin{align}\label{eq:uev1}
\begin{split}
 (D_{\tau} + \Delta)U &=  (\Delta U - \Uc(\Delta\Rm)) + \Uc((D_{\tau} + \Delta)\nabla\Rm) \\
 &= (\Delta U - \Uc(\Delta\Rm)) + \Uc(J) + \Uc(L).
\end{split}
\end{align}
The estimate \ref{eq:uev}  then follows from the results in Propositions \ref{prop:ucomm}, \ref{prop:uj}, and
\ref{prop:ul}
below.
\end{proof}

The proofs of Propositions and \ref{prop:ucomm}, \ref{prop:uj}, and \ref{prop:ul}  --- particularly the first ---
involve detailed computations. To streamline the exposition, we will use the notation
\[
Y \sim Z
\]
for two tensors $Y$ and $Z$ to mean that
\begin{align*}
  |Y - Z| &\lesssim \Theta_1
    (|M| + |P| + |\nabla P| + |U| + |\nabla U|) \\
    &\phantom{\lesssim} + \Theta_2(|A|
    + |T^0| + |G| + |\nabla A| + |\nabla T^0|),
\end{align*}
where $\Theta_1$ and $\Theta_2$ are as defined in Proposition \ref{prop:uev}.

\subsection{Estimating the commutator of $\Delta$ and $\Uc$.}

\begin{proposition}\label{prop:ucomm}
The tensor $U = \Uc(\nabla \Rm)$ satisfies
\begin{align}
\begin{split}\label{eq:ucomm}
    |\Delta U - \Uc(\Delta \nabla\Rm)| &\lesssim \Theta_1 (|M| + |P| + |\nabla P| + |U| + \nabla U|)\\
    &\phantom{=} + \Theta_2 (|A| + |T^0| + |G| + |\nabla A| + |\nabla T^0|),
\end{split}
\end{align}
where
\begin{gather*}
  \Theta_1 = \Theta_1(|N|, |\Rm|, |\nabla \Rm|), \quad \Theta_2 =\Theta_2(|N|, |\Rm|, |\nabla\Rm|, |\nabla^{(2)}\Rm|).
\end{gather*}

\end{proposition}
\begin{proof} For convenience, let us write $X = \nabla \Rm$.
We will attack the computation by factoring the projection as $U = \Uc(X) = \Pc(\Hc(X))$ where
\[
 \Pc(X)_{aijkl} =  X_{aijkl} - X_{\ua\ui j k\ul} - \frac{1}{m}\left(X_{\bp\bp j k\ul}V_{ai}
  +X_{\bp\ui jk \bp}V_{al} + X_{\ua\bp j k\bp}V_{il}\right)
\]
is the projection defined in Section \ref{sec:pproj} acting on the first, second, and fifth components,
and
\[
 \Hc(X)_{aijkl} = X_{ai\uj\uk l}
\]
is the projection defined in Section \ref{sec:hproj} acting on the third and fourth components.

Let $\tilde{X} = \Hc(X)$. First, we apply \eqref{eq:plaplacian} to $U = \Pc(\tilde{X})$
to obtain
\begin{align}
\begin{split}\label{eq:ucomm1}
&  \Delta U_{aijkl} =  \Pc(\Delta \tilde{X})_{aijkl} + \Cc_{aijkl} + \Cc^{\prime}_{aijkl}\\
     &\phantom{=}+\frac{2}{m}\left(
        \nabla_{\ba}\tilde{X}_{N\ui jk \ul} + \frac{1}{m}\left(\nabla_{\ba}\tilde{X}_{N\bp jk\bp}V_{il} - \nabla_{\ba}\tilde{X}_{\bp\ui jk\bp}N_l - \nabla_{\ba}\tilde{X}_{\bp\bp jk \ul}N_i\right)\right)\\
     &\phantom{=}+ \frac{2}{m}\left(\nabla_{\bi}\tilde{X}_{\ua N jk \ul} + \frac{1}{m}\left(\nabla_{\bi}\tilde{X}_{\bp N jk \bp}V_{al}
     - \nabla_{\bi}\tilde{X}_{\ua\bp jk \bp}N_l - \nabla_{\bi}\tilde{X}_{\bp\bp jk \ul}N_a\right) \right)\\
     &\phantom{=} + \frac{2}{m}\left(\nabla_{\bl}\tilde{X}_{\ua\ui jk N}+ \frac{1}{m}\left(\nabla_{\bl}\tilde{X}_{\bp\bp jk N}V_{ai}
     - \nabla_{\bl}\tilde{X}_{\ua\bp jk \bp}N_i- \nabla_{\bl}\tilde{X}_{\bp\ui jk \bp}N_a\right)\right),
\end{split}
\end{align}
where
\[
 \Cc = N \ast \nabla U + N \ast N \ast U,
\]
and
\begin{align*}
 \Cc^{\prime} &=  (\nabla \tilde{X} + (\Ec^{\prime} + N)\ast \tilde{X}) \ast \Ec^{\prime} + \tilde{X} \ast \Ec^{\prime\prime} =(\nabla X + (\Ec^{\prime} + N)\ast X) \ast \Ec^{\prime} + \Ec^{\prime\prime}\ast X.\\
\end{align*}

We will need to break down the first term on the right and examine
a representative example of each type of term on in the second through the fourth lines. We start
with the latter.

First note that, using \eqref{eq:horizgrad}, we have
\begin{align}
\begin{split}\nonumber
     \nabla_{\ba}\tilde{X}_{\uq\ui jk \ul} &= \nabla_{\ba}X_{\uq\ui\uj\uk\ul}
     -\frac{1}{m}\left(X_{\uq\ui N\uk \ul} V_{aj} + X_{\uq\ui \uj N \ul} V_{ak}\right) \\
     &\phantom{=}-\frac{1}{m}\left(X_{\uq\ui \ba \uk\ul}N_{j} + X_{\uq\ui\uj\ba\ul}N_k\right) + X\ast \Ec^{\prime} \\
     &=  -\frac{1}{m}\left(X_{\uq\ui N\uk \ul} V_{aj} + X_{\uq\ui \uj N \ul} V_{ak}\right)
     + X \ast \Ec^{\prime} + \Rm \ast Q  + N \ast U + \nabla U
\end{split}
\end{align}
where, to obtain the second equality, we have used that
\begin{align*}
 \nabla_{\ba} X_{\uq\ui\uj\uk\ul} &= \nabla_{\uq}X_{\ba\ui\uj\uk\ul}
 - R_{\ba \uq\ui p}R_{p\uj\uk\ul} - R_{\ba \uq \uj p} R_{\ui p \uk\ul}
 - R_{\ba \uq \uk p} R_{\ui \uj p \ul} - R_{\ba \uq \ul p} R_{\ui\uj\uk p}\\
 &= \nabla_{\uq} U_{\ba\ui\uj\uk\ul} + X \ast \Ec^{\prime} + \Rm \ast Q,
\end{align*}
(here, as before, $\nabla_{\uq} X_{\ba\ui\uj\uk\ul} = \nabla_{\uq}U_{\ba\ui\uj\uk\ul} + X \ast \Ec^{\prime}$ as $\nabla_{\uq} H_{bc} = -\nabla_{\uq}V_{bc} = \Ec^{\prime}$), and
\[
  X_{\uq\ui \ba \uk\ul} = -X_{\ba \uq \ui\uk\ul} - X_{\ui\ba \uq \uk\ul} = -U_{\ba \uq \ui\uk\ul} - U_{\ui\ba \uq \uk\ul},
\]
and similarly for the second term in that equality.
Thus, the first term in the second line of in \eqref{eq:ucomm1} satisfies
\begin{align}\label{eq:ucomm121}
 \begin{split}
   \nabla_{\ba}\tilde{X}_{N\ui jk \ul} & =  \nabla_{\ba}\tilde{X}_{\uq\ui jk \ul} N_q\\
   &= -\frac{1}{m}\left(X_{N\ui N\uk \ul} V_{aj} + X_{N\ui \uj N \ul} V_{ak}\right) + N \ast \Rm \ast Q
     + N \ast X\ast \Ec^{\prime}\\
     &\phantom{=}  + N \ast N\ast  U + N \ast \nabla U\\
   &\sim -\frac{1}{m}\left(X_{N\ui N\uk \ul} V_{aj} + X_{N\ui \uj N \ul} V_{ak}\right).
 \end{split}
\end{align}

Next, computing that
\begin{align}\nonumber
 \nabla_{\ba} X_{\uq\bp \uj \uk \bp} &= \nabla_{\ba}\nabla_{\uq}R_{\uj\uk} - \nabla_{\ba}X_{\uq\up\uj\uk\up}\\
 \nonumber
 &= \nabla_{\uq}\nabla_{\ba}R_{\uj\uk} - R_{\ba \uq\uj p}R_{p\uk} - R_{\ba \uq \uk p}R_{\uj p} -\nabla_{\uq} X_{\ba\up\uj\uj\up}\\
 \nonumber
 &\phantom{=}  +R_{\ba \uq\up r} R_{r\uj\uk\up}
 +R_{\ba \uq\uj r} R_{\up r \uk\up} + R_{\ba \uq \uk r} R_{\up\uj r \up} + R_{\ba \uq \up r}R_{\up\uj\uk r}\\
 \label{eq:ucomm2}
 &= \nabla_{\uq} P_{\ba \uj\uk} + X\ast \Ec^{\prime} + \Rm \ast ( Q + M)  + \nabla U,
\end{align}
and that
\begin{align*}
  X_{\uq\bp\ba \uk\bp} &= \nabla_{\uq} R_{\ba\uk} - X_{\uq\up\ba\uk\up}
  = P_{\uq\ba\uk} + U_{\uq\ba\up\uk\up} +  X \ast \Ec^{\prime},
\end{align*}
and similarly for the like terms, we have
\begin{align*}
 \nabla_{\ba}\tilde{X}_{\uq\bp jk\bp}V_{il} &=\nabla_{\ba} X_{\uq\bp \uj \uk \bp} V_{il}-\frac{1}{m}\left(X_{\uq\bp\ba \uk\bp} N_j + X_{\uq\bp\uj \ba \bp}N_k\right)V_{il}\\
 &\phantom{=}-\frac{1}{m}\left(X_{\uq\bp N \uk \bp}V_{aj} + X_{\uq\bp \uj N \bp}V_{ak}\right)V_{il}
 + X\ast \Ec^{\prime}  \\
&= -\frac{1}{m}\left(X_{\uq\bp N \uk \bp}V_{aj} + X_{\uq\bp \uj N \bp}V_{ak}\right)V_{il} + X \ast\Ec^{\prime}  + \Rm \ast (Q+M)\\
&\phantom{=} + \nabla P + N \ast U + \nabla U.
\end{align*}
Thus, the second term in the second line of \eqref{eq:ucomm1} satisfies that
\begin{align}\label{eq:ucomm122}
 \nabla_{\ba} \tilde{X}_{N\bp \uj \uk \bp} &= \nabla_{\ba} \tilde{X}_{\uq\bp \uj \uk \bp} N_q \sim -\frac{1}{m}\left(X_{N\bp N \uk \bp}V_{aj} + X_{N\bp \uj N \bp}V_{ak}\right)V_{il}.
\end{align}

As for the third term,  we have
\begin{align}
\begin{split}\nonumber
 \nabla_{\ba}\tilde{X}_{\bp\ui jk\bp}N_l &=\nabla_{\ba} X_{\bp\ui \uj \uk \bp} N_l-\frac{1}{m}\left(X_{\bp\ui\ba \uk\bp} N_j + X_{\bp\ui\uj \ba \bp}N_k\right)N_{l}\\
 &\phantom{=}-\frac{1}{m}\left(X_{\bp\ui N \uk \bp}V_{aj} + X_{\bp\ui \uj N \bp}V_{ak}\right)N_l
 + N  \ast \nabla\Rm \ast \Ec^{\prime}.
\end{split}
\end{align}
Using the Bianchi identities and \eqref{eq:ucomm2}, we obtain
\[
 \nabla_{\ba} X_{\bp\ui \uj \uk \bp}N_l = -(\nabla_{\ba} X_{\ui \uj \bp \uk \bp} + \nabla_{\ba} X_{\uj \bp\ui \uk \bp})N_l = (\nabla_{\ba} X_{\ui \bp \uj \uk \bp} -\nabla_{\ba} X_{\uj \bp\ui \uk \bp})N_l \sim 0,
\]
which, with the facts that $X_{\bp\ui\ba \uk\bp} = -U_{\bp\ba\ui\uk\bp}$ and $X_{\bp\ui\uj \ba \bp} = -U_{\bp\ba\ui\uj\bp} + U_{\bp\ba\uj\ui\bp}$,
implies that
\begin{align}
\begin{split}\label{eq:ucomm123}
 \nabla_{\ba}\tilde{X}_{\bp\ui jk\bp}N_l &\sim -\frac{1}{m}\left(X_{\bp\ui N \uk \bp}V_{aj} + X_{\bp\ui \uj N \bp}V_{ak}\right)N_l.
\end{split}
\end{align}
Likewise,
we see that the fourth term in the second line  of \eqref{eq:ucomm1} satisfies that
\begin{align}\label{eq:ucomm124}
 \nabla_{\ba}\tilde{X}_{\bp\bp jk \ul}N_i&= \nabla_{\ba}\tilde{X}_{\bp\ul kj \bp}N_i \sim-\frac{1}{m}\left(X_{\bp\ul N \uj \bp}V_{ak} + X_{\bp\ul \uk N \bp}V_{aj}\right)N_i.
\end{align}
Returning to \eqref{eq:ucomm1} and permuting indices in \eqref{eq:ucomm121}, \eqref{eq:ucomm122}, \eqref{eq:ucomm123}, and \eqref{eq:ucomm124}, we then obtain
\begin{align}\label{eq:ucommmain1}
 \begin{split}
  &\Delta U_{aijkl}\sim  \Pc(\Delta \tilde{X})_{aijkl}
    \\
     &  -\frac{2}{m^2}\bigg(X_{N\ui N\uk \ul} V_{aj} + X_{N\ui \uj N \ul} V_{ak}
     +X_{\ua N N \uk \ul} V_{ij}
       + X_{\ua N \uj N \ul} V_{ik} \\
       &\phantom{=}\qquad\qquad+X_{\ua \ui N \uk N} V_{jl} + X_{\ua \ui \uj N N} V_{kl}\bigg)\\
     &  - \frac{2}{m^3}\bigg(\left(X_{N\bp N \uk \bp}V_{aj} + X_{N\bp \uj N \bp}V_{ak}\right)V_{il} +\left(X_{\bp N N \uk \bp}V_{ij} + X_{\bp N \uj N \bp} V_{ik}\right)V_{al}
     \\
     & \qquad\qquad+\left(X_{\bp\bp N\uk N} V_{jl} + X_{\bp\bp \uj N N} V_{kl}\right)V_{ai} \bigg)\\
     &+\frac{2}{m^3}\bigg(
     X_{\bp\ui N\uk \bp}(V_{aj}N_l + V_{jl} N_a) + X_{\bp\ui\uj N\bp}(V_{ak}N_l + V_{kl}N_a)\\
     & \qquad\qquad +X_{\bp\bp N\uk \ul}(V_{aj}N_i + V_{ij}N_a) +X_{\bp\bp \uj N \ul}(V_{ak}N_i +V_{ik}N_a)\\
      & \qquad\qquad + X_{\ua\bp N \uk \bp}(V_{ij}N_l + V_{jl}N_i) + X_{\ua\bp \uj N \bp}(V_{ik}N_l+V_{kl} N_i) \bigg).
 \end{split}
\end{align}

On the other hand, from \eqref{eq:hlap2}, we have
\begin{align*}
\begin{split}
  &\Delta \tilde{X}_{aijkl} = \Delta X_{ai\uj\uk l}
+ (\Ec^{\prime} + N)\ast X\ast \Ec^{\prime} +  X\ast \Ec^{\prime\prime}\\
     &\phantom{=}\quad  - \frac{2}{m}\left(\nabla_{\bj}\tilde{X}_{aiN\uk l} + \nabla_{\bk}\tilde{X}_{ai\uj N l}\right)
     +\frac{2}{m}\left(\frac{2}{m}\tilde{X}_{aiNNl}V_{jk} - \tilde{X}_{ai\uj Nl}N_k - \tilde{X}_{aiN\uk l}N_j\right)\\
     &\phantom{=}\quad -\frac{2}{m}\left(\ \nabla_{\bp}X_{ai\bp\uk l}N_j + \nabla_{\bp}X_{ai \uj\bp l}N_k
     -\frac{1}{m}X_{ai\bp\bp l}N_jN_k
     \right)
\end{split}\\
&\quad \sim \Delta X_{ai\uj\uk l} - \frac{2}{m}\left(\nabla_{\bj}\tilde{X}_{aiN\uk l} + \nabla_{\bk}\tilde{X}_{ai\uj N l} + \tilde{X}_{ai\uj Nl}N_k + \tilde{X}_{aiN\uk l}N_j \right)\\
&\phantom{=}\quad  -\frac{2}{m}\left(\ \nabla_{\bp}X_{ai\bp\uk l}N_j + \nabla_{\bp}X_{ai \uj\bp l}N_k
     -\frac{1}{m}\nabla_aR_{i l}N_jN_k.
     \right) + \frac{2}{m^3}\tilde{X}_{aiNNl}V_{jk}.
\end{align*}
Thus, applying the projection $\Pc$,
\begin{align*}
 &\Pc(\Delta \tilde{X})_{aijkl}\\
 &\quad\sim \Uc(\Delta X)_{aijkl} -\frac{2}{m}\left(U_{ai\uj Nl}N_k + U_{aiN\uk l}N_j -\frac{2}{m}U_{aiNNl}V_{jk}\right)  + \frac{2}{m^2}P_{ail}N_k \\
 &\quad\phantom{=} - \frac{2}{m}\left(\Pc(\nabla\tilde{X})_{\bj aiN\uk l} + \Pc(\nabla\tilde{X})_{\bk ai\uj N l}+ \Pc(\nabla X)_{\bp ai\bp\uk l}N_j + \Pc(\nabla X)_{\bp ai \uj\bp l}N_k \right)\\
  &\quad\sim - \frac{2}{m}\left(\Pc(\nabla\tilde{X})_{\bj aiN\uk l} + \Pc(\tilde{X})_{\bk ai\uj N l} +\Pc(\nabla X)_{\bp ai\bp\uk l}N_j + \Pc(\nabla X)_{\bp ai \uj\bp l}N_k\right),
\end{align*}
where, here, $\Pc$ acts on the arguments corresponding to the indices $a$, $i$, $l$,
e.g.,
\[
\Pc(\nabla\tilde{X})_{\bj aiN\uk l} = \Pc_{ail}^{bcd}\nabla_{\bj}\tilde{X}_{bc N\uk d},
\]
and similarly for the other terms.

Now, according to \eqref{eq:pgradient}, we have
\begin{align}
\begin{split}
-\frac{2}{m}\Pc(\nabla \tilde{X})_{\bj aiN\uk l} &\sim -\frac{2}{m}\nabla_{\bj}U_{aiN\uk l}\\
&\phantom{=}
  +\frac{2}{m^2}\left(X_{N\ui N\uk\ul}V_{ja} + X_{\ua N N\uk \ul}V_{ji} + X_{\ua\ui N\uk N}V_{jl}\right)\\
     &\phantom{=}   +\frac{2}{m^3}\left(X_{N\bp N\uk\bp}V_{ja}V_{il}-X_{\ua\bar{p} N\uk\bar{p}}
     \left(V_{ji}N_l + V_{jl}N_i\right)\right)  \\
     &\phantom{=}+\frac{2}{m^3}\left(X_{\bp N N \uk \bp}V_{ji}V_{al} - X_{\bar{p}\ui N\uk\bar{p}}\left(V_{ja}N_l + V_{jl}N_a\right)\right)
     \\
     &\phantom{=} +\frac{2}{m^3}\left(X_{\bp\bp N \uk N}V_{jl}V_{ai}-X_{\bar{p}\bar{p} N\uk\ul}
     \left(V_{ji}N_a + V_{ja}N_i\right) \right),
\end{split}
\end{align}
and
\begin{align}
\begin{split}
-\frac{2}{m}\Pc(\nabla \tilde{X})_{\bk ai\uj N l} &\sim -\frac{2}{m}\nabla_{\bk}U_{ai\uj N l}\\
&\phantom{=}
  +\frac{2}{m^2}\left(X_{N\ui \uj N\ul}V_{ka} + X_{\ua N \uj N \ul}V_{k i} + X_{\ua\ui \uj N N}V_{kl}\right)\\
     &\phantom{=}   +\frac{2}{m^3}\left(X_{N\bp \uj N\bp}V_{ka}V_{il}-X_{\ua\bar{p} \uj N\bar{p}}
     \left(V_{ki}N_l + V_{kl}N_i\right)\right)  \\
     &\phantom{=}+\frac{2}{m^3}\left(X_{\bp N \uj N \bp}V_{ki}V_{al} - X_{\bar{p}\ui\uj N\bar{p}}\left(V_{ka}N_l + V_{kl}N_a\right)\right)
     \\
     &\phantom{=} +\frac{2}{m^3}\left(X_{\bp\bp \uj N N}V_{kl}V_{ai}-X_{\bar{p}\bar{p} \uj N\ul}
     \left(V_{ki}N_a + V_{ka}N_i\right) \right).
\end{split}
\end{align}
Thus, returning to our above expression and cancelling terms, we see that
\begin{align}\label{eq:ucomm2b}
 \Delta U_{aijkl} \sim  \Pc(\Delta \tilde{X})_{aijkl} - \frac{2}{m}\left(\Pc(\nabla X)_{\bp ai\bp\uk l}N_j + \Pc(\nabla X)_{\bp ai \uj\bp l}N_k\right).
\end{align}

We will need to work a bit harder to see that the remaining two terms on the right of \eqref{eq:ucomm2b} have the form that we claim.
For the term $\Pc(\nabla X)_{\bp ai\bp\uk l}$, note first that
\begin{align}\label{eq:ucomm3}
\begin{split}
 \nabla_{\bp}X_{ai\bp\uk l} &= \nabla_{p} X_{aip\uk l} - \nabla_{\up}X_{ai\up\uk l}\\
 &= \nabla_a X_{p i  p \uk l} - R_{paiq}R_{qp\uk l}
    -R_{papq}R_{iq\uk l} - R_{pa\uk q}R_{i p q l}\\
    &\phantom{=} - R_{palq}R_{ip\uk q} - \nabla_{\up}X_{ai\up\uk l}\\
    &=\nabla_a\nabla_{\uk}R_{li} - \nabla_a\nabla_{l}R_{\uk i} - R_{paiq}R_{qp\uk l}
    -R_{papq}R_{iq\uk l}\\
    &\phantom{=}  - R_{pa\uk q}R_{i p q l} - R_{palq}R_{ip\uk q} - \nabla_{\up}X_{ai\up\uk l}.
\end{split}
\end{align}
Now, the projection $\Pc$ annihilates tensors $Y_{aijkl}$ of the form
\[
  Y_{\ua\ui jk \ul}, \quad Z_{\ua jk} V_{il}, \quad Z_{\ui jk }V_{al}, \quad\mbox{and}\quad Z_{jk\ul}V_{ai}.
\]
We will use $K$ below to denote any term in the kernel of $\Pc$.

First, recall from  \eqref{eq:qdef} that
\[
     R_{a\uk l p} = R_{\ua\uk \ul \up} + (W_{\uk \ul}V_{ap} - W_{\uk \up}V_{al}) + Q_{a\uk l p}
\]
where $W_{ij} = \frac{1}{m}R_{\bp \ui\uj \bp}$.
Thus, for the first term on the third line of \eqref{eq:ucomm3}, we see that
\begin{align*}
   &\nabla_a\nabla_{\uk}R_{li} = \nabla_{\uk}\nabla_a R_{li} -R_{a\uk lp}R_{p i} - R_{a\uk i p}R_{l p}\\
                              &\quad= \nabla_{\uk}\nabla_a R_{li} + \Rm \ast Q -R_{\ua\uk \ul\up}(R_{\up \ui} + M_{\up \bi})
                              -W_{\uk \ul}\left(R_{\ba \ui} + M_{\ba\bi} + \frac{\hat{R}}{m}V_{ai}\right)\\
                              &\quad\phantom{=}+ W_{\uk \up}V_{al}R_{\up i} - R_{\ua\uk \ui \up}(R_{\ul \up} + M_{\bl\up})
                              -W_{\uk\ui}\left(M_{\ul\ba} + M_{\bl\ba} + \frac{\hat{R}}{m}V_{al}\right) +W_{\uk\up}V_{ai}R_{l\up}\\
                              &\quad = \nabla_{\uk}\nabla_a R_{li} +  \Rm \ast (M+ Q) + K.
\end{align*}
But this says that
\begin{equation}\label{eq:ucomm4}
 \Pc(\nabla\nabla\Rc)_{a\uk li} \sim \Pc(\nabla\nabla \Rc)_{\uk a li} \sim \nabla(\Pc(\nabla\Rc))_{\uk ali} = \nabla_{\uk}P_{ali}
\end{equation}
where for the third equality we have used that  $\nabla_{\uk} H_{ij} =\Ec^{\prime} = -\nabla_{\uk}V_{ij}$.

For the second term in the third line of \eqref{eq:ucomm3}, we argue similarly.
If $\Pc^{\prime}$ is the projection $\Pc$ acting on the indices $l$, $k$, $i$, we have
\begin{align*}
 &\nabla_a P_{lki} = \Pc^{\prime}(\nabla\nabla \Rc)_{alki} + X \ast \Ec^{\prime} + N \ast P  \\
  &\phantom{=}
  +\frac{1}{m}\left(\nabla_{N}R_{\uk\ui}V_{al} + \nabla_{\ul}R_{N \ui}V_{ak} + \nabla_{\ul}R_{\uk N}V_{ak}\right)\\
     &\phantom{=}   +\frac{1}{m^2}\left(\nabla_{N}R_{\bp\bp}V_{al}V_{ki}-\nabla_{\ul}R_{\bar{p}\bar{p}}
     \left(V_{ak}N_i + V_{ai}N_k\right)\right)  \\
     &\phantom{=}+\frac{1}{m^2}\left(\nabla_{\bp}R_{N \bp}V_{ak}V_{il} - \nabla_{\bar{p}}R_{\uk\bar{p}}\left(V_{al}N_i + V_{ai}N_l\right)\right)
     \\
     &\phantom{=} +\frac{1}{m^2}\left(\nabla_{\bp}R_{\bp N}V_{ai}V_{lk}-\nabla_{\bar{p}}R_{\bar{p}\ui}
     \left(V_{ak}N_l + V_{al}N_k\right) \right),
\end{align*}
so
\begin{align*}
 \nabla_a P_{l\uk i} &= \Pc^{\prime}(\nabla\nabla\Rc)_{al\uk i} + X \ast \Ec^{\prime} + N \ast P +\frac{1}{m}\nabla_{N}R_{\uk\ui}V_{al}-\frac{1}{m^2}\nabla_{\ul}R_{\bar{p}\bar{p}}V_{ai}N_k  \\
     &\phantom{=}-\frac{1}{m^2}\left(\nabla_{\bar{p}}R_{\uk\bar{p}}\left(V_{al}N_i + V_{ai}N_l\right) +\nabla_{\bar{p}}R_{\bar{p}\ui}V_{al}N_k\right) \\
     &\sim \Pc^{\prime}(\nabla\nabla\Rc)_{al\uk i} + K,
\end{align*}
where $K$ belongs to the kernel of $\Pc$. On the other hand,
\begin{align*}
 \Pc^{\prime}(\nabla\nabla\Rc)_{al\uk i} &= \nabla_a \nabla_{\ul}R_{\uk i} - \nabla_{a}\nabla_{\ul}R_{\uk\ui}
 -\frac{1}{m}\nabla_a\nabla_{\bp}R_{\uk\bp}V_{li}\\
 &=\nabla_a \nabla_{l}R_{\uk i} - \nabla_{\ba}\nabla_{\ul}R_{\uk\ui} + K\\
 &= \nabla_a \nabla_{l}R_{\uk i} - \nabla_{\ul}\nabla_{\ba}R_{\uk\ui} + R_{\ba\ul\uk p}R_{p\ui}
 +R_{\ba\ul\ui p}R_{\uk p} + K\\
 &= \nabla_a \nabla_{l}R_{\uk i} - \nabla_{\ul}P_{\ba\uk\ui} + R_{\ba\ul\uk p}R_{p\ui}
 +R_{\ba\ul \ui p}R_{\uk p} + X \ast \Ec^{\prime} + \Rm \ast Q \\
 &\phantom{\sim}+ \Rm \ast M + K\\
&\sim \nabla_a \nabla_{l}R_{\uk i} + \nabla P +  K.
\end{align*}
Therefore
\begin{equation}\label{eq:ucomm5}
 \Pc(\nabla\nabla\Rc)_{al\uk i} N_j \sim \Pc(\Pc^{\prime}(\nabla\nabla\Rc))_{al\uk i} N_j \sim N \ast \nabla P \sim  0.
\end{equation}

Now we consider the curvature terms in \eqref{eq:ucomm3}. For the first one, as above, we have
\begin{align}\label{eq:ucomm6}
\begin{split}
  R_{paiq}R_{qp\uk l} &= R_{\up a i\uq}R_{\uq\up\uk\ul} + R_{\up ai\bl}W_{\up\uk} - R_{\bl ai\uq}W_{\uq\uk} + \Rm \ast Q\\
  &= (R_{\up\ua\ui \uq} + W_{\up\uq}V_{ai})R_{\uq\up\uk\ul}
   -V_{al}W_{\up\ui}W_{\up\uk}
   +V_{il}W_{\ua\uq}W_{\uq\uk} + \Rm \ast Q \\
   &\sim K.
\end{split}
\end{align}
Likewise, for the second, we have
\begin{align}\label{eq:ucomm7}
\begin{split}
 R_{papq}R_{iq\uk l} &= -R_{aq}R_{iq\uk l}\\
 &= -\left(M_{aq} + R_{\ua\uq} + \frac{\hat{R}}{m}V_{aq}\right)
 \left(Q_{iq\uk l} + R_{\ui\uq\uk\ul} + V_{il}W_{\ua\uk} - V_{ql}W_{\ui\uk}\right)\\
 &\sim -R_{\ua\uq}(R_{\ui\uq\uk\ul} + V_{il}W_{\ua\uk}) + \frac{\hat{R}}{m}V_{al}W_{\ui\uk}\\
 &\sim K.
\end{split}
\end{align}
For the third and fourth curvature terms in \eqref{eq:ucomm3}, we can argue exactly as for the first such term to conclude that
\begin{equation}\label{eq:ucomm8}
  R_{pa\uk q}R_{i p q l}\sim K \quad  \mbox{and} \quad R_{palq}R_{ip\uk q}\sim K.
\end{equation}
Hence, up to terms controlled by the right hand side of \eqref{eq:ucomm}, the quadratic curvature terms
in \eqref{eq:ucomm3}
on the  belong to the kernel of $\Pc$.

Finally, for the last term in \eqref{eq:ucomm3},
using again that $\nabla_{\up}H_{bc} =- \nabla_{\up}V_{bc} = \Ec^{\prime}$, we compute that
\[
  \nabla_{\up}X_{ai\up\uk l} = \nabla_{\up}(X_{ai\up\uk l}) + X\ast \Ec^{\prime}
  = \nabla_{\up}\tilde{X}_{ai \up k l} + X\ast \Ec^{\prime},
\]
and hence obtain that
\begin{align}
\begin{split}\label{eq:ucomm9}
\Pc(\nabla X)_{\up ai\up\uk l} &= \Pc(\tilde{X})_{\up ai\up k l} + X\ast \Ec^{\prime} =
\nabla_{\up}(\Pc(\tilde{X}))_{ai \up k l} + X\ast \Ec^{\prime}\\
&=
\nabla_{\up} U_{ai\up kl} + X\ast \Ec^{\prime}.
\end{split}
\end{align}
Thus, combining \eqref{eq:ucomm4} - \eqref{eq:ucomm9} in \eqref{eq:ucomm3}, we see at last that the
term $\Pc(\nabla X)_{\bp ai\bp\uk l}N_j$ from \eqref{eq:ucomm2b} satisfies
\[
 \Pc(\nabla X)_{\bp ai\bp\uk l}N_j \sim 0.
\]

Since $\Pc(\nabla X)_{\bp ai\uj\bp l} = \Pc(\nabla X)_{\bp al\bp\uj i}$ by the symmetries of $\Rm$ and $\Pc$,
it follows that the last term from \eqref{eq:ucomm2b} satisfies
\[
 \Pc(\nabla X)_{\bp ai\uj\bp l}N_k \sim 0
\]
as well.
Hence, returning to \eqref{eq:ucomm2b} we see that
\begin{align*}
 \Delta U_{aijkl} \sim  \Pc(\Delta \tilde{X})_{aijkl} \sim \Uc(\Delta X)_{aijkl}
\end{align*}
as claimed.
\end{proof}

\subsection{Estimating the reaction terms}

\begin{proposition}\label{prop:uj}
The tensor $J$ defined by \eqref{eq:jdef} satisfies
\begin{align}\label{eq:uj}
\begin{split}
 \Uc(J)_{aijkl} &= \nabla \Rm \ast (Q + M) + \Rm \ast (P + U) \sim 0.
\end{split}
\end{align}
\end{proposition}
\begin{proof}
Recall that we can write $\Uc(J)_{aijkl} = \Pc(J)_{ai\uj\uk l}$
where $\Pc$ is the projection defined above acting on the indices $a$, $i$, and $l$. In our calculations below, we will continue to use $K$ to denote an element of the kernel of the projection $\Uc$.

We start with the equation
\begin{align}\label{eq:nablab}
  \nabla_a B_{i\uj\uk l} &= \nabla_a R_{pi\uj q}R_{p\uk l q} + R_{pi\uj q}\nabla_aR_{p\uk l q}.
\end{align}
Expanding the first term on the right and simplifying, we see that
\begin{align}
\label{eq:uj1}
\begin{split}
& \nabla_a R_{pi\uj q}R_{p\uk l q} =
\nabla_a R_{pi\uj q}(R_{\up\uk\ul\uq} +(V_{pq}W_{\uk\ul} - V_{pl}W_{\uk\uq}) + Q_{p\uk l q})\\
&\qquad = \nabla_a R_{\up i\uj\uq}R_{\up\uk\ul\uq} + \nabla_a R_{\bp i\uj \bp}W_{\uk\ul}
- \nabla_a R_{\bl i \uj \uq}W_{\uk\uq} + \nabla \Rm \ast Q.
\end{split}
\end{align}
Looking more closely at the first term on the right of this equation, we see that
\begin{align*}
 \nabla_{a} R_{i\up \uj\uq}R_{\up\uk\ul\uq} &=
 \nabla_{\ua} R_{\ui\up \uj\uq}R_{\up\uk\ul\uq}
 + \nabla_{\ba} R_{\ui\up \uj\uq}R_{\up\uk\ul\uq}
 \nabla_{\ua} R_{\bi\up \uj\uq}R_{\up\uk\ul\uq}
 +\nabla_{\ba} R_{\bi\up \uj\uq}R_{\up\uk\ul\uq}\\
 &= \Rm\ast U + K,
\end{align*}
whereas for the second and third terms in \eqref{eq:uj1}, we have
\begin{align*}
 \nabla_a R_{\bp i\uj \bp}W_{\uk\ul} &= \nabla_a R_{i\uj}W_{\uk\ul} - \nabla_aR_{\up i\uj\up}W_{\uk\ul}=\nabla_a R_{i\uj}W_{\uk\ul} + \nabla_aR_{i \up \uj\up}W_{\uk\ul} \\
 &= \Rm \ast P + \Rm \ast U  + K,
\end{align*}
and
\begin{align*}
 \nabla_a R_{\bl i \uj \uq}W_{\uk\uq} = (\nabla_a R_{\bl\uq \uj i}-\nabla_a R_{\bl \uj \uq i }) W_{\uk\uq}
  &= \Rm \ast U + K.
\end{align*}
Thus, noting that the second term in the expression \eqref{eq:nablab} is of an analogous form, we have
\begin{align}\label{eq:jfirst}
 \nabla_a B_{a\uj\uk l} &=  \nabla\Rm \ast Q + \Rm\ast U + \Rm \ast P + K
\end{align}
Essentially identical computations then show that
\begin{equation}\label{eq:jmid}
  \nabla_a B_{i\uj i l \uk }  =  \nabla_a B_{i\uk \uj l} =  \nabla\Rm \ast Q + \Rm\ast U + \Rm \ast P + K.
\end{equation}

The final term in $J$ has a somewhat different form than the others. For this term,
we argue as follows, starting from
\begin{equation}\label{eq:jlast1}
\nabla_a B_{i l \uj \uk} = \nabla_a R_{pilq}R_{p\uj\uk q} + R_{pilq}\nabla_a R_{p\uj\uk q}.
\end{equation}
For the first term in \eqref{eq:jlast1}, we have
\begin{align*}
 \nabla_a R_{pilq}R_{p\uj\uk q} &= \nabla_{a} R_{i\bp\bq l}R_{\bp\uj\uk \bq}
 + \nabla_{a} R_{i\up\uq l}R_{\up\uj\uk \uq}\\
 &= (\nabla_aR_{il} -\nabla_aR_{i\up\uq l})W_{\uj\uk} + \nabla_{a} R_{i\up\uq l}R_{\up\uj\uk \uq}
 + \Rm \ast U + \Rm \ast P\\
 &\phantom{=}+ \nabla \Rm \ast Q +K
\end{align*}
On the other hand, for the second term, we have
\begin{align*}
& R_{pilq}\nabla_a R_{p\uj\uk q} \\
&\quad=R_{pilq}\left(U_{ap j k q} + \nabla_{\ua}R_{\up\uj\uk\uq}
+\frac{1}{m}\left(V_{ap}\nabla_{\br}R_{\br\uj\uk \uq}
 + V_{aq}\nabla_{\br}R_{\up\uj\uk \br} + V_{pq}\nabla_{\ua}R_{\br\uj\uk \br}\right)
 \right)\\
&\quad= \Rm\ast U + R_{\up il \uq}\nabla_{\ua}R_{\up\uj\uk\uq}
+ \frac{1}{m}\left(R_{\ba il \uq}\nabla_{\br}R_{\br\uj\uk\uq} +
R_{\up il \ba}\nabla_{\br}R_{\up \uj\uk \br} + R_{\bp il \bp}\nabla_{\ua}R_{\br \uj\uk\br}\right)\\
&\quad=  \Rm\ast U + \nabla\Rm \ast Q + K + \frac{1}{m}(R_{il} - R_{\up il \up})\nabla_{\ua}R_{\br \uj\uk\br}\\
&\quad =  \Rm \ast U + \nabla\Rm \ast Q + \nabla \Rm \ast M + K.
\end{align*}
Returning to \eqref{eq:jlast1} with the above two identities, and combining
them with \eqref{eq:jfirst} and \eqref{eq:jmid}, we obtain that
\begin{align*}
  \Uc(J) & =   \nabla \Rm \ast M + \nabla\Rm \ast Q + \Rm \ast P +  \Rm\ast U
\end{align*}
as claimed.
\end{proof}

\begin{proposition}\label{prop:ul}
The tensor $L$ defined by \eqref{eq:ldef} satisfies
\begin{align}\label{eq:ul}
\begin{split}
 \Uc(L) &= \nabla \Rm \ast (Q + M) + \Rm \ast (P + U) \sim 0.
\end{split}
\end{align}
\end{proposition}
\begin{proof}
As in the proof of Proposition \ref{prop:uj}, we write $\Uc(L)_{aijkl} = \Pc(L)_{ai\uj\uk l}$ and start from the equation
\begin{equation}\label{eq:lident1}
 L_{ai\uj\uk l} = 2\big(R_{iqap}\nabla_p R_{q\uj\uk l} + R_{\uj qap}\nabla_p R_{iq\uk l} + R_{\uk qap}\nabla_pR_{i\uj ql} + R_{lqap}\nabla_pR_{i\uj\uk q}\big).
\end{equation}

Consider the first term on the right in \eqref{eq:lident1}. Using the identity from the proof of Proposition \ref{prop:uj}, we see that
\begin{align*}
&R_{iqap}\nabla_p R_{q\uj\uk l}\\
&\quad=R_{iqap}\left(U_{pq j k l} + \nabla_{\up}R_{\uq\uj\uk\ul}
+\frac{1}{m}\left(V_{pq}\nabla_{\br}R_{\br\uj\uk \ul}
 + V_{pl}\nabla_{\br}R_{\uq\uj\uk \br} + V_{ql}\nabla_{\up}R_{\br\uj\uk \br}\right)
 \right)\\
 &\quad= U \ast \Rm + R_{i\uq a \up} \nabla_{\up}R_{\uq\uj\uk\ul} +
 \frac{1}{m}\left(R_{i\bp a\bp}\nabla_{\br}R_{\br\uj\uk\ul}
 +R_{i\uq a \bl}\nabla_{\br}R_{\uq\uj\uk\br} +R_{i\bl a \up}\nabla_{\up} R_{\br\uj\uk \br}\right)\\
 &\quad= \nabla \Rm \ast Q + \Rm\ast U + K
 + \frac{1}{m}(R_{ia} -R_{i\up a \up})\nabla_{\br}R_{\br\uj\uk\ul},
\end{align*}
where here, again, $K$ denotes an element of the kernel of $\Uc$. Hence
\begin{equation}\label{eq:lident2}
 R_{iqap}\nabla_p R_{q\uj\uk l} = \nabla \Rm\ast M + \nabla \Rm\ast Q + \Rm\ast U + K.
\end{equation}

Similarly, for the second term in \eqref{eq:lident1}, we see that
\begin{align*}
R_{\uj qap}\nabla_p R_{iq\uk l} &= (R_{\uj \uq\ua \up}
+ V_{aq}W_{\uj\up} - V_{pq}W_{\uj \ua}) \nabla_p R_{iq\uk l} + Q \ast \nabla \Rm\\
&= R_{\uj \uq\ua \up} \nabla_{\up} R_{i\uq\uk l}
 +\nabla_{\up} R_{i\ba\uk l}W_{\uj\up} -  \nabla_{\bp} R_{i\bp\uk l}W_{\uj\ua}.
\end{align*}
Now,
\[
 R_{\uj \uq\ua \up} \nabla_{\up} R_{i\uq\uk l} = \Rm \ast U + K,
\]
and
\begin{align*}
 \nabla_{\up} R_{i\ba\uk l}W_{\uj\up} & = (\nabla_{\ba}R_{i
 \up\uk l} -\nabla_i R_{\ba\up\uk l})W_{\uj\up} = \Rm \ast U  + K,
\end{align*}
while
\begin{align*}
 \nabla_{\bp} R_{i\bp\uk l}W_{\uj\ua} &= (\nabla_{\uk}R_{li} - \nabla_lR_{\uk i}
 - \nabla_{\up}R_{i\up\uk l})W_{\uj\ua} = \Rm \ast P + \Rm \ast U + K,
\end{align*}
so this term satisfies
\begin{equation}\label{eq:lident3}
 R_{\uj qap}\nabla_p R_{iq\uk l} = \nabla \Rm \ast Q + \Rm \ast P  +\Rm\ast U + K.
\end{equation}
But the third term in \eqref{eq:lident1} can be treated in the same way as the second term since
\[
 R_{\uk qap}\nabla_pR_{i\uj ql} = R_{\uk qap}\nabla_pR_{lq \uj i},
\]
and the fourth term is analogous to the first. Thus we conclude that
\[
 L = \nabla \Rm\ast M + \nabla \Rm \ast Q + \Rm \ast P  + \Rm \ast U + K,
\]
and \eqref{eq:ul} follows.
\end{proof}

\subsection{The PDE-ODE System}
We now formally group our connection and curvature invariants together in a single structure. Let
\begin{equation*}
    \Xc \dfn T_2(M) \oplus T_3(M) \oplus T_5(M),
\end{equation*}
and
\begin{equation*}
    \Yc \dfn T_2(M) \oplus T_3(M)\oplus T_3(M) \oplus T_4(M) \oplus T_4(M),
\end{equation*}
and define the sections
\begin{equation}\label{eq:xydef}
    \ve{X} \dfn (M, P, U), \quad \ve{Y} \dfn (G, A, T^0, \nabla A, \nabla T^0),
\end{equation}
of $\Xc$ and $\Yc$ in terms of the invariants $A$, $T^0$, $G$, $M$, $P$, and $U$ associated to the distributions $\Hc$ and $\Vc$.

The net result of the computations in the preceding sections is that (along an arbitrary smooth solution to the Ricci flow, and with respect to an arbitrary pair of evolving orthogonal distributions $\Hc(\tau)$ and $\Vc(\tau)$), the sections $\ve{X}$ and $\ve{Y}$ satisfy a \emph{closed} system of inequalities in the following sense.
\begin{theorem}\label{thm:xysys}
Suppose that $g(\tau)$ is a smooth solution the backward Ricci flow on $M\times [0, T]$. Let $\Hc(\tau)$ and $\Vc(\tau)$
be complementary orthogonal distributions evolving by \eqref{eq:hvev} with $\dim \Vc(\tau) = m$. Then, on $M\times [0, T]$,
we have
\begin{align}
\begin{split}\label{eq:xysys}
 |D_{\tau} \ve{X} + \Delta \ve{X}| &\leq
 \Theta_1 (|\ve{X}| + |\nabla\ve{X}|) + \Theta_2|\ve{Y}|\\
 |D_{\tau} \ve{Y}|&\leq C(|\ve{X}| + |\nabla \ve{X}|) + \Theta_2|\ve{Y}|,
\end{split}
\end{align}
for some constant $C = C(n)$ and polynomials
\[
 \Theta_1 = \Theta_1(|N|, |\Rm|, |\nabla \Rm|),
 \quad \Theta_2 = \Theta_2(|N|, |\Rm|, |\nabla \Rm|,
 |\nabla^2\Rm|, |A|, |T^0|),
\]
satisfying the conditions of Section \ref{ssec:notation}.
\end{theorem}
\begin{proof}
This follows directly from the estimates proven in Propositions \ref{prop:aev}, \ref{prop:t0ev}, \ref{prop:dadt0}, \ref{prop:gev},
\ref{prop:mev}, \ref{prop:pev}, and \ref{prop:uev}.
\end{proof}

\section{Backward propagation of warped-product structures under the flow}

We now turn to the proofs of Theorems \ref{thm:warped} and \ref{thm:warped2}. The main analytic ingredient is the following general backward uniqueness
principle proven in Theorem 1.1 of \cite{KotschwarRFBU} (cf. Theorem 3 of \cite{KotschwarFrequency}). Here $\Xc$ and $\Yc$ denote
direct sums of tensor bundles over $M$ with metrics and connections induced by the solution $g$.

\begin{theorem}[\cite{KotschwarRFBU, KotschwarFrequency}]\label{thm:rfbu}
Suppose that $g = g(\tau)$ is a smooth complete solution to \eqref{eq:brf} on $M\times [0, \Omega]$ of uniformly bounded curvature.
Assume that $\ve{X} = \ve{X}(\tau)$ and $\ve{Y} = \ve{Y}(\tau)$ are smooth, uniformly bounded families of sections of $\Xc$ and $\Yc$
satisfying the system
\begin{align}\label{eq:rfpdeode}
\begin{split}
\left|D_{\tau}\ve{X} + \Delta \ve{X}\right| &\lesssim \left(|\ve{X}| + |\nabla \ve{X}| + |\ve{Y}|\right)\\
\left|D_{\tau}\ve{Y}\right| &\lesssim \left(|\ve{X}| + |\nabla \ve{X}| + |\ve{Y}|\right),
\end{split}
\end{align}
on $M\times [0, \Omega]$ and that $\ve{X}(0) \equiv 0$ and $\ve{Y}(0) \equiv 0$
on  $M$. Then $\ve{X}(\tau) \equiv 0$ and $\ve{Y}(\tau) \equiv 0$ on $M\times [0, \Omega]$.
\end{theorem}

We will verify shortly that, under the assumptions of Theorems \ref{thm:warped} and \ref{thm:warped2},
the system consisting of $\ve{X}=(M, P, U)$ and  $\ve{Y}=(A, T^0, G, \nabla A, \nabla T^0)$ defined in terms of the invariants above
satisifies the hypotheses of Theorem \ref{thm:rfbu}.  Before we do so, we first observe a few simple consequences of the vanishing of this particular choice of $\ve{X}$ and $\ve{Y}$.

\subsection{Some consequences of the vanishing of $\ve{X}$ and $\ve{Y}$.}

Our first observation is that if $\Hc(\tau)$ and $\Vc(\tau)$ are complementary orthonormal distributions evolving along the flow according to \eqref{eq:vev}-\eqref{eq:hvev},
and the sections $\ve{X}(\tau)$ and $\ve{Y}(\tau)$ defined in terms of $\Hc(\tau)$ and $\Vc(\tau)$ vanish identically, then $\Hc(\tau)$ and $\Vc(\tau)$
are actually independent of $\tau$.
\begin{lemma}\label{lem:projconst}
 Suppose that $g(\tau)$ is a solution to \eqref{eq:brf} on $M\times [0, \Omega]$, and $V = V(\tau)$, $H = H(\tau)\in \operatorname{End}(TM)$ are smooth families
 of complementary orthogonal projections evolving according to \eqref{eq:vev} - \eqref{eq:hev}.  If the tensor $M = M(\tau)$ vanishes identically, then $V(\tau) \equiv V(0)$ and $H(\tau) \equiv H(0)$.
\end{lemma}
\begin{proof}
The vanishing of $M$ implies that the endomorphism $\Rc:TM\longrightarrow TM$ has a block diagonal decomposition with respect to the orthogonal direct sum decomposition $TM = \Hc\oplus \Vc$. In particular, $\Rc$ commutes with $H$ and $V$. Thus, returning to the system \eqref{eq:vev} - \eqref{eq:hev}, we see that
\[
 \partial_{\tau} V = D_{\tau}V -\Rc \circ V + V\circ \Rc \equiv 0,
\]
and
\[
\partial_{\tau} H = -\partial_{\tau} V \equiv 0,
\]
on $M\times [0, \Omega]$. Consequently, $H(\tau) \equiv H(0)$ and $V(\tau) \equiv V(0)$ for all $\tau \in [0, \Omega]$.
\end{proof}

The second observation is that, in the setup above, the vanishing of $\ve{X}$ and $\ve{Y}$ imply that the vertical trace $\hat{R} = \trace_V(\Rc)$ of $\Rc$
is locally constant on the fibers.

\begin{lemma}\label{lem:rhatfiber}
 Suppose that $g(\tau)$ is a solution to the backward Ricci flow, and $V(\tau)$ and $H(\tau)$ are complementary evolving orthogonal projections as in Lemma \ref{lem:projconst}.  Assume that the associated tensors $A$, $T^0$, $M$, and $P$ vanish identically. Then $(\nabla\hat{R})^V \equiv 0$.
\end{lemma}
\begin{proof} The vanishing of $A$ and $T^0$ imply the vanishing of the tensor $\Ec^{\prime}$ given by \eqref{eq:eprimedef}, and thus we have
\[
   \nabla_{i} V_{ab} = \frac{1}{m}(V_{ia} N_b + V_{ib} N_a)
\]
on all of $M$. Then
\begin{align}
    (\nabla\hat{R})^V_i &= \nabla_{\bi} \hat{R} = \nabla_{\bi}(R_{ab}V_{ab}) = \nabla_{\bi}R_{\bp\bp} + \frac{1}{m}\left(V_{ia}N_b + V_{ib}N_a\right)R_{ab}\\
    &= P_{\bi\bp\bp} + \frac{2}{m}M_{\bi N} = 0,
\end{align}
so $(\nabla\hat{R})^V \equiv 0$ as claimed.
\end{proof}

\subsection{Proofs of Theorems \ref{thm:warped} and \ref{thm:warped2}}
Our task now is to assemble the pieces we have established above into a proof. Let us first work under the assumptions of Theorem \ref{thm:warped2}, as Theorem \ref{thm:warped} is essentially a special case of its statement. Suppose $g(\tau)$ is a solution
 to \eqref{eq:brf} on $M\times [0, \Omega]$ such that $(M, g(\tau))$ is complete for each $\tau\in [0, \Omega]$
 and
 \[
     \sup_{M\times [0, \Omega]} |\Rm|(x, \tau) \leq K_0
 \]
for some constant $K_0$.  Assume further that we are given
a Riemannian submersion $\pi:(M, g(0))\longrightarrow (B, \ch{g}_{0})$ which is (everywhere) a locally-warped product
with connected Einstein fibers of bounded mean curvature.  Since $(M, g(0))$ is complete, it follows from \cite{Hermann} that $\pi:M \longrightarrow B$ is a fiber bundle,
so each $b\in B$ has a neighborhood $U$ such that $\pi^{-1}(U) \approx U \times F$ for some fixed smooth manifold $F$. As above, write $m = \operatorname{dim}(F)$.

Now let $\Vc_{0} = \ker d\pi$ and $\Hc_0 = \Vc_0^{\perp}$ be the vertical and horizontal distributions associated to $\pi$, and $V_0$ and $H_0$ the the orthogonal projections onto those distributions.  Let $V = V(\tau)$ and $H = H(\tau)$ be the family of projections evolving according to \eqref{eq:vev} - \eqref{eq:hev}
with $V(0) = V_0$ and $H(0) = H_0$, and let $\Vc(\tau) \dfn V(\tau)(TM)$ and $\Hc(\tau) \dfn H(\tau)(TM)$ be the families of complementary $g(\tau)$-orthogonal distributions that are the images of those projections for each $\tau$. Finally, define the families of tensors $A$, $T^0$, $G$, $M$, $P$, and $U$ in terms of $H$ and $V$ as above,
and let
\[
\ve{X} \dfn (M, P, U) \quad\mbox{and} \quad \ve{Y} \dfn (A, T^0, G, \nabla T^0, \nabla A).
\]

As  $\Hc_0$ and $\Vc_0$ are the horizontal and vertical distributions associated to a locally-warped product with Einstein fibers,
it follows from Proposition \ref{prop:wpmpu} that $\ve{X}(0) \equiv 0$ and $\ve{Y}(0) \equiv 0$. Moreover, by Theorem \ref{thm:xysys},  $\ve{X}$ and $\ve{Y}$ satisfy the system \eqref{eq:xysys} on $M\times [0, \Omega]$ for some universal constant $C$ and some polynomials $\Theta_1$, $\Theta_2$ as described
in Section \ref{ssec:notation}.

We verify next that, under our hypotheses, the sections $\ve{X}$ and $\ve{Y}$ and the polynomial expressions $\Theta_1(|N|, |\Rm|, |\nabla \Rm|)$ and $\Theta_2(|N|, |\Rm|, |\nabla \Rm|,
 |\nabla^2\Rm|, |A|, |T^0|)$ in \eqref{eq:xysys}
will be uniformly bounded on $M\times [0, \Omega -\delta]$
for any $\delta > 0$. Indeed, the standard derivative estimates \cite{Shi} for \eqref{eq:rf} imply uniform bounds  $|\nabla^{(k)}\Rm|$ on $M\times [0, \Omega_{\delta}]$ for all $k$
where $\Omega_{\delta} \dfn \Omega -\delta$. Thus $M$, $P$, and $U$ are uniformly bounded. Moreover, since $A$ and $T^0$ vanish identically at $\tau = 0$ and $N$ is assumed to be bounded at $\tau=0$, we have uniform bounds on these tensors for $M\times [0, \Omega_{\delta}]$ in view of the evolution equations \eqref{eq:aev},
\eqref{eq:nev}, and \eqref{eq:t0ev}. The bounds on $A$ and $T^0$ then imply via \eqref{eq:nablah} that $\nabla H$ and $\nabla V = -\nabla H$ are uniformly bounded on $M\times [0, \Omega_{\delta}]$. With the bounds on $\Rm$, $\nabla \Rm$ and $\nabla^{(2)}\Rm$, it follows then that $\nabla M$, $\nabla P$, and $\nabla U$ are uniformly bounded on the same set. Using all of the above bounds we can then successively bound $\nabla A$, $\nabla T^0$, and $G$ via \eqref{eq:daev}, \eqref{eq:dtev},
and \eqref{eq:gev}.

Thus, for all $\delta > 0$, there is a constant $C$ depending on $\delta$ and $K_0$
such that
\[
 |\ve{X}| + |\nabla\ve{X}| + |\ve{Y}| \leq C,
\]
and
\begin{align}\label{eq:xysys3}
\begin{split} 
\left|D_{\tau}\ve{X} + \Delta \ve{X}\right| &\leq C\left(|\ve{X}| + |\nabla \ve{X}| + |\ve{Y}|\right)\\
\left|D_{\tau}\ve{Y}\right| &\leq C\left(|\ve{X}| + |\nabla \ve{X}| + |\ve{Y}|\right),
\end{split}
\end{align}
on $M\times [0, \Omega_{\delta}]$. Applying Theorem \ref{thm:rfbu}, we conclude that $\ve{X}(\tau)\equiv 0$ and $\ve{Y}(\tau)\equiv 0$ on $M\times [0, \Omega_{\delta}]$,
Then sending $\delta \longrightarrow 0$, we see that $\ve{X} \equiv 0$ and $\ve{Y}\equiv 0$ on all of $M\times [0, \Omega]$.

Now, according to Lemma \ref{lem:projconst}, the projections $\Hc(\tau) = \Hc(\tau)(TM)$ and $\Vc(\tau) = V(\tau)(TM)$ are independent of $\tau$.
In particular, $\Vc(\tau) \equiv \Vc_0 = \operatorname{ker} d\pi$.  We claim that there is a family of metrics $\ch{g}(\tau)$ on $B$ such that $\pi:(M, g(\tau))\longrightarrow (B, \ch{g}(\tau))$ is a Riemannian submersion for all $\tau\in [0, \Omega]$.

To see this, note that, from the fact that $\ve{Y}\equiv 0$, we  have $A \equiv 0$, $T^0\equiv 0$, and $G\equiv 0$.  Let $U\subset B$ be an open neighborhood
for which $\pi$ admits a local trivialization $\varphi: \pi^{-1}(U)\longrightarrow U\times F$. Let $p = \varphi^{-1}(b, x) \in \pi^{-1}(U)$ and
suppose that $\ch{X}$ and $\ch{Y}$ are arbitrary smooth vector fields on $B$ with horizontal lifts $X$ and $Y$ on $\pi^{-1}(U)$. Let $W$ be an arbitrary vertical vector field on $\pi^{-1}(W)$.
Then
\begin{align*}
  W(g(X, Y)) &= g(\nabla_W X, Y) + g(X, \nabla_W Y) \\
  &= g([W, X], Y) + g([W, Y], X) + g(A_X W, Y) + g(A_Y W, X) \equiv 0
\end{align*}
as $A\equiv 0$ and $[W, X]$ and $[W, Y]$ are  vertical. (To see the latter, note, e.g., that
\[
\pi_{*}[W,  X] = [\pi_*W, \pi_*X] = [0, \ch{X}] = 0.
\]
so $[W, X] \in \operatorname{ker} d\pi$.)
Since $F$ is connected, it follows that the value of $g(X, Y)$ on $\pi^{-1}(U) \approx U\times F$ is independent of $x\in F$.

Now, given a point $b\in B$, and vector fields $\ch{X}$, $\ch{Y}$ defined in a neighborhood $U$ of $b$ as above, there are unique horizontal vector fields
$X$ and $Y$ on $\pi^{-1}(U)$ that are $\pi$-related to  $\ch{X}$ and $\ch{Y}$. We may then define
\begin{equation}\label{eq:chgdef}
   \ch{g}(b, \tau)(\ch{X}, \ch{Y}) \dfn g(b, x, \tau)(X, Y)
\end{equation}
for any $x \in \pi^{-1}(b) \approx \{b\}\times F$. As the right-hand side of \eqref{eq:chgdef} is independent of $x$, $\ch{g} = \ch{g}(\tau)$ is a well-defined family of Riemannian metrics on $B\times [0, \Omega]$. By the construction of $\ch{g}(\tau)$, $\pi: (M, g(\tau)) \longrightarrow (B, \ch{g}(\tau))$ is a Riemannian submersion. To see this, fix any $p\in M$ and $X$, $Y\in \Hc_p \subset T_pM$. Then $X$ and $Y$
are the unique horizontal lifts of $d\pi_p(X)$, $d\pi_p(Y)\in T_{\pi(p)}B$ to $T_pM$ and so, by definition,
\[
 \ch{g}(\pi(p), \tau)(d\pi_p X, d\pi_p Y) = g(p, \tau)(X, Y).
\]
Thus $\pi$  is Riemannian submersion with respect to the families of metrics $g(\tau)$ on $M$ and $\ch{g}(\tau)$ on $B$ for all $\tau\in [0, \Omega]$.

Then it follows from Lemma \ref{lem:wpconditions} and the further vanishing of $T^0$ and $G$ that, on any neighborhood $U\subset B$ which admits a trivialization
$\varphi:\pi^{-1}(U)\longrightarrow U\times F$ as above, there is a smooth family of positive functions $h(\tau)$ on $U$ and metrics $\hat{g}(\tau)$ on $F$ such that $g$ at least
admits a representation of the form
\[
g(b, x, \tau) = \pi^*\ch{g}(b, \tau) + h^2(b, \tau) \hat{g}(x, \tau).
\]
We can be more precise about the structure of the second term.

Let us continue to work on $\pi^{-1}(U)\approx U\times F$, identifying $g$ with $(\varphi^{-1})^*g$ on $U\times F$, and $\pi$ with the projection $U\times F \longrightarrow U$. By assumption, $F$ admits an Einstein metric $\bar{g}$ such that
\[
  g(b, x, 0) = \pi^*\ch{g}(b) + h^2_0(b) \bar{g}(x)
\]
on $U\times F$.
Since $M \equiv 0$, we have
\[
   \pdtau g_{ij} = 2R_{ij} = 2\left(\Rc^H_{ij} + \frac{\hat{R}}{m}V_{ij}\right)
\]
on $(U\times F)\times [0, \Omega]$,
where, as above, $V_{ij} = g^V_{ij}$ is the two-tensor obtained from the endomorphism $V$ by lowering an index relative to $g$. In particular (using that the endomorphism $V_{\alpha}^{\beta}$ is time-independent),
\begin{equation}\label{eq:vertev}
   \pdtau V_{ij} = \frac{2}{m}\hat{R}V_{ij}, \quad V_{ij}(0) = h^2_0 \bar{g}_{ij}.
\end{equation}
Now, by Lemma \ref{lem:rhatfiber}, and the fact that $F$ is connected, $\hat{R}$ is a function only of $b\in U$ and $\tau\in [0, \Omega]$.
Thus if we define $h\in C^{\infty}(U\times [0, \Omega])$ by
\begin{equation}\label{eq:wfdef}
    h(b, \tau) = h_0(b) \exp\left(\frac{1}{m}\int_0^{\tau} \hat{R}(b, s)\,ds\right),
\end{equation}
then
$k = k(b, x, \tau) \dfn h^2(b, \tau)\bar{g}(x)$ satisfies
\[
   \pdtau k_{ij} = \frac{2}{m} \hat{R}k_{ij}, \quad k_{ij}(0) = h^2_0 \bar{g}_{ij}.
\]
Comparing with \eqref{eq:vertev}, we see that, for each fixed $(b, x) \in U\times F$, $V_{ij}(b, x, \tau)$ and $k_{ij}(b, x, \tau)$
satisfy the same ODE in $\tau$ with the same initial data. Hence $k_{ij} = V_{ij}$ on all of $(U\times F) \times [0, \Omega]$, that is,
\[
    g(b, x, \tau) = \pi^*\ch{g}(b, \tau) + h^2(b, \tau) \bar{g}(x)
\]
on $\pi^{-1}(U)\times [0, \Omega]$,
where $h$ is given by \eqref{eq:wfdef}. This completes the proof of Theorem \ref{thm:warped2}.

For Theorem \ref{thm:warped}, by working on individual connected components, we may assume that $B$ and $F$ are connected. Then
we may simply apply the above argument with $U = B$ to obtain the corresponding global representation. \qed

\subsection{Preservation of multiply-warped products}

With the help of the following lemma, we can generalize Theorem \ref{thm:warped} to multiply-warped products with Einstein fibers.
\begin{lemma}\label{lem:2wp}
 Let $M = B \times F_1 \times F_2$ where $B$ is a smooth manifold of dimension $p$ and $(F_i, \hat{g}_i)$ are Riemannian manifolds of dimension $m_i$. Assume that $g$ is a Riemannian metric on $M$
 such that
 \begin{equation}\label{eq:2wp}
     g(b, x, y) = \pi_1^*\ch{g}_1(b, x) + h_1^2(b, x)\hat{g}_2(y) = \pi_2^*\hat{g}_2(b, y) + h_2^2(b, y)\hat{g}_1(x),
 \end{equation}
where $\ch{g}_i$ and $h_i$, $i=1, 2$, are metrics and positive functions, respectively, on $B\times F_i$, and $\pi_i:M\longrightarrow B\times F_i$ are the projections. Then
\[
h_1(b, x) = h_1(b), \quad  h_2(b, y) =h_2(b),
\]
and
\[
     g(b, x, y) = \pi^*\ch{g}(b) + h_1^2(b)\hat{g}_1(x) + h_2^2(b)\hat{g}_2(y),
\]
for some Riemannian metric $\ch{g}$ on $B$ where  $\pi:M\longrightarrow B$ is the projection.
\end{lemma}
\begin{proof} By passing to connected components, we may assume that the manifolds $B$, $F_1$, and $F_2$ are all connected.
Fix any $p_0 = (b_0, x_0, y_0)\in M$ and let $\{B_i\}_{i=1}^p$  be a local frame defined on a connected neighborhood $V\subset B$ about $b_0$.
Then let  $\{X_i\}_{i=1}^{m_1}$ and $\{Y_i\}_{i=1}^{m_2}$ be local $\hat{g}_1$-orthonormal and $\hat{g}_2$-orthonormal frames, respectively, on connected neighborhoods $W_1\subset F_1$ of $x_0$ and $W_2\subset F_2$ of $y_0$.
We will regard the elements of these frames as vector fields on the the neighborhood $U = V\times W_1\times W_2$ of $p_0$.

From \eqref{eq:2wp}, we have
\[
\pi_1^*\ch{g}_1(B_i, X_j) \equiv 0, \quad \pi_2^*\ch{g}_2(B_i, Y_j) \equiv 0,
\]
and
\[
 h_1^2(b, x)\delta_{ij} = \ch{g}_1(b, y)(Y_i, Y_j), \quad  h_2^2(b, y)\delta_{ij} = \ch{g}_2(b, x)(X_i,  X_j),
\]
on $U$. The latter says that $h_1$ and $\ch{g}_2(X_i, X_j)$ are independent of $x$ on $U_0\times V_0$, and that $h_2$ and $\ch{g}_1(Y_i, Y_j)$ are independent of $y$ on $U_0\times W_0$. In particular, we have $h_1(b, x) = h_1(b, x_0)$ and $h_2(b, y) = h_2(b, y_0)$.
Additionally, we have that
\[
    \ch{g}_{1}(b, x)(B_i, B_j)= \ch{g}_2(b, y)(B_i, B_j),
\]
so that $\ch{g}_1(b, x)(B_i, B_j) = \ch{g}_1(b, x_0)(B_i, B_j)$.

The above argument shows that $h_1$, and $h_2$ are locally independent of $x$ and $y$, respectively, and that that $\ch{g}_1|_{TB}$ is locally independent of $x$. Moreover, the subspaces of $TM$ tangent to the factors $B$, $F_1$, and $F_2$ are orthogonal.
Thus, as $M$ is connected, we have
\begin{gather*}
   h_1(b, x) = h_1(b, x_0) \dfn h_1(b), \quad h_2(b, y) = h_2(b, y_0) \dfn h_2(b),
\end{gather*}
and
\begin{gather*}
   \ch{g}_1(b, x)|_{TB} = \ch{g}_1(b, x_0)|_{TB},
\end{gather*}
on all of $M$. Thus if we define the metric $\ch{g}$ on $B$ by $\ch{g}_b(E_1, E_2) = \ch{g}_1(b, x_0)(E_1, E_2)$ (again identifying the vectors $E_1$ and $E_2$ with their horizontal lifts to $M$) we have
\[
    g = \pi^*\ch{g} + h_1^2\hat{g}_1 + h_2^2\hat{g}_2
\]
on $M$ as claimed.
\end{proof}

Now we are ready to prove Corollary \ref{cor:mwp}.

\begin{proof}[Proof of Corollary \ref{cor:mwp}]
For $i=1, 2, \ldots, k$, we may regard the multiply-warped product metric
\[
       g(0) = \pi^*\ch{g}_0+ h_{1}^2\bar{g}_1 + \cdots + h_{k}^2\bar{g}_k
\]
as a singly-warped product metric
\[
        \pi_i^*\ch{g}_i + h_i^2\bar{g}_i
\]
on $B_i \times F_i$ where $B_i = B\times F_1 \times \cdots \times \widehat{F}_{i}\times \cdots \times F_k$,
$\pi_i: M\longrightarrow B_i$ is the projection, and
\[
   \ch{g}_{k, 0} = \pi^*\ch{g}_0 + h_1^2\bar{g}_1 + \cdots + \widehat{h_{i}^2\bar{g}_{i}} + \cdots + h_{k}^2\bar{g}_{k}.
\]
Applying Theorem \ref{thm:warped} to each of these representations, we obtain families of metrics $\ch{g}_i(\tau)$ and positive functions $h_i$ on $B_k\times [0, \Omega]$
for $i=1, 2, \ldots, k$ such that $\ch{g}_i(0) = \ch{g}_{i, 0}$ and $h_i(0) = h_i$ and
\[
   g(\tau) = \pi_i^*\ch{g}_i(\tau) + h_i^2(\tau)\bar{g}_i
\]
on $M\times [0, \Omega]$. The claim then follows by applying Lemma \ref{lem:2wp} inductively to these representations.
\end{proof}

\section{An application to asymptotically conical shrinkers}
Now we apply the framework established above to study the ends of asymptotically conical shrinkers. The idea is to reduce the statement of Theorem \ref{thm:soliton} to a parabolic problem of backward uniqueness, in which
the end of the cone and the end of the shrinker (or isometric copies thereof) are realized, respectively, as the initial and terminal time-slices of a common smooth backward Ricci flow.

\subsection{Reduction to a parabolic problem}
We will recall from \cite{KotschwarWangIsometries} the following summary of the details of the  normalizations
made in Section 2 of \cite{KotschwarWangConical}.

\begin{proposition}[Proposition 2.1, \cite{KotschwarWangConical}]
\label{prop:brf}
Suppose the shrinker $(M, \tilde{g}, \tilde{f})$ is asymptotic to $\Cc^{\Sigma}$ along the end $V\subset (M, \tilde{g})$.
Then there exists $r_0 > 0$ and a diffeomorphism $F: \Cc_{r_0}\longrightarrow W$ onto an end $W\subset V$
such that $\bar{g} = F^*\tilde{g}$ and $\bar{f} = F^*\tilde{f}$ satisfy the following properties.
\begin{enumerate}

\item[(1)] The solution $\Phi = \Phi_{\tau}(x) = \Phi(x, \tau)$ to the ODE
 \begin{equation}
 \label{eq:phisys}
  \frac{d\Phi}{d\tau} = -\frac{1}{\tau}\delb \bar{f} \circ \Phi, \quad \Phi_1 = \operatorname{Id},
 \end{equation}
is well-defined on $\Cc_{r_0}^{\Sigma}\times (0, 1]$,
and the maps $\Phi_{\tau}:\Cc_{r_0}^{\Sigma}\longrightarrow \Cc_{r_0}^{\Sigma}$ are each injective local diffeomorphisms for $\tau \in (0, 1]$.

\item[(2)] The family of metrics $g(\tau) = \tau \Phi_{\tau}^*\bar{g}$ is a smooth solution to \eqref{eq:brf} on $\Cc_{r_0}^{\Sigma}\times (0, 1]$
and converges smoothly to $\hat{g}$ on $\overline{\Cc_{a}^{\Sigma}}$ for all $a > r_0$ as $\tau \longrightarrow 0$. Moreover, there is a constant $K_0$ such that
\begin{align}
\label{eq:curvdecay}
\sup_{\Cc_{r_0}^{\Sigma}\times [0,1]} \left(r^{m+2}+1\right)|\nabla^{(m)}\Rm(g(\tau))| & \leq K_0.
\end{align}
Here $|\cdot| = |\cdot|_{g(\tau)}$ and $\nabla = \nabla_{g(\tau)}$ denote the norm and the Levi-Civita connection associated to the metric $g = g(\tau)$.

\item[(3)] If $f$ is the function on $\Cc_{r_0}^{\Sigma}\times (0,1]$ defined by 
$f(\tau)=\Phi_\tau^\ast \bar{f}$, then $\tau f$ converges smoothly as $\tau \longrightarrow 0$
to $r^2/4$ on  $\overline{\Cc^{\Sigma}_{a}}$ for all $a > r_0$, and satisfies
\begin{align}
\label{eq:fid0}
r^2-\frac{N_0}{r^{2}} \le 4 \tau f(r, \sigma, \tau) \le r^2 + \frac{N_0}{r^{2}}, \quad \tau \nabla f = \frac{r}{2}\pd{}{r},
\end{align}
on $\Cc^{\Sigma}_{r_0}\times (0, 1]$ for some constant $N_0 > 0$.
\item[(4)] Together,  $g= g(\tau)$ and $f = f(\tau)$ satisfy 
\begin{equation}
 \label{eq:grstau}
 \Rc(g) + \nabla\nabla f = \frac{g}{2\tau}, \quad R + |\nabla f|^2 = \frac{f}{\tau}
\end{equation}
on $\Cc_{r_0}^{\Sigma}\times (0, 1]$.
\end{enumerate}
Here $r$ denotes the radial distance $r(x) = d(\mathcal{O}, x)$ on $\Cc^{\Sigma}$.
\end{proposition}

As in \cite{KotschwarWangIsometries}, we will say that a shrinker $(\Cc^{\Sigma}_{r_0}, \bar{g}, \bar{f})$ satisfying properties (1)-(4) of Proposition \ref{prop:brf} is \emph{dynamically asymptotic} to $(\Cc^{\Sigma}_{r_0}, \hat{g})$.
Note that if $(\Cc^{\Sigma}_{r_0}, \bar{g}, \bar{f})$ is dynamically asymptotic to $(\Cc_{r_0}^{\Sigma}, \hat{g})$, and therefore gauged so that $\bar{\nabla} \bar{f} = \frac{r}{2}\pd{}{r}$ on $\Cc_{r_0}^{\Sigma}$, the family of injective diffeomorphisms $\Phi = \Phi_{\tau}$ take the simple form
\[
       \Phi(r, \sigma, \tau) = (r/\sqrt{\tau}, \sigma)
\]
on $\Cc_{r_0}^{\Sigma}\times (0, 1]$.

\subsection{Proof of Theorem \ref{thm:soliton}}
Suppose that $(M, g, f)$ and $(\Sigma, g_{\Sigma})$
satisfy the assumptions of Theorem \ref{thm:soliton}.
For the proof, it suffices to consider the case $k=1$, that is, the case in which $(\Sigma, g_{\Sigma})$ is a single Einstein factor: the argument we give makes no special use of the form or dimension of the base manifold until after the singly-warped product structure has been shown to be preserved. Thus it can be iterated as in the proof of Corollary \ref{cor:mwp} from Theorem \ref{thm:warped}. By Proposition \ref{prop:brf}, we may further assume that $M = \Cc^{\Sigma}_{r_0}$ for some $r_0 \geq 1$ and that
$(\Cc^{\Sigma}_{r_0}, g, f)$ is dynamically asymptotic to $(\Cc^{\Sigma}_{r_0}, \hat{g})$.

Let $g = g(\tau)$ denote the associated solution to \eqref{eq:brf} described by Proposition \ref{prop:brf}
on $\Cc_{r_0}^{\Sigma}\times [0, 1]$.
Note that $g(0) = dr^2 +r^2 g_{\Sigma}$ is (in particular) a warped-product whose fibers have mean curvature vector $N_0$ satisfying $|N_0| = m|\nabla\log r|$.
Let $\Hc_0$ and $\Vc_0$ denote the horizontal and vertical distributions, and let $\Hc = \Hc(\tau)$ and $\Vc = \Vc(\tau)$ denote the families of orthogonal extensions
of $\Hc_0$ and $\Vc_0$ defined by \eqref{eq:hvev}.
Finally, define
\[
\ve{X} = (M, P, U), \quad \ve{Y} = (G, A, T^0,\nabla A, \nabla T^0),
\]
in terms of the tensors $M$, $P$, $U$, and $A$, $T^0$, and $G$ determined by $g$, $\Hc$, and $\Vc$ as above.
Note that $\ve{X}(0) = 0$ and $\ve{Y}(0) =0$ in view of Proposition \ref{prop:wpmpu}.

\begin{proposition}
On $\Cc_{r_0}^{\Sigma}\times [0, 1]$, we have the bounds
\[
        |\Rm| + |\nabla\Rm| + |\nabla\Rm|  \leq \frac{C_0}{r^2}, \quad |A| + |T^0| + |G| + |N| +  |\nabla A| + |\nabla T^0| \leq \frac{C_0}{r},
\]
for some constant $C_0$. Consequently, $\ve{X}$ and $\ve{Y}$ are uniformly bounded on $\Cc^{\Sigma}_{r_0}\times[0, 1]$ and satisfy the system
\begin{align}\label{eq:xysysbounds}
\begin{split}
    \left|D_{\tau} \ve{X} + \Delta\ve{X}\right| &\leq \frac{C}{r}\left(|\ve{X}| + |\nabla \ve{X}|+ |\ve{Y}|\right)\\
    \left|D_{\tau}\ve{Y}\right| &\leq C\left(|\ve{X}| + |\nabla \ve{X}|\right) + \frac{C}{r}|\ve{Y}|
\end{split}
\end{align}
for some constant $C$.
\end{proposition}
\begin{proof}
  According to Proposition \ref{prop:brf}, we have $|\nabla^{(l)}\Rm| \leq C/r^{l+2}$ on $\Cc_{r_0}^{\Sigma}\times [0, 1]$,
  and hence bounds of the form $|M| + |P| + |U| \leq C/r^2$. Since $|N| \leq C/r$ initially, it follows from \eqref{eq:nev} that
  $N\leq C/r$ for all $\tau\in [0, 1[$. Similarly, since $A$ and $T^0$ vanish initially, it follows from \eqref{eq:aev} and \eqref{eq:t0ev} that $|A|$, $|T^0|$
  (and therefore) $|B|$ satisfy bounds of the form $C/r$. Continuing in this way, we obtain bounds for $\nabla M$, $\nabla P$, and $\nabla U$,
  and thus for $G$, $\nabla A$, and $\nabla T^0$. Then \eqref{eq:xysysbounds} follows from Theorem \ref{thm:xysys} using the bounds
  just obtained to estimate the coefficients $\Theta_1$ and $\Theta_2$.
\end{proof}

To prove the vanishing of $\ve{X}$ and $\ve{Y}$, we will apply the following result, which is a special case of Theorem 4.1 in \cite{KotschwarTerminalCone}, and a slight generalization of the backward uniqueness principle underlying
the main theorem in \cite{KotschwarWangConical}.

\begin{theorem}\label{thm:solitonbu}
Suppose $(\Cc_{r_0}^{\Sigma}, g_1, f_1)$ is a shrinking Ricci soliton which is dynamically asymptotic to $(\Cc^{\Sigma}, g_c)$. Let $g(\tau)$ be the associated solution to \eqref{eq:brf} on $C_{r_0}^{\Sigma}\times [0, 1]$ with $g(0) = g_c$ and $g(1) = g_1$.
Let $\Xb = \Xb(\tau)$ and $\Yb=\Yb(\tau)$ be smooth families of sections of the bundles $\Xc$ and $\Yc$ over $\Cc^{\Sigma}_{r_0}$
\begin{equation}\label{eq:xybounds}
  \sup_{\Cc_{r_0}^{\Sigma}\times[0, 1]}\left\{|\Xb| + |\nabla \Xb| + |\Yb|\right\} \leq L,
\end{equation}
and
\begin{align}\label{eq:xysys2}
\begin{split}
    \left|D_{\tau}\ve{X} + \Delta\ve{X}\right| &\leq \varepsilon\left(|\ve{X}| + |\nabla \ve{X}|+ |\ve{Y}|\right)\\
    \left|D_{\tau}\ve{Y}\right| &\leq L\left(|\ve{X}| + |\nabla \ve{X}| + |\ve{Y}|\right),
\end{split}
\end{align}
for some constant $L > 0$ and some function $\varepsilon = \varepsilon(r) > 0$ with $\varepsilon(r)\longrightarrow 0$ as $r\longrightarrow\infty$. Then, if $\ve{X}(x, 0) \equiv 0$ and $\ve{Y}(x, 0) \equiv 0$ on $\Cc_{r_0}^{\Sigma}$, there are $r_1 = r_1(n, K, L) \geq r_0$ and $\tau_0 = \tau_0(n) \in  (0, 1]$ such that
$\Xb(x, \tau) \equiv 0$ and $\Yb(x, \tau) \equiv 0$ on $\Cc_{r_1}^{\Sigma}\times[0, \tau_0]$.
\end{theorem}

In the terminology of \cite{KotschwarTerminalCone}, the solution $g(\tau)$ to \eqref{eq:brf} associated to our shrinking soliton structure $(\Cc_{r_0}^{\Sigma}, g_1, f_1)$
\emph{emanates smoothly} from the cone $g_c$ at $\tau =0$; in fact, the decay rates \eqref{eq:curvdecay} of the derivatives of curvature are more than what is required. The above theorem can also be obtained from \cite{KotschwarWangConical} with only a few small amendments to the arguments there. As noted in \cite{KotschwarKaehler}, the sets of Carleman estimates proven in Propositions 4.7, 5.7, and 5.9 of \cite{KotschwarWangConical} are valid on any asymptotically
conical shrinking self-similar solution background metric $g(\tau)$, and while the sections $\ve{X}$ and $\ve{Y}$ defined above are different than those in \cite{KotschwarWangConical},
the argument there does not make use of any properties of $\ve{X}$ and $\ve{Y}$ other than those in the hypotheses in Theorem \ref{thm:solitonbu}. The chief difference is that the term $(C/r)|\nabla \ve{X}|$ on the right side of \eqref{eq:xysysbounds} does not appear in the corresponding
system for $\ve{X}$ and $\ve{Y}$ in \cite{KotschwarWangConical}. But this term can be handled with only a few minor modifications.

\begin{proof}[Proof of Theorem \ref{thm:soliton}]
It follows from \eqref{eq:xysysbounds} and  Theorem \ref{thm:solitonbu} that $\ve{X}(\tau) \equiv 0$ and $\ve{Y}(\tau)\equiv 0$
on $\Cc_{r_1}^{\Sigma}\times [0, \tau_0]$ for some $0 < \tau_0 \leq 1$ and $r_1\geq r_0$.
As in the proof of Theorem \ref{thm:warped}, from these facts we may conclude that $\Hc(\tau) \equiv \Hc_0$ and $\Vc(\tau) \equiv \Vc_0$
are independent of $\tau$,
and that there are families $\ch{g}(\tau)$ and $h(\tau)$ of metrics and functions on $(r_1, \infty)$ such that $g(\tau) = (\pi^*\ch{g}) + h^2\hat{\pi}^*g_{\Sigma}$ on
$\Cc^{\Sigma}_{r_1}\times [0, \tau_0]$. For each $\tau$, we may reparametrize by $s = s_{\tau}(\tau)$ so that
$g(\tau) = ds^2 + h^2(s, \tau)\hat{\pi}^*g_{\Sigma}$. Since $g_1 = \tau_0^{-1}(\Phi^{-1}_{\tau_0})^*g(\tau_0)$ on $\Cc_{r_1/\sqrt{\tau_0}}^{\Sigma}$, Theorem \ref{thm:soliton} follows.
\end{proof}

\end{document}